\documentclass[12pt,reqno]{amsart}
\usepackage[margin=1in]{geometry}
\usepackage{hyperref}

\def\XXint#1#2#3{{\setbox0=\hbox{$#1{#2#3}{\int}$ }
\vcenter{\hbox{$#2#3$ }}\kern-.6\wd0}}

\newcommand{\rmi}{\mathrm{i}}
\newcommand{\ord}{\mathrm{ord}}

\newcommand{\wt}{\widetilde}

\newcommand{\dd}{{\mathrm d}}

\newcommand{\Ss}{{\mathbb S}}
\newcommand{\bcs}{\begin{cases}}
\newcommand{\ecs}{\end{cases}}

\usepackage{mathrsfs, amsmath,amssymb,amsfonts, amsthm, dsfont, tikz-cd}
\newtheorem{theorem}{Theorem}[section]

\newtheorem{proposition}[theorem]{Proposition}
\newtheorem{corollary}[theorem]{Corollary}
\newtheorem{lemma}[theorem]{Lemma}
\newtheorem{remark}[theorem]{Remark}
\newtheorem{definition}[theorem]{Definition}

\numberwithin{equation}{section}

\newcommand{\Q}{{\mathbb Q}}
\newcommand{\Z}{{\mathbb Z}}

\newcommand{\hcal}{{\mathcal H}}
\newcommand{\ical}{{\mathcal I}}

\newcommand\WF{\operatorname{WF}'}

\numberwithin{equation}{section}

\newtheorem{theo}{{\sc Theorem}}

\newtheorem{prob}[theo]{{\sc Problem}}

\newcommand{\R}{{\mathbb R}}

\usepackage{epsfig}
\usepackage{amscd}
\usepackage{tikz-cd}
\usepackage{amsmath, amssymb, amsthm}
\usepackage{graphics}
\usepackage{color}
\usepackage{dsfont}
\newtheorem*{main-theorem}{Main Theorem}
\newtheorem*{old-thm}{Theorem}
\theoremstyle{definition}
\numberwithin{equation}{section}

\def\11{\mathds{1}}

\def\WF{\mathrm{WF}\,}

\def\supp{\mathrm{supp}\,}

\def\phi{\varphi}
\def\half{{\frac{1}{2}}}

\def\be{\begin{eqnarray*}}
\def\ee{\end{eqnarray*}}
\def\ben{\begin{eqnarray}}
\def\een{\end{eqnarray}}

\def\L2R{L_{\text{Rest}}^2}

\newcommand{\N}{{\mathbb N}}
\newcommand{\ccal}{\mathcal{C}}
\newcommand{\dcal}{\mathcal{D}}

\newcommand{\gcal}{\mathcal{G}}

\newcommand{\pcal}{\mathcal{P}}

\newcommand{\qcal}{\mathcal{Q}}

\newcommand{\scal}{\mathcal{S}}
\newcommand{\tcal}{\mathcal{T}}

\newcommand{\codim}{\operatorname{codim}}
\newcommand{\Char}{\operatorname{Char}}

\begin{document}
\title[]{Geodesic bi-angles and Fourier coefficients of  restrictions of eigenfunctions}
\author{Emmett L. Wyman}
\address{Department of Mathematics, Northwestern University, Chicago IL}
\email{emmett.wyman@northwestern.edu}

\author{Yakun Xi}
\address{School of Mathematical Sciences, Zhejiang University, Hangzhou 310027, PR China}
\email{yakunxi@zju.edu.cn}

\author{Steve Zelditch}
\address{Department of Mathematics, Northwestern University, Chicago IL}
\email{s-zelditch@northwestern.edu}

\maketitle
	\setcounter{tocdepth}{1}

\begin{abstract} This article concerns joint  asymptotics of Fourier coefficients of restrictions   of Laplace eigenfunctions $\phi_j$ of  a 
compact Riemannian manifold to a submanifold
$H \subset M$. We fix a number $c \in (0,1)$ and study the asymptotics of the thin sums, 
$$ N^{c} _{\epsilon, H  }(\lambda): = 
\sum_{j,  \lambda_j \leq \lambda}  \sum_{k: |\mu_k - c \lambda_j | < \epsilon} \left| \int_{H} \phi_j \overline{\psi_k}dV_H \right|^2 $$
where $\{\lambda_j\}$ are the eigenvalues of $\sqrt{-\Delta}_M,$ and $\{(\mu_k, \psi_k)\}$ are the eigenvalues, resp.  eigenfunctions, of $\sqrt{-\Delta}_H$. The inner sums
represent the `jumps' of $ N^{c} _{\epsilon, H  }(\lambda)$ and reflect the geometry of geodesic c-bi-angles with one leg on $H$ and a second leg on $M$ with
the same endpoints and compatible initial tangent vectors $\xi \in S^c_H M, \pi_H \xi \in B^* H$, where $\pi_H \xi$ is the orthogonal projection of $\xi$
to $H$. A c-bi-angle occurs when  $\frac{|\pi_H \xi|}{|\xi|}  = c$. Smoothed sums in $\mu_k$ are also studied, and give sharp estimates on the jumps. The  jumps themselves may jump as $\epsilon$ varies, at certain values of $\epsilon$  related to periodicities
in  the c-bi-angle geometry. Subspheres of spheres and certain subtori of tori illustrate these jumps. The results refine those of the previous article \cite{WXZ20} where the
inner sums run over   $k: | \frac{\mu_k}{\lambda_j} - c|  \leq \epsilon$ and where geodesic bi-angles do not play a role.

\end{abstract}

\tableofcontents

\section{Introduction}

Let  $(M, g)$ be  a compact, connected Riemannian manifold  of dimension $n$ without boundary, let $\Delta_M = \Delta_g$ denote its  Laplacian, and  let $\{\phi_j\}_{j=1}^{\infty} $ be an orthonormal basis of its eigenfunctions,
$$(\Delta_M + \lambda_j^2) \phi_j  = 0, \qquad \int_M \phi_j \overline{\phi_k} dV_M = \delta_{jk}, $$ 
where $dV_M$ is the volume form of $g$, and where the eigenvalues are enumerated in increasing order (with multiplicity), 
$\lambda_0 =0 < \lambda_1 \leq \lambda_2 \cdots \uparrow \infty.$
Let  $H \subset M$ be an embedded  submanifold  of dimension $d \leq n -1$  with induced metric $g |_{H}$, let $\Delta_H$ denote the Laplacian of $(H, g |_H)$, and let
 $\{\psi_k\}_{k=1}^{\infty} $ be an orthonormal basis of its eigenfunctions on $H$,
$$(\Delta_H + \mu_k^2) \psi_k = 0, \qquad \int_H \psi_j \overline{\psi_k} dV_H = \delta_{jk}, $$ 
where $dV_H$ is the volume form of $g|_H$.  We denote the restriction operator to $H$  by $$\gamma_H: C(M) \to C(H), \;\; \gamma_H f = f |_H. $$  This article is concerned
with the Fourier coefficients  \begin{equation} \label{FCDEF} a_{\mu_k} (\lambda_j) :=    \langle \gamma_H \phi_j, \psi_k \rangle_H =  \int_H (\gamma_H \phi_j) \psi_k dV_H \end{equation}  in the $L^2(H)$-orthonormal expansion,
\begin{equation} \label{ONR} \gamma_H \phi_j (y) = \sum_{k=1}^{\infty} \langle \gamma_H \phi_j, \psi_k \rangle_H\; \psi_k(y) \end{equation}
of the restriction of $\phi_j$ to $H$.
 Our goal is to understand the joint asymptotic behavior
of the Fourier coefficients \eqref{FCDEF} as   the eigenvalue pair $(\lambda_j, \mu_k)$ tends to infinity along a `ray' or `ladder' in the joint
spectrum of $(\sqrt{-\Delta_M} \otimes I, I \otimes \sqrt{- \Delta_H})$ on $M \times H$.
The motivating problem is to determine estimates or asymptotics for all of the Fourier coefficients \eqref{FCDEF} of an individual eigenfunction
$\phi_j$ as $\lambda_j \to \infty$.

This problem originates  in the spectral theory of automorphic forms, where $(M,g)$ is a
 compact (or finite area, cusped) hyperbolic surface, and where $H$ is a closed geodesic, or a distance circle, or a closed horocycle (in the cusped case). 
In this case,  \eqref{FCDEF} takes the form,
\begin{equation} \label{FOURIER} \gamma_H\phi_j (s)= \sum_{n \in \Z} a_n(\lambda_j) e^{\frac{2 \pi i n s}{L}}, \end{equation} 
where $L $ is the length of $H$ and $s$ is the arc-length parameter.  See Section \ref{HYPERBOLIC} for a
 more precise statement when $H$ is a closed horocycle. Classical estimates and asymptotics of Fourier coefficients of modular forms are studied by Rankin \cite{Ra77}, Selberg \cite{Sel65}, Bruggeman \cite{Br81}, Kuznecov \cite{K80}, and many others. Systematic expositions in this context may be found in \cite{Br81, G83,  I02}. It  is `folklore',
and proved also in this article,  that $ |a_n(\lambda_k)| = O_{k, \epsilon} (n^{-N})$ for all $N \geq 0$ for $n \geq \lambda_k + \epsilon$, and for any $\epsilon > 0$;  see also \cite{Wo04,  Xi}  for some estimates of this type.   If one thinks of $\lambda_k$ as the `energy' and $n$ as the
angular momentum, then rapid decay  occurs in the ``forbidden region''  where the angular momentum exceeds the energy see Section \ref{AFSECT}). Hence, we restrict to the
case where $|n| \leq \lambda_j$.  The motivating
question is, how does the dynamics of the geodesic flow of $(M,g)$ and of $(H, g |_H)$ determine the equidistribution properties
of the restricted Fourier coefficients?

Studying the joint asymptotics of the Fourier coefficients \eqref{FCDEF}  for individual eigenfunctions of  general compact Riemannian manifolds $(M,g)$ and submanifolds
$H \subset M$ is a very difficult problem and all but intractable except in special cases such as subspheres of  standard spheres. Asymptotic knowledge of the coefficients \eqref{FCDEF} would afortiori  imply  asymptotic knowledge of the $L^2$ norms of restrictions of eigenfunctions, 
\begin{equation} \label{PLANCHEREL} \int_H |\gamma_H \phi_j|^2 d S_H = \sum_{k} |a_{\mu_k}(\lambda_j)|^2, \end{equation}
which are themselves only known in special cases. General estimates of \eqref{PLANCHEREL} 
are proved in \cite{BGT}, but are only sharp in special cases.  Knowledge of the distribution of Fourier coefficients and their relations to the
geodesic flow are yet (much)  more complicated. In the case where $(M,g)$ has ergodic geodesic flow and $H \subset M$ is a hypersurface, the $L^2$ norms of restrictions is known as the QER (quantum ergodic restriction)
problem, i.e. to determine when $\{\gamma_H \phi_j\}_{j=1}^{\infty}$ (or at least a subsequence of density one) is quantum ergodic along $H$. For instance, if $H$ is a hypersurface of a compact negatively curved manifold, it is shown in  \cite{TZ13} that  there exists  a full density subsequence of the restrictions $\{\gamma_H \phi_{j_k} \}$ such that \begin{equation} \label{QER} \int_H |\gamma_H \phi_{j_k}|^2 dS_H = \sum_{m: \mu_m \leq \lambda_{j_k} }  |a_{\mu_{m} } (\lambda_{j_k})|^2 \simeq 1, \;\; (k \to \infty). \end{equation}
The Fourier coefficient distribution problem is to determine how the `mass' of the Fourier coefficients are distributed. 
 It has been conjectured  that when the geodesic flow is `chaotic' (e.g. for a negatively curved surface), the Fourier coefficients should 
 exhibit  `equipartition of energy', i.e. all be roughly of the same size.  By comparison, for standard spherical harmonics $Y_N^m $ on $\Ss^2$, the restrictions have only one non-zero Fourier coefficient. The general Fourier coefficient problem is to find conditions on   $(M, g, H)$ 
under which the Fourier coefficients $a_{\mu}(\lambda_j)$
concentrate at  one particular frequency (tangential eigenvalue), or under which they exhibit equipartition of energy. 

In this article,  we take  the first step in this program by studying the averages in $\lambda_j$ of sums of squares of a localized set of Fourier coefficients.  Equivalently, we study the  joint asymptotics of \eqref{FCDEF} of the thinnest possible sums over the joint spectrum $\{(\lambda_j, \mu_k)\}_{j, k =1}^{\infty}$.  Here, `thin' refers to both the width of the window in the  spectral averages in $\lambda_j$, and also the width of the 
window of Fourier modes $\mu_k$.  In  comparison, conic or wedge windows are studied in \cite{WXZ20}.



\subsection*{Kuznecov-Weyl sums}

The Kuznecov sum formula 
 in the sense of \cite{Zel92}  refers to the asymptotics
of the Weyl-type sums,
\begin{equation} \label{f} N_H(\lambda; f): = \sum_{j: \lambda_j \leq \lambda} \left| \int_H f \phi_j dV_H \right|^2,  \end{equation}
where $f \in C^{\infty}(H)$.  Here, and henceforth, we drop the restriction operator $\gamma_H$ from $\gamma_H \phi_j$ if it is 
obvious from the context that $\phi_j$ is being restricted. 
The asymptotics are controlled by the structure of the `common orthogonals', i.e.  the set of geodesic arcs which hit $H$ orthogonally at both endpoints.
The original Kuznecov formulae pertained to special curves on arithmetic hyperbolic surfaces \cite{K80}, but the Weyl asymptotics  have been generalized to 
submanifolds of general Riemannian manifolds in \cite{Zel92}. Recent improvements of spectral projection estimates under various conditions  appear in   \cite{ ChS15, CGT17, SXZh17, Wy17a, CG19,  WX}. Fine asymptotics in the arithmetic setting are given in \cite{M16}. We refer to these articles for many further
references.


In this article, we are interested in the joint asymptotics of $ |\langle \gamma_H \phi_j, \psi_k \rangle_{L^2(H)}|^2$ as 
 the pair $(\lambda_k, \mu_j)$  tends to infinity along a ray in $\R^2$. Except for exceptional $(M, g, H)$ it is not possible to obtain asymptotics along a single ray. Instead, we study the joint asymptotics in a thin strip around the ray. Our
 main result deploys a smooth version of such a strip, sometimes called
`fuzzy ladder asymptotics' in the sense of \cite{GU89}.  Let $\psi \in \scal(\R)$ (Schwartz space)
 with $\hat{\psi} \in C_0^{\infty}(\R)$ a positive test function. We then define the fuzzy-ladder sums,
  \begin{equation} \label{cpsi}
 N^{c} _{\psi, H  }(\lambda): = 
\sum_{j:  \lambda_j \leq \lambda}  \sum_{k=0}^{\infty}  \psi( \mu_k - c \lambda_j)  \left| \int_{H} \phi_j \overline{\psi_k}dV_H \right|^2.
 \end{equation}
 The asymptotics of \eqref{cpsi} depend on the geometry of what we term $(c, s, t)$ bi-angles in Section \ref{GEOMSECT}. Roughly speaking, such
 a bi-angle consists of two geodesic arcs, one a geodesic of  $H$ of length $s$, the other a geodesic of $M$ of length $c s + t$, with common endpoints
 and making a common angle ${\rm arccos} \; c$ with $H$. When $t= 0$,  the set of such bi-angles is defined by,
 \begin{equation}  \label{EQs=0} \gcal^0_c = \{(q, \xi) \in S^*_HM: \; |\pi_H \xi| \; = c |\xi|, \;G_H^{ -  s} \circ \pi_H \circ G_M^{c s} (q, \xi)  = (q, \pi_H (\xi))\}, 
  \end{equation} where $\pi_H: T^*_q M \to T_q^*H$ is the orthogonal projection at $q \in H$.

 It turns out that the ``spectral edge" case $c=1$ requires different techniques from the case $c < 1$.  It is the 	edge or  `interface' between the allowed
interval $[0, 1] $ and the complementary  `forbidden' intervals of  $c$-values (see Section \ref{AFSECT} for discussion).  
As often happens at interfaces, there is a variety
of possible behaviors when $c=1$.  As discussed below, the case $c=1$ corresponds
 to tangential intersections of geodesics with $H$, which depend on whether 
  $H$ is totally geodesic or whether it has non-degenerate second fundamental form. For this reason, we separate out the cases $0< c <1 $ and $c=1$ and only study $0< c < 1$ in this article. The case of $c=1$ and $H$
 totally geodesic is studied in  \cite{WXZ+}. The case of $c=1$ and $H$ with a non-degenerate second fundamental form requires quite different
 techniques from the totally geodesic case, and  is currently under investigation. 

 The first result is a ladder refinement of the main result of \cite{WXZ20}.  For technical convenience, we assume the test function $\psi$ is even and  non-negative.
\begin{theorem} \label{main 2} Let $\dim M = n$ and let $\dim H =d$. Let  $0 < c <1$  and assume that $\gcal_c^0$  \eqref{EQs=0} is clean
in the sense of Definition 
\ref{CLEAN}.  Then, if $\psi\ge0$ is even, $\hat{\psi} \in C^{\infty}_0(\R)$,  and $\supp \hat{\psi}$ is contained in a sufficiently small interval around $s = 0$,
there exist universal constants $C_{n,d}$ such that 
$$N_{\psi, H}^{c} (\lambda) =
 C_{n,d} \;\; a_c^0(H, \psi) \lambda^{n-1 } +   R_{\psi, H}^{c} (\lambda), \;\; \text{with}\;\; R_{\psi, H}^{c} (\lambda) =  O(\lambda^{n-2}), 
 $$
where the leading coefficient is given by,
\begin{equation} \label{acHpsi0} a_c^0(H, \psi): = 
 \hat{\psi}(0) \; c^{d-1} (1 - c^2)^{\frac{n-d-2}{2}}  \hcal^{d}(H), 
\end{equation}
\end{theorem}

The definition of ``$\supp\; \hat{\psi}$ is sufficiently small" is given in Definition \ref{SUFFSMALLDEF}  below. For the moment, we say that 
when $0 < c < 1$,  it means that the only    component  of \eqref{EQs=0} with    $s \in \supp\hat{\psi}$ is the component  with $s=0$.
  It is straightforward to remove the condition that $\rm{supp} \;\hat{\psi}$ is small, but then there are further contributions  from other components of
$\gcal_c^0$, which require further definitions and notations.  We state the generalization in Theorem \ref{main 5} below. 
 The coefficient $a_c^0(H, \psi)$  is discussed in Section \ref{SHARPWKINTRO}.  

Under additional dynamical assumptions on the geodesic flow $G^t_M$ of  $(M,g)$ and the geodesic flow $G_H^t$ of $(H, g|_{H})$, one can improve the remainder estimate of Theorem \ref{main 2} (see Theorem \ref{main 2b} and Theorem
\ref{main 3}.) The  improvement requires the inclusion of all components of $\gcal_c^0$ and therefore requires the generalization
Theorem \ref{main 5} of Theorem \ref{main 2}. Since it requires further notation and definitions, we postpone it until later in the introduction.

\subsection{Jumps in the Kuznecov-Weyl sums} To obtain results for individual eigenfunctions (more precisely, for individual eigenvalues)
by the techniques of this article, we study the jumps,
\begin{equation} \label{JDEFpsi}  J^{c} _{ \psi, H  }(\lambda_j): =  \sum_{\ell: \lambda_{\ell} = \lambda_j} 
\sum_{ k =0}^{\infty}  \psi( \mu_k - c \lambda_j)  \left| \int_{H} \phi_{\ell} \overline{\psi_k}dV_H \right|^2,\end{equation} 
in the Kuznecov-Weyl sums \eqref{cpsi}  at the eigenvalues $\lambda_j$. The sum over $\ell$ is a sum over an orthonormal basis for the eigenspace $\hcal(\lambda_j)$ of $-\Delta_M$ of eigenvalue
$\lambda_j^2$. Since the leading term is continuous, the jumps are jumps of the remainder, 
\begin{equation} \label{JUMPREM} J^{c}_{\psi, H}(\lambda_j) =  R_{\psi, H}^{c} (\lambda_j + 0 ) - R_{\psi, H}^{c} (\lambda_j - 0). \end{equation} 
By \eqref{JUMPREM} and Theorem \ref{main 2}, we get  
\begin{corollary} \label{main 2cor}  With the same assumptions and notations as in Theorem \ref{main 2},  for
any positive even test function $\psi$ with $\hat{\psi} \in C_0^{\infty}(\R)$,  there exists a constant $C_{c, \psi}  > 0$ such that
$$
	J_{\psi, H}^{c} (\lambda) \leq  \begin{array}{ll}   C_{c, \psi}\; \lambda^{n-2}, & 0 < c < 1. \end{array}
$$
\end{corollary} 
Corollary \ref{main 2cor} has a further implication on the `sharp jumps' where we replace $\psi$ by an indicator function, ${\bf 1}_{[-\epsilon, \epsilon]}$. 
\begin{equation} \label{JDEF}  J^{c} _{ \epsilon, H  }(\lambda_j): = \sum_{\ell: \lambda_{\ell} = \lambda_j} 
\sum_{  \substack{k: | \mu_k - c \lambda_j|  \leq \epsilon}}   \left| \int_{H} \phi_{\ell} \overline{\psi_k}dV_H \right|^2. \end{equation}
 By choosing $\psi$ carefully in Corollary \ref{main 2cor} we prove,

\begin{corollary} \label{JUMPCOR} With the above notation and assumptions, 
$$
	J^{c}_{\epsilon, H}(\lambda_j) =  O_{c,\epsilon} (\lambda^{n-2}), \qquad 0 < c < 1.
$$
\end{corollary}
This estimate is shown to be sharp for closed geodesics in general spheres $\Ss^n$ in Section \ref{SnSHARPSECT}. The sharpness
of the jump estimates is intimately related to the periodicity properties of both $G^t_M$ and $G^t_H$, and is discussed in detail in 
Section \ref{CLUSTERSECT}.
The dependence on $\epsilon$ in Corollary \ref{JUMPCOR} is complicated in general, as illustrated in 
Section \ref{S2SPARSE} and Section \ref{SnSHARPSECT}.  This is because of potential concentration of the Fourier coefficients
at joint eigenvalues $(\lambda_j, \mu_k)$ at  the `edges' or endpoints of the strip of width  $\epsilon$ around a ray. See Lemma \ref{JUMPSPHERE} for the explicit formula in the case of subspheres
of spheres. 
When $\dim M =2,\, \dim H=1$, the eigenvalues of  $|\frac{d }{ds}|$  of $H $ are of course multiples 
of $n \in \N$ involving the length of $H$ and are therefore separated by gaps of size depending on the length.  The  jumps \eqref{JDEF} can themselves jump as $\epsilon $ varies, as
illustrated by restrictions from spheres to subspheres in the same section. See also Section \ref{REMAINDERSECT}
for further discussion. These edge effects arise again in Theorem \ref{main 3}.
For $\epsilon$ sufficiently small, only one eigenvalue will occur in the sum over $k: |\mu_k - c \lambda_j| \leq \epsilon$ and therefore Corollary \ref{GENCOR} gives a bound on the size of an individual Fourier coefficient
of an individual eigenfunction under the constraints on $(c, H)$ in the corollary.  We refer to Section \ref{2Sphere} for examples of curves on $\Ss^2$. We also refer to the second author's work \cite{Xi} and the first two authors' work \cite{WX} for prior results on  bounds on Fourier coefficients.

\begin{remark} \label{GENCOR} 

For a generic metric $g$ on any manifold $M$, the spectrum of the Laplacian is simple,
i.e. all eigenvalues have multiplicity one \cite{U}.  In this case,  the $\lambda_j$ sum of \eqref{JDEF} reduces to a single term and one gets, 
 hence,
\begin{equation} \label{INDIVIDEST}  \left| \int_{H} \phi_j \overline{\psi_k}dV_H \right|^2 \leq J^{c}_{\epsilon, H}(\lambda_j):= 
\sum_{  \substack{k: | \mu_k - c \lambda_j|  \leq \epsilon}}   \left| \int_{H} \phi_{j} \overline{\psi_k}dV_H \right|^2, \;\; \forall k: |\mu_k - c \lambda_j| \leq \epsilon. \end{equation}

 \end{remark}

\subsection{Allowed and forbidden joint eigenvalue regions}\label{AFSECT} 
We briefly indicate why we restrict to $c \in [0,1]$. The reason is that the asymptotics are trivial outside
this range.

	 \begin{lemma}\label{FORBIDDEN} If $\mu_k/\lambda_j \geq 1+ \epsilon$, then for any $N \geq 1$,
	 $  \left| \int_{H} \phi_j \overline{\psi_k}dV_H \right|^2
\leq C_N(\epsilon) \;\lambda_j^{- N}$. Hence,  \eqref{JDEF} for  $c > 1$  is rapidly decaying in $\lambda$.
	 \end{lemma}
	 
	 The proof is contained in the proof of Theorem \ref{main 3}. The main point of the proof is that,  if one uses a cutoff $\psi( \mu_k - c \lambda_j)$ with $c > 1$, then there
	 $N_{H, \psi}^c(\lambda)$ is rapidly decaying. This is also proved in detail in \cite{Xi} for the case $\dim H = 1$.
	 
	 We will not comment on this issue further in this article.   As Lemma \ref{FORBIDDEN} indicates,   $c =1$ is an `interface' in the spectral asymptotics, separating an allowed and
	 a forbidden region. It would be interesting to study the interface asymptotics  around $c=1$, i.e. the transition
	 from the asymptotics of Theorem \ref{main 2} to trivial asymptotics for $c > 1$;  
	  we plan  to study the scaling   asymptotics around $c=1$ in future work.

 \subsection{Geometric objects: $(c, s,t)$ bi-angles}\label{GEOMSECT}
 
We now give the geometric background to Theorem \ref{main 2} and Theorem \ref{main 3}, in particular
defining the relevant notion of `cleanliness' and explaining the coefficients.  We will need some further notation (see \cite{TZ13} for further details). Let $G_M^t$, resp. $G^t_H$ denote the geodesic flow of $(M,g)$ on $\dot{T}^*M$, resp. $\dot{T}^* H$  (since the flow is homogeneous, and $S^*M$ is invariant, 
 it is sufficient to consider the flow on $S^*M$).

 \begin{remark} \label{NOTATION} Notational conventions: For any manifold $X$ we denote by 
 $\dot{T}^* X = T^* X \backslash 0$ the punctured cotangent bundle. All of the canonical relations in this paper
 are homogeneous and are subsets of $\dot{T}^* X$.   \end{remark}

 We denote by $S^*_H M$ the covectors $(y, \xi) \in S^*M$ with footpoint $y \in H$, and by
 $\pi_H(y, \xi) = (y, \eta) \in B^*H$ the orthogonal projection of $\xi$  to the unit coball  $B_y^*H$ of $H$ at $y$. 
 We also denote by 
 $T^c_H M$  the cone  of covectors in $\dot T^*_HM$ making an angle $\theta$ to $T^* H$ with 
 $\cos \theta = c$. It is  the cone through
 \begin{equation} \label{SHc} 
 	S^c_H M = \{(y, \xi) \in S^*_H M: |\pi_H \xi| = c \}.
\end{equation}

 
 \begin{definition} \label{BIANGLEDEF}
  By a $(c, s, t)$-bi-angle through $(q, \xi)   \in S^c_H M$,  we mean a periodic,  once-broken orbit of the composite  geodesic flow solving the equations,  
 \begin{equation}  \label{EQ} \begin{array}{l} G_H^{ - s} \circ \pi_H \circ G_M^{cs + t} (q, \xi) = \pi_H(q, \xi),\;\; (q, \xi)  \in S^c_H M 
\end{array} \end{equation}
 Thus,  a bi-angle is a pair of geodesic arcs with common endpoints, for which the $M$-length is $cs +t$ and whose $H$ length is $s$.
 \end{definition}

Note that the geometry changes considerably when  $c=1$  and $H$ is totally geodesic. In this case, when $t=0$,  a $(1, s, 0)$  bi-angle is a geodesic arc of $H$, traced forward for time $s$ and backwards for time $s$.   The consequences are explored in \cite{WXZ+}.

    The projection of a $(c, s, t)$  bi-angle   to $M$ consists of a geodesic arc of $M$ of length $ct +s$, with both endpoints on $H$,   making the angle
 $\theta$ with $\cos \theta = c$ at both endpoints $q$ resp. $ q'$, and an  $H$-geodesic arc of oriented  $H$-length $ -s$  with initial velocity
 $(q, \pi_H \xi)$ and with terminal velocity $(q', \pi_H G^t_M(q, \xi))$.  Equivalently,  let $\gamma^M_{x, \xi}$ denote the geodesic of $M$ with initial data $(x, \xi) \in S^*_x M$,
 and let $\gamma^H_{y,\eta}$ denotes the  geodesic of $H$ with initial data $(y, \eta) \in S^*H$. Then the
 defining property of a $(c, S, T)$- bi-angle is that it consists of a geodesic arc $\gamma^M_{x, \xi}$ of $M$ of  length $T$, and a geodesic arc $\gamma^H_{x, \eta}$  of $H$ of   length $S$,
 such that: 
 
 $$\left\{ \begin{array}{l} 

\gamma^M_{x, \xi}(0)= x = \gamma^H_{x, \eta}(0) \in H, \pi_H \dot{\gamma}^M_{x, \xi}(0) =  \dot{\gamma}^H_{x, \eta}(0); \\ \\ 

| \dot{\gamma}^M_{x, \xi}(0)| = 1,   | \dot{\gamma}^H_{x, \eta}(0)| = c; \\ \\ 

\gamma^M_{x, \xi}(T)=  \gamma^H_{x, \eta}(S) , \pi_H \dot{\gamma}^M_{x, \xi}(T) =  \dot{\gamma}^H_{y, \eta}(S).
 \end{array}. \right. $$
Some examples on the sphere $S^2$ are given in Section \ref{SHARPSECT}.

We denote the set of all solutions with a fixed $c$ by
\begin{equation} \label{gcalc}
\gcal_c = \{(s, t, q, \xi) \in \R \times \R \times S^c_H M: \eqref{EQ}  \text{ is satisfied}\}.
\end{equation}
For $t$ fixed, we  also define 
\begin{equation} \label{gcalct}
	\gcal_c^t  = \{(s,  y, \xi) \in \R  \times S_H^cM: (s, t, y, \xi) \in \gcal_c\}.
\end{equation} 
We  decompose  $\gcal_c$  into the subsets,
$$\gcal_c = \gcal_c^0 \bigcup \gcal_0^{\not=0}, \;\;\; \gcal_c^0 = \gcal_c^{(0,0)} \bigcup \gcal_c^{(0, \not=0)}, $$ where (cf. \eqref{EQs=0})
  \begin{equation} \label{gcalc0} \left\{ \begin{array}{l}  \gcal_c^0 = \{(s, 0, y, \xi) \in  \R \times S_H^cM: \eqref{EQ}  \text{ is satisfied}\}, \\ \\
   \gcal_c^{0,0} \simeq S^c_H M= \{(0,0, y, \xi) \in   S_H^cM: \eqref{EQ}  \text{ is satisfied}\} \\ \\ 
   \gcal_c^{\not= 0}  = \{(s,t, y, \xi) \in \R \times   \R \times S_H^cM: t \not= 0,  \eqref{EQ}  \text{ is satisfied}\}. \end{array} \right.  \end{equation} 
   
The principal term of \eqref{cpsi} only involves the set $ \gcal_c^0$.
   When $c < 1$, or when $c=1$ and $H$ has non-degenerate
  second fundamental form,  the equation cannot hold for small enough $s$ because the $M$-geodesic
  between the endpoints must be shorter than the $H$ geodesic.  More formally, we state,
  
\begin{lemma} \label{ISOLATED}  When $0 < c < 1$, or if $c =1$ and $H$ has non-degenerate second fundamental form,  $\gcal_c^{0,0}$   is a connected component of $\gcal_c^0$. If  $H$ is totally geodesic and $0 < c < 1$,
$\gcal_c^{0,0}$ is  also a connected component, i.e. there do not exist any $(c, s, 0)$ bi-angles when $s \not= 0$ is sufficiently small (depending on $c$). 
When $c=1$ and $H$ is totally geodesic, $\gcal_1^{0,0}$ is  not a connected component of $\gcal_1^0$. 
\end{lemma}

 In general, the  order of magnitude of
$N^{c}_{\psi, H}(\lambda)$  depends  on the dimension of \eqref{gcalc0} and on its symplectic  volume measure. 
\begin{definition} For $0< c < 1$  the symplectic volume ${\rm Vol}(\gcal_c^0)$  of \eqref{gcalc0} is the Euclidean measure of  $\gcal_c^0.$ The pushforward
volume form acquires an extra factor $(1 - c^2)^{-\half}$ (see \cite{WXZ20} for more details).  \end{definition}

 The additional components only contribute to the main term of the asymptotics, when they have the same dimension as $\gcal_c^{0,0}$.

 \begin{definition} \label{DOMDEF} For $0 < c < 1$ define the `principal component' of $\gcal_c^0$ to be $\gcal_c^{0, 0} $.
 We say that the principal component  is {\it dominant} if $\dim \gcal_c^{0,0} $ is strictly 
  greater than any other component of $\gcal_c^0$. 
\end{definition}

Examples on spheres illustrating
the various scenarios are given in Section \ref{SECTEX}. If $c \in {\mathbb Q},\, c < 1$, then $G_{\Ss^d} ^{-s}$ and $G_{\Ss^n}^{c s}$ have an arithmetic progression of common periods, hence there
are infinitely many components of the fixed point set \eqref{EQs=0}  of the same dimension as $\gcal_c^{0,0}$, all contributing to the main term
of the asymptotics.
For instance, if $H \subset \Ss^2$ is a latitude circle, then for $0 < c \leq 1$, $\gcal_c^{0,0}$ is not dominant, due to periodicity of
the geodesic flow and of rotations (see Section \ref{SHARPSECT}). On a negatively curved surface, $\gcal_c^{0,0}$ is dominant.
Note that $\gcal_c^0 \cap \{0\} \times \R_s \times S^*H$ is a single component when $c=1$ and $H$ is totally geodesic so that the
notion of principal and dominant is vacuous then.
When  $c=1$  \eqref{gcalc} consists
  of bi-angles formed by an $M$ geodesic hitting $H$ tangentially at both endpoints, closed up by an $H$-geodesic arc through the same endpoint
  velocities.   
      When $c = 0$, $\xi = \nu_y$ is the unit (co-)normal and \eqref{gcalc} consists of $M$-geodesic arcs hitting $H$ orthogonally 
  at both endpoints. These are the arcs relevant to the original Kuznecov formula of \cite{Zel92}.

   To eliminate the contribution of non-principal components when $0 < c < 1$, we assumed in Theorem \ref{main 2} that   $\supp \hat{\psi}$ is `sufficiently small'. We now define
   the term more precisely.   
   
   \begin{definition}\label{SUFFSMALLDEF}  When $0 < c < 1$, we  say that $\supp \hat{\psi}$ is sufficiently small if $\gcal_c^{0,0}$ is the only component of $\gcal_c^0$ with values $s \in \supp \hat{\psi}$.  \end{definition}

\subsection{Cleanliness and Jacobi fields}\label{CLEANJF}

  As in most asymptotics problems, we will be using the method of stationary phase, mainly implicitly when composing canonical relations. This requires
  the standard notion of cleanliness from \cite{DG75}.
  
\begin{definition} \label{CLEAN}
We say that \eqref{gcalc}, resp. \eqref{gcalc0}, is {\it clean} if $\gcal_c$, resp. $\gcal_c^0$, is a submanifold of $\R \times \R\times S^c_H M$, resp. $\R \times S^c_H M$, and if its tangent space at each point is the subspace fixed by $D_{\zeta}  G_H^{ - s} \circ \pi_H \circ G_M^{cs + t} $ (resp. the same with $s = 0$), where $\zeta$ denotes the $S^c_H M$ variables.
\end{definition}

  Some examples are given in Section \ref{HSECT}.   It is often not very difficult to determine whether  $\gcal_c, \gcal_c^0$ are manifolds. The tangent cleanliness condition is often difficult to check. 
It  may be stated in terms of {\it  bi-angle Jacobi fields}, as follows.
 A Jacobi field  along a geodesic arc $\gamma$ is the  vector field $Y(t)$ along $\gamma(t)$ arising from a 1-parameter variation $\alpha(t; s)$ 
 of geodesics, with $\alpha(t; 0) = \gamma(t)$.  An $(S, T)$  bi-angle consists of an $H$-geodesic arc $\gamma^H(t) $ of length $S$ and
  an $M$ geodesic arc $\gamma^M(t)$ of length $T$  which  have the same initial and terminal point, and such that the  projection of the initial and terminal velocities to  $\gamma^M(T)$  equal those of  $\gamma^H(S) $.  A bi-angle Jacobi field $Y(t)$  arises as  the variation vector field of a  one-parameter variation $(\gamma_{\epsilon}^M(t), \gamma_{\epsilon}^H(t) $)  of
  $(S(\epsilon), T(\epsilon)) $- bi-angles.  It consists of a  Jacobi field $J_M(t)$ along $\gamma^M(t) $ and
  a Jacobi field $J_H(t) $ along $\gamma^H(t) $ which are compatible at the endpoints.  Namely,  if we differentiate in $\epsilon$ the equations  $$
  \begin{array}{l} \gamma^H_{\epsilon} (0) = \gamma^M_{\epsilon} (0),
  \gamma^H_{\epsilon} (S(\epsilon))  = \gamma^M_{\epsilon}  (T(\epsilon) ), \\ \\  \pi_* \dot{\gamma}^M_{\epsilon} (0) = \dot{\gamma}^M_{\epsilon}(0),\;\;
  \pi_* \dot{\gamma}^M(S(\epsilon)) = \dot{\gamma}^M(T(\epsilon)), \ |\dot{\gamma}^H_{\epsilon} (0)| = c \end{array}$$
  at $\epsilon =0$ we get 
  $$J_H(0) = J_M(0), \;\; J_H (t_0) \dot{S} = J_M(T) \dot{T},$$
  together with two equations for $\frac{D}{Dt} J_H, \frac{D}{Dt} J_M.$  
  
  \begin{definition} A bi-angle Jacobi field is a pair $(J_H(t), J_M(t))$ of Jacobi fields, as above,  along the arcs of the bi-angle which satisfy
  the above compatibility conditions at the endpoints. \end{definition}  
  
  In defining variations of bi-angles, we may allow the angle parameter $c$ to vary with $\epsilon$ in a variation, or to hold it fixed. In Definition \ref{CLEAN}, the parameter $c$ is held
  fixed, and therefore all variations must have the same value of $c$.
  As with standard Jacobi fields, the bi-angle Jacobi field  arises  by varying the initial tangent vector $(x, \xi) \in S^*_x M, x \in H$
  along $S^*_H M$. A `horizontal' bi-angle Jacobi field is one that arises by varying $x \in H$, and a vertical bi-angle Jacobi field arises
  by fixing $x$ but varying the initial direction $\xi$. 
  
  \begin{remark} The vectors in the subspace fixed by $D G_H^{-s} \circ \pi_H \circ G_M^{cs +t} (\zeta) = \pi_H(\zeta) $  always correspond
  to bi-angle Jacobi fields along the associated bi-angle.  Lack of cleanliness occurs if there exists a bi-angle Jacobi field which
  does not arise as the variation vector field of  a variation of $(c, S, T)$-bi-angles of $(M, g, H)$. \end{remark}

Cleanliness is often difficult to verify. We  refer to Section \ref{HSECT} for further discussion and un-clean examples.
    
\subsection{Outline of the proof of Theorem \ref{main 2} and further results} To prove Theorem \ref{main 2} we first prove the related result where the indicator
functions are replaced by smooth test functions.  The exposition will set the stage for the further  results (Theorem \ref{main 5} and
Theorem \ref{main 3}).

The sharp sums over
 $\lambda_k$ in  \eqref{c}  and half-sharp sums \eqref{cpsi}  may be replaced by `twice- smoothed' sums with shorter windows if we introduce a second eigenvalue cutoff $\rho \in \scal(\R)$ with
 $\hat{\rho} \in C_0^{\infty}(\R)$ and define,  \begin{equation} \label{cpsirho}
 N^{c} _{\psi, \rho, H  }(\lambda): = 
\sum_{j, k}  \rho(\lambda - \lambda_k) \psi( \mu_j - c \lambda_k)  \left| \int_{H} \phi_j \overline{\psi_k}dV_H \right|^2.
 \end{equation}
 Under standard `clean intersection hypotheses,' the sums \eqref{cpsirho} admit complete asymptotic expansions. They are the raw data 
 of the problem, to which the methods of Fourier integral operators apply. To obtain the sharper results,  we apply some Tauberian
 theorems, first (when $\psi \geq 0$)  to replace $\rho$ by the corresponding indicator function to obtain two-term asymptotics for \eqref{cpsi}. Then, in Theorem \ref{main 2},  
 we apply a second Tauberian theorem to replace $\psi$ (when possible)  by the corresponding indicator function. In the following, we assume that
 $\supp\hat\psi$ is `sufficiently small'  as defined in Definition \ref{SUFFSMALLDEF}.

To study the asymptotics of \eqref{cpsirho},  we express the left side as the oscillatory integral, 
\begin{equation}\label{SpsiDEF}
\begin{array}{lll}
N^{c} _{\psi, \rho, H  }(\lambda) & =  & \int_{\R} \hat{\rho}(t) e^{-it \lambda} S^c(t, \psi)dt, \;\; \text{ where} \\&&\\ 
S^c(t, \psi): & = &  \sum_{j,k} e^{it \lambda_j } \psi (\mu_k-c \lambda_j) \left| \int_H \phi_{j,k}(x,x) dV_H(x) \right|^2, \end{array}
\end{equation}
 show that $S^c(t, \psi)$ is a Lagrangian distribution on $\R$,  and determine its   singularities (Proposition \ref{MAINFIOPROP}).
 
 \begin{definition} \label{SOJOURNDEF} Let $\Sigma^c(\psi ) = \operatorname{sing\, supp} S^c(t, \psi)$ be the set of singular points of $S^c(t,\psi)$. $\Sigma^c(\psi)$ is called the
 set of   `sojourn times', and consists 
 of $t$ for which there exist solutions of \eqref{EQ} with $s \in \supp \hat{\psi}.$ 

\end{definition} 
We then have,
\begin{theorem} \label{main 4} Let $\dim M = n,\, \dim H = d$. Let $\psi, \rho \in \scal(\R)$ with $\hat{\psi}, \hat{\rho} \in C_0^{\infty}(\R)$.  Assume that
$\supp \hat{\rho} \cap \Sigma^c(\psi) = \{0\}$ and  $\hat{\rho}(0) =1$. Also assume that
$\supp \hat{\psi}$ is sufficiently small in the sense of Definition \ref{SUFFSMALLDEF} . Let  $c \in (0,1)$  and assume that $\gcal_c^0$ is clean in the sense of Definition 
\ref{CLEAN}.    Then, there exists a complete asymptotic expansion of $N^{c} _{\psi, \rho, H  }(\lambda)$ with principal terms,
$$ N^{c} _{\psi, \rho, H  }(\lambda) = C_{n, d} \; a_c^0(H, \psi) \lambda^{n-2 }  + O( \lambda^{n-3}), \qquad  (0 < c < 1). 
$$
where the leading coefficient is given by \eqref{acHpsi0}. 
\end{theorem}

While this theorem follows as a corollary to \cite[Theorem 1.4]{WXZ20}, our proof here also admits a more general statement (see Theorem \ref{main 5}). Theorem \ref{main 4} is proved in Section \ref{ASYMPTOTICSECT}.
Note that the order of asymptotics  is $1$ less than in Theorem \ref{main 2} and Theorem  \ref{main 3}. This is due to the fact that
the $\lambda$ intervals are `thin' in Theorem \ref{main 4} and `wide' in the previous theorems.

To obtain the `sharp' Weyl asymptotics of Theorem \ref{main 2}  in the $\lambda$-variable, we need to replace $\rho$ by an indicator function. 
This is done in Section \ref{SS Tauberian} by a standard Tauberian theorem.

\subsection{Asymptotics when $\supp  \hat{\psi}$ is arbitrarily large}\label{supplarge}

Theorem \ref{main 2} and Theorem \ref{main 4} both assume that $\supp \hat{\psi}$ is sufficiently small in the sense
of Definition \ref{SUFFSMALLDEF}. They readily admit generalizations to any $ \hat{\psi} \in C_0^{\infty}$ where we take into
account the additional components of $\gcal_c^0$ coming from large $s$.  Only the components of $\gcal_c^0$ of the same  maximal dimension as the
principal component contribute to the principal term of the asymptotics.

In the notation of   \cite[Theorem 4.5]{DG75}, we are assuming that the set of $(c, s, 0)$-bi-angles with $t = 0$  is a union of clean components $Z_j(0)$ of dimension $d_j$. In our situation $Z_j(0)$ is a component of  $ \gcal^0_c$. 
Then, for $t$ sufficiently close to $0$,  each $Z_j(0)$  gives rise to a  Lagrangian distribution $\beta_j(t)$ on $\R$ with singularities only at $t=0$, such that, \begin{equation}\label{betaeq}  \begin{array}{l} S^c(t, \psi) = 
\sum_j \beta_j(t), \;\; \text{ where} \;\; \beta_j(t) = \int_{\R} \alpha_j(s) e^{- i s t} ds, \\ \\ \text{ with}\;\; \alpha_j(s) \sim (\frac{s}{2 \pi i})^{ -1 + \half (n -d)+\frac{d_j}{2}}\;\; i^{- \sigma_j} \sum_{k=0}^{\infty} \alpha_{j,k} s^{-k}, \end{array} \end{equation}
where $d_j $ is the dimension of the component $Z_j(0) \subset \gcal_c^0$.

\begin{definition}\label{MAXDEF}  We say that a connected component $Z_j(0)$ of $\gcal_c^0$ is {\it maximal} if its dimension is the same as $\gcal_c^{0,0}$. We denote
the set of maximal components by $\{Z_j^m(0) = \gcal_{c, j}^{0, m}\}_{j=1}^{\infty}$. When $0 < c < 1$, $s$ is constant on each maximal component and we
denote these values of $s$ by $s_j^m$. 
\end{definition}

Existence of a maximal component other than $\gcal_c^0$ implies that there exists periodicity in the geodesic flows
$G^t_M$ and $G_H^t$ (see  Section \ref{CLUSTERSECT}),
 and that $c$  is a special value where there exists periodicity in the flows $G^{-cs}_H$.
For instance, when $H = \Ss^d \subset \Ss^n$, and  when $c = \frac{p}{q} \in {\mathbb Q}$, there exist maximal components when $s = 2 k \pi q$ with 
$k = 1, 2, \dots$, since then the left side of \eqref{EQs=0} is the identity, 
$G^{-s}_{\Ss^d} \circ S^{cs }_{\Ss^n} = G^{-2 k \pi q }_{\Ss^d} \circ S^{2 k \pi p  }_{\Ss^n} = {\rm Id} \times {\rm Id}.$  This is the geometric reason behind
the discontinuous behavior of the jumps in Section \ref{S2SPARSE}. 
Recall $\Sigma^c(\psi)$ from Definition \ref{SOJOURNDEF} for the following theorem.

\begin{theorem} \label{main 5} Let $\dim M = n, \dim H = d$. Let $\psi, \rho \in \scal(\R)$ with $\hat{\psi}, \hat{\rho} \in C_0^{\infty}(\R)$. Assume that
 $\hat{\rho}(0) =1$ and that
$\supp \hat{\rho} \cap \Sigma^c(\psi) = \{0\}$.  Let  $0 < c <1$  and assume that $\gcal_c^0$ is clean in the sense of Definition 
\ref{CLEAN}.   Then, there exists a complete asymptotic expansion of $N^{c} _{\psi, \rho, H  }(\lambda)$ with principal terms,
$$ N^{c} _{\psi, \rho, H  }(\lambda) =
\begin{array}{ll} C_{n, d} \; a_c^0(H, \psi) \lambda^{n-2 }  + O_\psi( \lambda^{n-5/2}), & 0 < c < 1
\end{array} 
$$
where the leading coefficient is given by the following sum over the maximal components $Z_j^m(0)$ of $\gcal_c^0$, 
$$a_c^0(H, \psi): =
\left(\sum_j \hat{\psi}(s_j^m) \right) \; c^{d-1} (1 - c^2)^{\frac{n-d-2}{2}}  \hcal^{d}(H).
$$
Moreover, if we replace $\rho$ by ${\bf 1}_{[0, \lambda]}$,  then 
$$ N^{c} _{\psi, H  }(\lambda)  = C_{n, d} \; a_c^0(H, \psi) \lambda^{n-1 }  + O_{\psi}( \lambda^{n-3/2}). $$


\end{theorem}

The exponent $n-5/2$ in the second term of the asymptotics can arise  from   components $Z_j(0)$ of dimension $d_j$ possibly one less than maximal. This drops the exponent of the top term of $\beta_j$ of \eqref{betaeq} by $1/2$. If no such components occur, the remainder is of order $\lambda^{n-3}$. 
See Section \ref{3PROOFS} for the proofs of Theorem \ref{main 4} and Theorem \ref{main 5}.

\begin{remark} The formula for $a_c^0(H, \psi)$ reflects the fact that, 
under  the cleanliness hypothesis of Definition \ref{CLEAN}, 
each maximal component $Z_j^m(0) \simeq \gcal_c^{0,0}$ must be essentially the same as the principal component modulo the change in the $s$ parameter, hence
the corresponding  leading coefficient $\alpha_{j, 0}$ in \eqref{betaeq} is the same as for the principal component. 
The only essential change to the principal coefficient in  Theorem \ref{main 4} is that it is now necessary to compute the sum of the canonical 
symplectic  volumes of all 
maximal components with $s_j^m \in \supp\; \hat{\psi}$. When $c < 1$, the set of such $s_j^m$ is discrete and thus we sum $\hat{\psi}$ over this
finite set of parameters.  See Section \ref{CLUSTERSECT} for a discussion of the periodicity properties necessary to have maximal
components. 

\end{remark}
\subsection{Two term asymptotics with small oh remainder} \label{REMAINDERSECT}

Theorem \ref{main 2}, Theorem \ref{main 4} and Theorem \ref{main 5} assume that $\supp \hat{\rho} \cap \Sigma^c(\psi) = \{0\}$ and  $\hat{\rho}(0) =1$. By  allowing general $\rho \in C^{\infty}(\R)$ with $\hat{\rho} \in C_0^{\infty}$, and by studying long-time asymptotics of $G^t_M$ and $G^t_H$, it is possible under favorable circumstances to prove two-term asymptotics of the type initiated by Y. Safarov
(see \cite{SV} for background). From Theorem \ref{main 5}, we note that the order of the second term depends on the geometry
of $\gcal_c$.  When the second term of Theorem \ref{main 5} has order $\lambda^{n-3}$,  the  two term asymptotics have the following form: for
any $\epsilon >0$, we have as $\lambda \to \infty$, 
\begin{equation}\begin{array}{lll} C_{n, d} \; a_c^0(H, \psi) \lambda^{n -  1 }  + \qcal^c_{H, \psi} (\lambda - \epsilon ) \lambda^{n-2} -
o(\lambda^{n-2}) &  \leq  N^c_{\psi, H}(\lambda)   &   \\ && \\
 \leq   C_{n, d} \; a_c^0(H, \psi) \lambda^{n -  1 }    +
\qcal^c_{H, \psi} (\lambda + \epsilon) \lambda^{n-2} + o(\lambda^{n-2}),&&
\end{array} \end{equation}
where $\qcal^c_{H, \psi}(\lambda)$ is an oscillatory function determined by the singularities for all $t$ (see Section \ref{SINGtnot0}).

 The rather unusual notion of second term asymptotics is necessary and was first proved by Safarov  for the pointwise Weyl
 law or the fully integrated (over $M$) Weyl law; see \cite{SV} and also \cite[(29.2.16)- (29.2.17)]{HoIV}.  
Depending on the periodicity properties of $G_M^t$ and $ G_H^s$, $\qcal^c_{H, \psi}(\lambda)$ can be continuous or have jumps. 
 Existence of jumps in  $\qcal^c_{H, \psi}(\lambda)$  is not a sufficient condition for jumps of size $\lambda^{n-2}$ in 
$N^c_{\psi, H}(\lambda) $ but it is a necessary one and can be used to analyze situations where maximal jumps occur.
  In view of the many possible types of phenomena discussed in Section \ref{SHARPSECT}, it is a lengthy
additional problem to prove such two term asymptotics and in particular to calculate $\qcal^c_{H, \psi}(\lambda)$ explicitly, and we defer their
study to a future article. In this section, we state some results on situations where $\qcal^c_{H, \psi}= 0$.

The next result improves the  remainder estimate in Theorem \ref{main 2} under the
dynamical hypothesis that the $t=0$ singularity is dominant in the sense of Definition \ref{DOMDEF}.

\begin{theorem} \label{main 2b} With the  notation and assumptions as in Theorem \ref{main 2}, we  assume $\hat{\psi}$ has small support. We  further assume that the singularity at $t=0$ is dominant, i.e. there
do not exist maximal components   $Z_j(T)$ for $T \not=0.$   
 Then,

$$ \begin{array}{l} N^{c} _{\psi, H  }(\lambda)  = C_{n, d} \; a_c^0(H, \psi) \lambda^{n - 1 }  + R_{\psi, H}^c(\lambda),  \;\; \text{ where}\\ \\
R_{\psi, H}^c(\lambda)    \; = \; o_{\psi} (\lambda^{n-2}),\;\;\;  J_{\psi, H}^c(\lambda)    \; = \; o_{\psi}(\lambda^{n-2}). \end{array}$$

\end{theorem}
To prove Theorem \ref{main 2b},  we first need to prove that the   coefficient  in the expansion of $N_{\psi, \rho, H}(\lambda)$ of the term of order $\lambda^{n-2}$ is zero. This is 
 shown in  Section \ref{SUBPRINCIPALSECT}.  From Theorem \ref{main 2b} we obtain an improved estimate on the jumps in the
 case where both $G^t_M$ and $G^s_H$ are `aperiodic', i.e. where the Liouville measure in $S^*M$, resp. $S^*H$ of the periodic orbits
 of $G^t_M$, resp. $G^s_H$ are zero. 
The following estimate follows directly from Theorem \ref{main 2b} and the same technique of bounding ${\bf 1}_{[-\epsilon, \epsilon]}$ by a test function $\psi $ with $\hat{\psi} \in C_0^{\infty}$  used in Corollary \ref{JUMPCOR}.

\begin{corollary}\label{JUMPCONJ}  If both $G^t_M$ and $G^s_H$ are aperiodic  (see Section \ref{CLUSTERSECT}), then for any $\epsilon > 0$,
$$J_{\epsilon, H}^c (\lambda_j) = o_{\epsilon} (\lambda_j^{n-2}). $$ 
\end{corollary}

The proof of Corollary \ref{JUMPCONJ} from Theorem \ref{main 2b} follows a well-known path and is sketched in  Section \ref{main 2bSECT}.

\subsection{Sharp-sharp asymptotics}\label{SHARPSHARPTHM}

Instead of the fuzzy ladder sums \eqref{cpsi}, it may seem preferable to study the sharp Weyl-Kuznecov sums,

 \begin{equation} \label{c}
 N^{c} _{\epsilon, H}(\lambda): = 
\sum_{j, k:  \lambda_j  \leq \lambda, \; | \mu_k - c \lambda_j|  \leq \epsilon}   \left| \int_{H} \phi_j \overline{\psi_k}dV_H \right|^2,
 \end{equation}
 in which we constrain the tangential modes $\mu_k$ to lie in an $\epsilon$- window around $c \lambda_j$ for $0 < c < 1$.  



\begin{theorem} \label{main 3}
$\dim M = n$, $\dim H = d$. Let  $0 < c < 1$,  and assume that $\gcal_c^0$ is clean in the sense of Definition 
\ref{CLEAN}.  Assume that  no component of $\gcal_c^0$ has maximal dimension except for the principal component (cf. Definition \ref{DOMDEF}).
Then, in the notation of Theorem \ref{main 2} - Theorem \ref{main 5}, the  Weyl - Kuznecov sums \eqref{c} satisfy:

$$
	N^c_{\epsilon, H}(\lambda) = C_{n,d} a_c^0 (H, \epsilon)  \lambda^{n - 1} + o (\lambda^{n - 1}),
$$
where
\begin{equation} \label{acHpsiep} 
a_c^0(H, \epsilon) := \epsilon c^{d-1} (1 - c^2)^{\frac{n-d-2}{2}} \hcal^{d}(H),
\end{equation}
where $C_{n,d}$ is some constant depending only on the dimensions $n$ and $d$.
\end{theorem}

Note that the remainder estimate is much weaker than for Theorem \ref{main 2b}, which has a hypothesis on the $t \not= 0$ singularities
(i.e. on the ``$\rho$ aspect''), and has  a smooth test function $\psi$ with
$\hat{\psi} \in C_0^{\infty}(\R)$ instead of ${\bf 1}_{[-\epsilon, \epsilon]}$. 
Indeed, 
 by Corollary \ref{JUMPCOR} we get the jump estimate, 
\begin{equation} \label{epjump} 
N^c_{\epsilon, H} (\lambda_j) - N^c_{\epsilon, H} (\lambda_j - 0) = J_{\epsilon, H}^c (\lambda_j) = O_{\epsilon} (\lambda_j^{n-2}). \end{equation}
The worse remainder in Theorem \ref{main 3}  is due to the explicit dependence on $\epsilon$ of the main term $a_c^0 (H, \epsilon)$, which results in   two parameters having possible jumps: $\lambda_j$ and $\epsilon$.  
The $\epsilon$-dependence of the coefficient is discussed in Section \ref{SHARPWKINTRO}. 
  The large jump size  reflects the new `layer' of eigenvalues one gets when $\epsilon$ hits
  its critical values, in cases where each `layer' has the same order of magnitude as the principal layer. The layers correspond to 
  connected components of $\gcal_c^0$. In the case $c < 1$ the components indexed by certain  values of $s$.
  As illustrated in the case of spheres in Section \ref{S2SPARSE},   there can exist large contributions from the edges (endpoints) of the interval $\mu_k - c \lambda_j \in [-\epsilon, \epsilon]$ for
special values of $\epsilon$, i.e. $J_{\epsilon, H}^c (\lambda_j)$ itself can jump as $\epsilon$ increases by the amount $\lambda_j^{n-2}$. But by Corollary \ref{JUMPCONJ} again, this cannot happen if $G^t_M$ and $G^s_H$ are aperiodic. Hence the principal component condition is necessary. 
 
 To obtain the `doubly sharp' Weyl asymptotics of Theorem \ref{main 3}, we need to replace $\psi$
by an indicator function. This is done in 
Section \ref{BUSECT} by a Tauberian argument of semi-classical type  adapted from  Petkov-Robert \cite{PR85}.

  \subsection{The principal coefficient $ a_c^0(H, \psi)$ in Theorems \ref{main 2} - \ref{main 3}}\label{SHARPWKINTRO}
  There are two aspects (roughly speaking) to the principal coefficient  \eqref{acHpsi0}: the volume aspect and the $\psi$-aspect. 
  
In the case where $\hat \psi$ and $\hat \rho$ have small support, the symbol calculations in Theorem \ref{main 3} are contained in \cite{WXZ20}. In this article, we use the symbol calculus of Fourier integral operators under pullback and pushforward as in \cite{GS77, GS13} to calculate symbols.

  We note that for $0 < c < 1$, $ c^{d-1} (1 - c^2)^{\frac{n-d-1}{2}} \hcal^d(H)$ is the $(n-2)$-dimensional  volume  of the set in  $S^*_{H, c}  \subset  S^*_H M$ projecting
to $\gcal_c^{0,0}$.  Indeed, for each  $x \in H$,  $\pi_H \xi \in \gcal_c^{0,0}$ if and only if the components
 $\xi = \xi^{\perp} + \xi^{T}$ of $\xi \in S^*_x M$ of the orthogonal decomposition  $T_x M = T_x H \oplus (N_x H)$ have norms $\sqrt{1-c^2}, $ resp. $c$, so that $S^*_{H, c}  \simeq \Ss_c^{d-1} \times \Ss^{n-d-1}_{\sqrt{1 - c^2}}$ where $S_r^k$ is the Euclidean $k$-sphere of radius $r$; the
$n-2$ dimensional surface measure of $S^*_{H, c} $ is $c^{d-1} (1 - c^2)^{\frac{n-d-1}{2}}$ times the
$n-2$-volume of $\Ss^{d-1} \times \Ss^{n-d-1}.$ The extra factor of  $ (1 - c^2)^{-\frac{1}{2}} $ is due to the density $\frac{1}{|\det D \pi_H|}$ of the Leray
measure relative to the Euclidean volume measure. The projection $\pi_H: S_q^* M \to S^*_q H$  has a fold singularity where $c=1$ along 
$S_q^*H$ with a one-dimensional kernel, so the Leray density vanishes to order $1$ when $c=1$, hence the difficulties at the $c = 1$ interface.

Now let us check the $\epsilon$-dependence of $a^0_{H, \epsilon}$ in Theorem \ref{main 3}. This is different from the analysis above,  because
the spectral weight  ${\bf 1}_{[-\epsilon, \epsilon]}$ has a non-compactly supported Fourier transform and indeed it   has very long tails. The hypothesis of Theorem \ref{main 3} rules out the case of $\Ss^d \subset \Ss^n$ when $c < 1$  because
in the rational case the latter has many maximal dimensional components due to fixed point sets of the periodic Hamilton flow. In the proof,  we replace ${\bf 1}_{[-\epsilon, \epsilon]}$ with $\psi_{T, \epsilon}: = \theta_T * {\bf 1}_{[-\epsilon, \epsilon]}$ and 
  then the principal coefficient is  $a_c(H, \psi_{T, \epsilon})$,  with  remainder of order $\frac{1}{T}$. To obtain the remainder estimate,  we take the limit as $T \to \infty$. By Theorem \ref{main 5},
 we get $\hat{\psi}_{T, \epsilon}(0)$ for $c < 1$.  As $T \to \infty$, 
 $\psi_{T, \epsilon} \to {\bf 1}_{[-\epsilon, \epsilon]}$, so $\hat{\psi}_{T, \epsilon}(0) \to 2\epsilon $.

\subsection{Remarks on Tauberian theorems}
If we think of
$\lambda_j^{-1} =: \hbar_j$ as the Planck constant, then we are considering eigenvalues of the zeroth order operator $\hbar \sqrt{-\Delta}_H$ in the semi-classical
thin interval $|\hbar_j\mu_k - c | \leq \epsilon \hbar_j$ (whose width is one lower order than that of the operator). The Tauberian theorem   
of Section \ref{BUSECT} is indeed modeled on a semi-classical Tauberian theorem. However, in several essential respects, the sharp jumps and
the sharp sums in $\lambda_j $ do not behave in a  semi-classical way and the results are homogeneous rather than semi-classical. For instance,
the jumps rarely have asymptotic expansions.

\subsection{Related results and problems} \label{RELATED} 

We first compare the results to those of \cite{WXZ20}. 
In comparison to this article, `conic'  sums
$$  N^{m,c}_{ \epsilon, H  }(\lambda): = 
\sum_{j, k:  \lambda_k \leq \lambda, \; | \frac{\mu_j}{\lambda_k} - c|  \leq \epsilon}   \left| \int_{H} \phi_j \overline{\psi_k}dV_H \right|^2$$
are emphasized in \cite{WXZ20}.
 The sums $N^{m, c}_{\epsilon, H}$ are wide in $\lambda$ and  `conic'  in $\mu$, so they are wide in every sense. Hence, we expect to 
 have asymptotics of these sums with no dynamical hypotheses.  
 In \cite{Zel92}, the special (and singular) case $c= 0$ was studied, and indeed, the Fourier coefficient
 along $H$ as fixed at $0$. The $\lambda$-sums were wide and sharp. In more recent work, \cite{SXZh17, CGT17, CG19,  Wy17a}, short and sharp
 sums were considered under various geometric and dynamical hypotheses.
 
 Some of the symbol calculations needed for this article are contained in \cite{WXZ20}. We refer there
 for the calculations and do not duplicate them here. 
 
 \subsubsection{Two-term asymptotics}
 As discussed in Section \ref{REMAINDERSECT}, a significant refinement of the results of this article is to prove a two-term 
 asymptotic expansion for $N^c_{\psi, H}(\lambda)$ and to calculate $Q_{\psi, H}^c(\lambda)$. The existence of the two-term
 asymptotics is sketched in Section \ref{SINGtnot0} but the calculation of $Q_{\psi, H}^c(\lambda)$ is deferred to later work.
 
 \subsubsection{$c=1$ and H totally geodesic}\label{c=1+} 
As indicated above, the case  $c=1$ is the edge case, and  there are several different
types of `singular' or extremal behavior in this case. 
When $H$ is totally geodesic, the  asymptotics for $c=1$ are determined
 in the subsequent article \cite{WXZ+}. It turns out that the power of $\lambda$ in Theorem \ref{main 2}  is $\lambda^{\frac{n+d}{2}}$, 
 hence depends on $d$ as well as on $n$.
  
 \subsubsection{ Submanifolds with non-degenerate second fundamental form  when $c=1$}\label{CAUSTICSECT}
When
$c=1$ and $H$ has non-degenerate second fundamental form, one expects diffractive or Airy type effects when $H$ is a caustic hypersurface
for the geodesic flow. This type of eigenfunction concentration seems to require Airy integral operator techniques that are beyond
the scope of this article. Simple examples of eigenfunctions in this sense   on caustic  latitude circles  of $\Ss^2$ are presented in Section
\ref{2Sphere}. The example of closed horocycles of cusped hyperbolic surfaces is studied in \cite{Wo04}.  It would be interesting to see if there exist more exotic examples involving more general singularities than 
fold singularities. 

\subsubsection{Equipartition of energy among Fourier coefficients}

This problem is mentioned around \eqref{QER}. 
\begin{prob} How is the mass of the Fourier coefficients distributed among the $a_n(\lambda_j)$? \end{prob} 

In the case of `chaotic' geodesic flow, it is plausible that the distribution of mass among the Fourier
coefficients should be roughly constant, at least away from the endpoints of the allowed interval. It is a very difficult problem to determine if and when
the Fourier coefficients are equidistributed, or when they have singular concentration, but the problem guides many of the studies of averaged Fourier
coefficients. 

\subsubsection{Estimates for individual eigenfunctions}\label{INDIVIDSUBSECT}

It would be most desirable to obtain estimates on individual
Fourier coefficients of individual eigenfunctions, but again it is usually only feasible to  study averages over thin windows of  the Fourier coefficients when $\dim M >2$.

\begin{prob} What are sharp upper bounds on the individual terms \eqref{INDIVIDEST}. Which are the quadruples  $(M, g, H, \phi_j)$ for which the individual Fourier coefficients \eqref{INDIVIDEST} are maximal? Do they have the property that
$\gamma_H \phi_j$ is an eigenfunction of $H$? In that case, \eqref{INDIVIDEST} are the same as $L^2$ norms, for which sharp estimates
are given in \cite{BGT} (see Section \ref{BGTREV}).  When does \eqref{JDEF} have the same order of magnitude as  \eqref{INDIVIDEST}?
\end{prob} 

It may be expected that the maximal case occurs when $H$ has many almost orthogonal Gaussian beams. This occurs when $H$ is a totally geodesic
subsphere of a standard sphere.  Gaussian beams may be constructed around elliptic periodic geodesics, but usually as quasi-modes rather than modes.
The quasi-modes give rise to jumps in the second term but not jumps in the Weyl-Kuznecov sums.

\subsubsection{When is the restriction of an eigenfunction an eigenfunction?}\label{EIGRESEIG}

It is  plausible that restrictions of  individual eigenfunctions of $M$ with maximal Fourier coefficients  are eigenfunctions
of $H$.  Otherwise its Fourier coefficients
are spread out too much. \bigskip

\begin{prob} 
What are necessary and sufficient conditions that 
the restriction of a $\Delta_M$ eigenfunction
is a $\Delta_H$ eigenfunction? 
\end{prob}

It is likely that this problem has been studied before, but the authors were unable to find a reference. 
The known examples all seem to involve separation of variables.

Examples where $\gamma_H \phi_j$ is an eigenfunction of $H$ are standard exponentials $e^{i x \cdot \xi}$ on a flat torus, where $H$ is a totally geodesic sub-torus. Other 
examples are the standard spherical harmonics $Y_N^{\vec m}$ on the $n$-sphere $\Ss^n$, where $H$ is a `latitude sphere' i.e. an orbit
of a point under $SO(k)$ for some $k \leq n$.
In Section \ref{SHARPSECT} we consider various examples on the standard spheres $\Ss^n$.

\subsection{Acknowledgments} Xi was partially supported by National Key R\&D Program of China
No. 2022YFA1007200, National Natural Science Foundation of China No. 12171424. Wyman
was partially supported by National Science Foundation of USA No. DMS-2204397 and by
the AMS Simons travel grants. Zelditch was supported by National Science Foundation of
USA Nos. DMS-1810747 and DMS-1502632.

\section{Sharpness of the remainder estimates} \label{SHARPSECT}

In this section, we give  examples illustrating  the sharpness of the remainder estimate of Theorem \ref{main 3} and the behavior of the jumps
\eqref{JDEF} and the estimate of  Corollary \ref{JUMPCOR}. We also illustrate the cleanliness issues in Definition
\ref{CLEAN} with some examples that explain the reasons for assuming that $0 < c < 1$. The main example of jump behavior is that of totally geodesic or of  latitude spheres $\Ss^d \subset \Ss^n$ in standard spheres. We postpone
the discussion of these examples to \cite{WXZ+}.   

As mentioned above, the behavior of remainders and jumps depends on the periodicity properties of $G^t_M$ and $G^s_H$. We
begin by discussing the role of periodicities in spectral asymptotics.

   \subsection{Spectral clustering, jump behavior and periodicity of geodesic flows}\label{CLUSTERSECT}
   
   There is a well-known dichotomy among geodesic flows and Laplace spectra which plays an important, if implicit, role in the
   Kuznecov-Weyl asymptotics. Namely, if the geodesic flow $G^t_M$ of $(M,g)$ is periodic in the sense that $G^T = Id$ then the spectrum of $\sqrt{-\Delta_M}$ 
   clusters along an arithmetic progression $\{ \frac{2\pi}{T} k + \frac{\beta}{4}, k \in {\mathbb N}\}$ where $T$ is the minimal period and $\beta$ is the common Morse index of the closed geodesics. On the other hand, if the geodesic flow is ``aperiodic'' in the sense that the 
   set of closed geodesics has Liouville measure zero in $S^*M$, then the spectrum is uniformly distributed modulo one. We refer to
   \cite{DG75} for the original theorem of this kind and to \cite{Zel17} for further background. There also exist intermediate cases with
   a positive measure but not a full measure of closed geodesics. 
   
The principal term of Theorem \ref{main 2} (and subsequent theorems) does not depend on whether the eigenvalues cluster or are
uniformly distributed, but the remainder terms and jump formulae do. In the  examples of subspheres $H = \Ss^d \subset  M = \Ss^n$,  both $G^t_M $ and $G^t_H$ are periodic and both Laplacians have spectral clustering. Indeed, the eigenvalues of $\sqrt{-\Delta_{\Ss^n}}$
concentrate along the arithmetic progression $\{N + \frac{n-1}{2}\}$ 
and have multiplicities of order $N^{n-1}$. As discussed in detail in Section \ref{SHARPSECT}, this causes huge jumps in the $\lambda$ aspect at eigenvalues of $\sqrt{-\Delta_{\Ss^n}}$. Furthermore, 
the equation $\mu_k = c \lambda_j$ for fixed $(\lambda_j, c)$ can have many solutions when $\sqrt{-\Delta}_H$ has spectral clustering, i.e. when
$G^t_H$ is periodic. The number of solutions depends on the relation between the periods and therefore on $c$. 

To be more precise, the Kuznecov-Weyl  asymptotics are determined by the dimension of the `fixed point set'   \eqref{EQs=0}.  For general
$(n,d)$ this is not literally the fixed point set of a flow, since the relevant `joint flow'  at $t=0$ is the family of   maps (depending on $s$), 
$$G_H^{ -  s} \circ \pi_H \circ G_M^{c s} (q, \xi):    S^c_H M \to B^*H, \;\; B^*H = \{(q, \eta) \in T^*H: |\eta|_H \leq 1\}$$
between different spaces.
 In the special case where $\dim H = \dim M-1$ is an oriented  hypersurface,  this joint flow may be considered a double-valued
flow on $B^*H$. As in \cite{TZ13}, we can define lifts
$\xi_{\pm}: B^* M \to S^*_H M$ where $\xi_{\pm}(q, \eta) $ are the two unit covectors (on opposite sides of $T^*H$) that lift $(q, \eta)$ in the sense that $\pi_H \xi_{\pm}(q, \eta) = (q, \eta)$. Then \eqref{EQs=0} is the fixed point equation at $t=0$  of the double-valued flow, 
$$G_H^{ -  s} \circ \pi_H \circ G_M^{c s}\; \xi_{\pm} (q, \eta):   B^*H \to B^*H.$$
When studying the singularities
at $t \not=0$ one has the equation \eqref{EQ}, which in the hypersurface case is the equation,
\begin{equation} \label{LIFT} G_H^{ - s} \circ \pi_H \circ G_M^{cs + t} \xi_{\pm} (q, \eta ) = (q, \eta). \end{equation}
When $c < 1$ and $t=0$,  the fixed point set has maximal dimension for $s \not= 0$ (i.e. if the principal component is non-dominant in the 
sense of Definition \ref{DOMDEF}) only if the double-valued flow is the identity map at time $s$. The main example is when both
$G^{-s} _H$ and $G^t_M$ are both periodic and $c$ is such that they have a common period.
When $c=1$ and $H$ is totally geodesic, this equation holds trivially for all $s$ and all $(q, \eta)$. Periodicity of $G_H^t$ 
is a necessary and sufficient condition to obtain singularities at times $t \not= 0$ as strong as the one at $t=0$ in the case $c=1$ and
$H$ totally geodesic. {In the language of \cite{TZ13}, there is a first return map $\Phi^c: S^c_H M \to S^c_H M$ defined by following
geodesics with initial data in $S^c_H M$ until they return to $S^c_H M$. The first return time to $S^c_H M$ is denoted by  $T_H^c$
and $\Phi^c = G_M^{T^c_H} $. Since $S^c_H M$ has codimension $> 1$ in $S^*M$, the first return
may be infinite
on a large subset of $S^c_H M$. We have  not formulated the results in terms  of $T_H^c$ or $S^c_H M$, but they are 
implicitly  relevant in the main results. We refer to Section \ref{SURFOFREV} for the example of convex surfaces of revolution, where $T_H^c$ is
a constant when $H$ is an orbit of the rotation group. }

For submanifolds $H$ of codimension $> 1$, the double valued lift generalizes to a correspondence taking $\eta \in B^*_y H$
to a sphere $S^{n-d-1}$ of possible covectors $\eta + \sqrt{1 - |\eta|^2} \nu \in S^*_H M$ projecting to $\eta$, as $\eta$ varies over
$S N^* H$. One may still think of \eqref{LIFT} as the fixed point equation for a symplectic correspondence rather than  a flow.

{Flat tori also exhibit certain kinds of periodicities. Suppose that $H = \{x_1 = 0\}$ is a totally geodesic coordinate slice of the flat torus
$\R^n/\Z^n$. The geodesic flow is $G^t(x, \xi) = (x + t \frac{\xi}{|\xi|}, \xi)$ and it leaves invariant the tori $T_{\xi} = \{(x, \xi): x \in \R^n/\Z^n\}
\subset T^* \R^n /\Z^n$. The coordinate slice defines a transversal to the Kronecker flow on each $T_{\xi}$. Fixing $|\pi_H \xi | = c$ forces
$|\xi_1| = \sqrt{1 - c^2}$. The return time of $(x, \xi) \in S^c_{x_1 = 0}  \R^n/\Z^n$ to the slice on $T_{\xi}$ is the time $t$ so that $t \xi_1 =0$,
or $t= (1 - c^2)^{-\half}$. Thus, the return time is independent of the invariant torus and one has periodicity of the return to $S^c_H M$ even
though the geodesic flow of $M$ fails to be periodic. }

It is  plausible from 
\eqref{JDEF} that  $J_{\epsilon, H}^c(\lambda_j)$ should   attain its maximal size when  the the multiplicity of $\lambda_j$ is
maximal and when there is clustering of the $\sqrt{-\Delta_H}$-spectrum $\{\mu_k\}$ around $\lambda_j$,  forcing both geodesic flows $G_M^t$ and 
$G_H^t$ to be periodic.  But the jump   depends on the sizes of the Fourier coefficients as well as 
 the spectrum.

To understand the general picture of jumps and remainder estimates, the  reader may keep in mind some examples in the  simplest case where  $\dim M =2$ and $H$ is a geodesic. Periodicity
of $G^s_H$ coincides with $H$ being a closed geodesic.  Let us consider four examples  (see Section
\ref{SHARPSECT} for further discussion): (i) $M = \Ss^2$ and $H = \gamma$ is the equator; (ii) $M $ is a convex surface of revolution 
and $H = \gamma$ is the equator (see Section \ref{SURFOFREV}); (iii) $M$ is a non-Zoll surface of revolution in the shape of a `peanut', i.e. has a periodic hyperbolic geodesic `waist' 
$\gamma$ and a top and bottom convex parts, each with a unique elliptic periodic geodesic; (iv) $\gamma$ is a closed geodesic of a hyperbolic
surface. In all cases, the geodesic flow of $H$ is periodic. In case (i) the geodesic flow of $M$ is periodic, while in case (ii) it is not.
In case (i) the multiplicity of the $N$th eigenvalue is $2N-1$ while in case (ii) all eigenvalues have multiplicity $\leq 2$.
Yet both (i)-(ii) have Gaussian beams along $\gamma$, and when $c=1$ they are the only eigenfunctions contributing to the Kuznecov-Weyl
asymptotics; hence the asymptotics are the same in both cases. Case (iii) is different in that $\gamma$ is now hyperbolic and there do
not exist standard Gaussian beams along it but there does exist an eigenfunction which concentrates on $\gamma$ due to the fact that
this example is quantum completely integrable. To our knowledge, the $L^2$ norm of its restriction (or equivalently, its Fourier coefficient
with $|n| = \lambda$) have not been determined;  In case (iv) there should not exist any such concentrating eigenfunctions. One would
expect at least logarithmic improvements on the Fourier coefficient bounds, as in the case where $c=0$ (see \cite{WX,   SXZh17, CG19}).

\subsection{Review of results on $L^2(H)$ norms of restrictions}\label{BGTREV}

Before discussing examples, we compare  the results of   Theorem \ref{main 3}  with
prior results on $L^2$ norms of restrictions \cite{BGT}.  In the notation of \cite{BGT}, the estimates take the form,\footnote{The notation in \cite{BGT} is $\dim M = d, \dim H = k$}
$$
	\|\phi_{\lambda}\|_{L^2(H)} \leq C (1 + \lambda)^{\rho(d, n)} \sqrt{\log \lambda} \|\phi_{\lambda}\|_{L^2(M)}
$$
where
$$
\rho(d,n) =
\begin{cases}
	\frac{n-1}{4} - \frac{n-2}{4} =  \frac{1}{4}, & d = n-1, \\
	\frac{1}{2}, & d= n-2, \\
	\frac{n-1}{2} - \frac{d}{2}, & 1 \leq d \leq n-3
\end{cases}
$$ 
and where the $\sqrt{\log \lambda}$ in the bound can be removed if $d \neq n-2$. (See also \cite{Hu}.)
 
The  problem of finding extremals for restricted $L^2$ norms on submanifolds is studied in \cite{BGT}. It is shown that 
extremals  vary between Gaussian beams and zonal spherical harmonics depending on the pair $(n, d)$. The most
difficult case is where $d= n -2$.

When $\dim M = 2, \dim H=1, c=1$ and $H$ is totally geodesic, the estimates on $\|\gamma_H \phi_j\|_{L^2(H)}^2$ and \eqref{JDEF}  can be  the same (see \cite{WXZ+})).  But for $\dim M > 2, $
 the estimates on individual norms are significantly smaller, illustrating that \eqref{JDEF} is an average and that,  when
 $\gamma_H \phi_j$ is not an eigenfunction of $H$ for every $j$,  the Kuznecov-Weyl sums are of a different nature from $L^2$-norms of restrictions. 
 In general, the  sum \eqref{JDEF} is a very thin sub-sum of  \eqref{PLANCHEREL}
and \eqref{c} is a very thin sub-sum of the  Weyl type function for restricted $L^2 $ norms,
\begin{equation} \label{L2DEF}  N_{L^2(H)}(\lambda): = \sum_{j: \lambda_j \leq \lambda}  \int_H |\gamma_H \phi_j|^2 d V_H.  \end{equation}

 Further details are given in the examples
below.

 \subsection{Examples illustrating different types of Fourier coefficient behavior}\label{SECTEX} As mentioned above, the jumps \eqref{JDEF} are
 averages over modes of $H$,  and also involve sums over repeated eigenvalues of $M$ and $H$  when there exist multiple eigenvalues.
 We now list  some of the issues involved in relating remainder estimates on \eqref{JDEF} to estimates on individual Fourier coefficients of
 individual eigenfunctions. The issues are illustrated on  standard spheres $\Ss^n$ in Section  \ref{2Sphere}.
\begin{itemize}
\item Multiplicity issues: The eigenspace $\hcal(\lambda_j)$ may have  a large dimension $m(\lambda_j)$ , so that \eqref{JDEF}
is an $m(\lambda_j)$-fold  sum over an orthonormal basis of eigenfunctions of $\hcal(\lambda_j)$.  See Section  \ref{2Sphere} for the  example of standard  spheres $\Ss^n$. 

Also, 
the $\sqrt{-\Delta}_H$-eigenspaces may have large dimension, so that for each $\lambda_j$, the $\mu_k$ in \eqref{JDEF}   sum is over many
'Fourier coefficients.' This again is illustrated by sub-spheres of spheres (Section \ref{2Sphere}).

  \bigskip

\item Fourier-sparsity of restricted eigenfunctions : It may occur  (and does in the case of latitude circles of $\Ss^2$)  that $\Delta_M$
has a sequence of  eigenspaces of high multiplicity but,  for each mode $\psi_k$ of $H$ and $\lambda$ in the spectrum of $\sqrt{-\Delta_M}$, there  exists a single
eigenfunction $\phi_{j}$ in a given orthonormal basis of the $\lambda$-eigenspace with   a non-zero $k$th Fourier coefficient $\langle \phi_{j}, \psi_k \rangle$.
Alternatively, for each eigenfunction $\phi_j$ in the eigenbasis for $L^2(M)$, there might exist a single $\psi_{k}$ for which the Fourier coefficient is non-zero.
An extreme (and interesting) case occurs   when the restriction $\gamma_H \phi_j$  of an eigenfunction of $M$ is an eigenfunction of $H$ (it is unknown
when this occurs; see Section \ref{EIGRESEIG}.)
 \bigskip

\item Codimension of $H$. 
 The   higher the dimension of $H$, the higher the number of eigenvalues
$\mu_k: |\mu_k - c \lambda_j| \leq \epsilon $, hence the greater amount of averaging in \eqref{JDEF} for fixed $\phi_j$.  In the extreme case of curves,  $\dim H =1$, the  $\sqrt{-\Delta_H}$-spectrum is an arithmetic progression with large gaps, and for $\epsilon$ sufficiently
small, the sum over $k$ might have just one element $\mu_k$. This eliminates the  $H$-multiplicity aspect. However, there can be many $\Delta_M$-eigenfunctions which
restrict to the same (up to scalar multiple)  eigenfunction of $H$; see Section \ref{2Sphere}.

The results of \cite{BGT} reviewed in Section \ref{BGTREV} show that the $L^2$ norms of restrictions
of eigenfunctions decrease linearly with the dimension of $H$. This in some sense balances the additional growth rate of eigenvalues as the dimension of $H$
increases. This issue is only  relevant for $c = 1$.
\bigskip

\item Uniformity of Fourier coefficients. 
Another  interesting scenario, which probably holds for compact hyperbolic surfaces at least, 
is where the Fourier coefficients $|\langle \phi_j, \psi_k \rangle_{L^2(H)}|$ are uniform in size as $\mu_k$ varies  in the `allowed window' where $|\mu_k| < \lambda_j$.
This is the opposite scenario from Fourier sparsity. 

\bigskip

\end{itemize}

 The sparsity phenomenon is illustrated in  Section \ref{2Sphere} for the  standard  2-spheres $\Ss^2$.  In the case where $\gamma_H \phi_j = c_{j,k} \psi_k$
 for some $(j,k)$, $|\langle \gamma_H \phi_j, \psi_k \rangle|^2 = c_{j,k}^{-1} \|\gamma_H \phi_j\|_{L^2(H)}^2$.

\subsection{Restrictions to curves in a convex  surface of revolution  in $\R^3$} \label{SURFOFREV}
 In this section, we illustrate some of the possible types of Fourier
coefficient behavior in the 
case where $H$ is a latitude circle (an orbit of the rotational action around the third axis) of a convex surface of revolution $(\Ss^2, g)$
in $\R^3$ and for the joint eigenfunctions
 $\phi^{m}_{\ell}$ of the Laplacian and of the generator $\frac{\partial}{\partial \theta} $ of
 rotations around the $x_3$ axis in $\R^3$.  Much of this
material can be found in \cite{WXZ20} for the standard metric on $\Ss^2$.


The geodesic flow of $(\Ss^2, g)$ is completely integrable, since
rotations commute with the geodesic flow. The Hamiltonian $|\xi|_g$ of the geodesic flow Poisson commutes with the  angular momentum, or 
Clairaut integral,
$p_{\theta}(x, \xi) = \langle \xi, \frac{\partial}{\partial \theta} \rangle = \left|\frac{\partial}{\partial \theta} \right|_{H_{\phi_0} } \cos \angle (\frac{\partial}{\partial \theta},
\dot{\gamma}_{x, \xi}(0)), \;\; (x, \xi) \in T_x^*\Ss^2. $
The moment map for the joint Hamiltonian action is defined by  $\pcal: = (|\xi|, p_{\theta}): T^*\Ss^2 \to \R^2$.
 A level set  $ \Lambda_{a} = \pcal^{-1}(a, 1) \subset S^* \Ss^2$  is a Lagrangian torus
when $a \not= \pm 1$ and is the equatorial (phase space) geodesic when $a = \pm 1$.
A ray or ladder in the image of the moment map $\pcal$  is defined by $\{(m, E): \frac{m}{E} = a\} \subset \R^2_+$, and its  inverse
image under $\pcal$ is $\R_+ \Lambda_{a}  \subset T^* \Ss^2$. 

 If $(\theta, \phi)$ denote spherical
coordinates with respect to $(M,g)$ (i.e $\phi$ is the distance from the north pole, $\theta$ is the angle of rotation from a fixed meridian), then
an orbit of the rotation action is a latitude circle $H_{\phi_0}$ with fixed $\phi = \phi_0$. We denote by  $\frac{\partial}{\partial \phi}$  the unit vector field tangent to the
meridians.

The parameter $c$ is related  to the values of $p_{\theta}$ by the formula, \begin{equation} \label{Tc} \frac{|p_{\theta}(x, \xi)|}{|\xi|} = c \left| \frac{\partial}{\partial \theta} \right|_{H_{\phi_0}}, \;\;\; (x \in H_{\phi_0}). \end{equation}
To see this, let  $u_{\theta}(\theta, \phi)  : =   \left| \frac{\partial}{\partial \theta} \right|^{-1}_{H_{\phi}} \;\frac{\partial}{\partial \theta} $
and   let $u_{\theta}^*, u_{\phi}^*$ be the dual unit coframe field.
The orthogonal projection from $T_{H_{\phi_0}} \Ss^2 \to T^* H_{\phi_0} $ is given by $ \pi_{H_{\phi_0}}(x, \xi)  = \langle \xi, u_{\theta} \rangle u^*_{\theta},$
and \eqref{Tc} follows. 
%
 The reason that the  parameter $c$ is
 not the usual ratio $\frac{p_{\theta}(x, \xi)}{|\xi|}$ is because we choose the operator on $H$ to
 be $\sqrt{\Delta_H}$ rather than $\frac{\partial}{\partial \theta}.$  

We now show that the first return time $T^c_H$ defined in Section \ref{CLUSTERSECT} is a constant when $H$ is a latitude circle of
a  surface of revolution. This is because the ration in  \eqref{Tc} between $c$ and $p{\theta}$ is constant on a latitude circle. Since $p_{\theta}$ is
constant along geodesics, an initial vector in $S^c_{H_{\phi_0}}  \Ss^2$ at time zero will return to $S^c_H \Ss^2$ each time the geodesic returns to $H$. Moreover, in the setting of curves on surfaces, $S^*_{H_{\phi_0}} \Ss^2$  is a cross section to the geodesic flow, hence the first return time to 
$S^*_{H_{\phi_0}} \Ss^2$ is finite almost surely. Given that $p_{\theta}$ is constant on orbits, it follows that $T_{H_{\phi_0}}^c$ is constant  too.

The flow on $H_{\phi_0}$ is of course periodic as well of period $\frac{2 \pi}{L}$ where $L$ is the length of $H_{\phi_0}$.  It follows that
the equation \eqref{LIFT} can have fixed point sets of maximal dimension when $t \not= 0$ on a convex surface of revolution, despite
the fact that the geodesic flow itself is not periodic. Indeed, $G_H^{ - s} \circ \pi_H \circ G_M^{cs + t}  (q, \xi ) = (q, \xi)$ for
any $(q, \eta) \in S^c_{H_{\phi_0}}$ if $t = T_{H_{\phi_0}}^c$.


We now introduce notation for quantum ladders. 
Let $\phi_{\ell}^m$ be the standard orthonormal basis of
joint eigenfunctions of $\Delta$ and of the generator $\frac{\partial}{\partial \theta}$ of rotations around the third axis.  The orthonormal eigenfunctions of $H_{\phi_0}$ are given by
$\psi_m(\theta) = C_{\phi_0} e^{im \theta}$ where $C_{\phi_0} = \frac{1}{L(H_{\phi_0})}$. Hence, the Fourier coefficients \eqref{FCDEF}
are constant multiplies of the Fourier coefficients relative to $\{e^{i m \theta}\}$. 
It follows that the  $m$th Fourier coefficient of $\phi_{\ell}^m$  is its only non-zero Fourier coefficient along any latitude
circle $H_{\phi_0}$, and that  $|\int_{H_{\phi_0}} \phi_{\ell}^m e^{-im \theta} d\theta|^2 = ||\phi_{\ell}^m||^2_{H_c}$.

On the quantum level, a ray corresponds to a  `ladder'   $\{\phi_{\ell}^m\}_{\frac{m}{\ell} = a}$ of eigenfunctions. The possible Weyl-Kuznecov sum formulae for latitude circles $H = H_{\phi_0}$ thus depend on the
two parameters $(\phi_0,  \frac{m}{\ell})$. The first corresponds to a latitude circle, the second to a ladder in the joint spectrum. It is better
to parametrize the ladder as $\frac{\mu_m}{\ell} = c $ as discussed above. 

\subsubsection{The standard $\Ss^2$  \label{2Sphere}}

The standard sphere $(\Ss^2, g_0)$  is of course a special case of a surface of revolution, and the  joint eigenfunctions $\phi_N^m$ are denoted by $Y_N^m$.
The special feature of the standard sphere is that its geodesic flow is periodic and the eigenspaces of the Laplacian have dimensions $2N+1$.
This gives it the special properties discussed in the next subsection.

\subsubsection{Fourier sparsity phenomena} \label{S2SPARSE}
In the case of $\Ss^2$, we slightly re-adjust the definition of $\sqrt{-\Delta} $ to $\sqrt{-\Delta + \frac{1}{4}}-\half$), whose eigenvalues are   $\lambda_N = N$.
Also,  $\mu_m = m \in \Z$. Suppose that
$c \in \Q_+$ and write it in lowest terms as $c = \frac{p}{q}$ with $(p,q) =1$. Then let $\epsilon   > 0$ and consider the
set $\{(m, N): |m - \frac{p}{q} N | < \epsilon  \} = \{(m, N): |\frac{m}{N} - \frac{p}{q} | < \frac{\epsilon}{ N}\}$. Roughly, this is the set of lattice points
inside a strip of width $\frac{1}{n}$ around the ray through $(0,0)$ of slope $\frac{p}{q}$. Of course, the lattice points $\{k(p, q), k \in {\mathbb N}\}$
lie in the strip. But for other lattice points, $ |\frac{m}{N} - \frac{p}{q} | =  |\frac{mq - N p}{N q} | \geq \frac{1}{N q}$, so there are no solutions aside from the lattice
points on the rational ray if $\epsilon < \frac{1}{q}$. Moreover, the possible `gaps'  $\{m - \frac{p}{q} N\} $  in this example
are   $\geq \frac{1}{q}$.  Hence, when $n =2, d=1$ the remainder terms in Theorem \ref{main 2} and 
elsewhere only sum over one eigenvalue of $H$ and the magnitude of \eqref{JDEF} is the magnitude of the extremal Fourier coefficient of
a restricted eigenfunction. 


\subsubsection{$H$ is a closed geodesic of $\Ss^2$ and $c<1$}\label{LEGENDREc<1}
Let $M = \Ss^2$ and let $H$ be a closed geodesic $\Ss^2$.
It is always the case that $\dim \gcal_c^{0,0} = \dim S_c^*H  = 1$ in the case of $\Ss^2$. 
In the rest of this section, we assume   $0 < c < 1$.

For concreteness suppose that  $H$ is a meridian through the north  pole $p$.
Then for any $\xi \in S^*_p S^2$, $G_{S^2} ^{\pi} (p, \xi) \in S^*_H S^2$ and $\exp_p (\pi \xi) = - p$. The same holds for
any $p$ on the meridian geodesic. 
In this case, $\gcal_c^{0} = \gcal_c^{0,0}$ if $H$ is totally geodesic and $c < 1$. There  exist  $(c, s, 0)$ 
bi-angles  if $s$ is a common period for $G_H^s$ and $G_M^{cs}$.  
For fixed $c < 1$, the  $L^2$ norms of the restrictions of spherical harmonics $Y^m_{\ell}$ with $\frac{m}{\ell} \simeq c < 1$ to $H$ are uniformly
bounded above, and therefore so are their $m$th Fourier coefficients. This is consistent with Corollary \ref{JUMPCOR}. On the other hand, $\gcal^{0,0}$ is non-dominant due to periodicity of the geodesic flow, and one cannot improve the remainder estimates. 
 We refer to \cite{Geis} for a recent study of how the restricted
$L^2$ norms vary with $c$.


On the other hand, if we  restrict $Y^m_{\ell}$ to a meridian geodesic, then all the Fourier coefficients in the
range $[-\ell, \ell]$ can be non-zero. We now show that the squares    $ \left| \int_H Y_N^0 (\phi) e^{- i \frac{N}{2} \phi} d \phi \right|^2 $
of the Fourier coefficients of the zonal spherical harmonic along a meridian geodesic are bounded above and below by positive constants, 
proving that  Corollary \ref{JUMPCOR} is sharp for $(n=2, d =1)$.

Let $Y^N_0(\theta, \phi) = \sqrt{(2 N + 1) } P^N_0(\cos \phi). 
$ be the zonal spherical harmonic on $\Ss^2$. Here, $P^N_0(\cos \phi)= P_N(\cos \phi)$ is a normalized Legendre polynomial.  Let $H$ be a meridian geodesic through
the poles of $Y_N^0$. 
The Fourier coefficients of $\gamma_H Y_N^0$ are known explicitly  \cite[(2.6)-(2.7a)-(2.7b)]{HP}. To quote one special value, 
$$\begin{array}{l} P_N(\cos \phi) = \sum_{k=0}^N p_k p_{N-k} \cos (n - 2k) \phi, 
 \end{array} $$
where
$p_j = 4^{-j} {2j \choose j}. $
  Multiplying by  $2N+1$ shows that the $L^2$ norm square of $Y_N^0$ is the
 the partial sum of the harmonic series and equals $\log N + \gamma$, where $\gamma$ is Euler's constant (this calculation was first done in 
  \cite{T09} by a different method). On the other hand, 
 \begin{equation} \label{ZONALMER} \sum_{k: |k  -c  N| < \epsilon } \left| \int_H Y_N^0 (\phi) e^{- i k \phi} d \phi \right|^2 =  (2N + 1)\;
 \sum_{k: |k - c N| < \epsilon}  |p_{N-k} p_{N+k}|^2, \end{equation}
for any $\epsilon > 0$ and $c \in (0, 1]$. Since the number of terms in the sum is bounded for fixed $\epsilon > 0$, it suffices to  calculate one term $|p_{N-k} p_{N+k}|^2$ asymptotically by Stirling's formula, and for simplicity of exposition we only calculate the middle case with $c = \half$ and for $N= 2n$ even. Using that 
${2 j \choose j} \simeq 2 \frac{2^{2j}}{\sqrt{\frac{ \pi j}{2}}}$
$$p_{2 n-n} p_{2n+ n} = p_{n} p_{3n} = C_0 \; n^{-1} 4^{-n}  2^{2n} 4^{- 3n} 2^{6n}  = C_0  n^{-1}> 0$$
for a certain constant $C_0 > 0$. Multiplying by $(2 N +1)$ shows that $ \left| \int_H Y_N^0 (\phi) e^{- i \frac{N}{2} \phi} d \phi \right|^2 $ is asymptotically a positive
constant,  corroborating Corollary \ref{JUMPCOR}.  Essentially the same calculation is valid for any $0 < c < 1$ (see \cite{Stan} for the relevant binomial asymptotics).

 \subsubsection{Gaussian beam sequences: $c=1$ and $H$ is totally geodesic}
 Another extremal scenario occurs when $c=1$ and $H$ is totally geodesic, where  the classical ray occurs on the boundary of the moment map image. The corresponding ladder of eigenfunctions consists of the Gaussian beams,
$ C_0 N^{\frac{1}{4} } (x_1 + i x_2)^N$,  around the equator $\gamma$. The standard  Gaussian beams (highest weight spherical harmonics) $\{Y_{N}^{N}\}_{N=0}^{\infty}$ are then a semi-classical sequence of   extremals. 
Their restrictions to the equator $\phi = \frac{\pi}{2}$ are equal to   $C_0 N^{1/4} e^{i N \theta}$. The only non-zero Fourier mode is the
$N$th, and that Fourier coefficient  is of magnitude $N^{1/4}$. The  growth rate of the Kuznecov-Fourier  sum \eqref{c} is that of $\sum_{N: N \leq \lambda} N^{1/2} \simeq \lambda^{\frac{3}{2}} $ with remainder of order $N^{1/2}$.  This situation is studied systematically
in  \cite{WXZ+}.



\subsubsection{Caustic sequences } 
 
 We briefly mention the case where where $c=1$ and $H$ has non-degenerate second fundamental form, although they are not studied in this article. In this case, there exist caustic effects which dominate the estimate of Fourier coefficients. The simplest example is the restriction of the standard
 spherical harmonics $Y_N^m$ to non-geodesic latitude circles, where the Fourier coefficients of certain sequences   blow up at the rate $N^{1/6}$. 
 Such caustic effects on restrictions of eigenfunctions will be investigated systematically in a future article.

\subsection{Higher dimensional spheres $\Ss^n$  }\label{SnSHARPSECT}
We denote by $\Pi_N^{\Ss^n}(x, y)$ the degree $N$ spectral projections kernel on $\Ss^n$. To verify the sharpness of
Corollary \ref{JUMPCOR} we will need to use some explicit formulae for this kernel. We follow  \cite{AH} for notation\footnote{ The pair $(d, n)$ in \cite{AH} corresponds to $(n+1, N)$ in this article.} and refer there for the proofs.

In a well-known way, we  slightly change the definition of $\sqrt{-\Delta}$ to obtain operators $A^{\Ss^n}$ with positive integer eigenvalues:
$$
	A^{\Ss^n} = \sqrt{-\Delta^{\Ss^n}  + \frac{(n-1)^2}{4}} + \frac{n-1}{2}.
$$
We replace $\{\lambda_j\}_{j=1}^{\infty}$ and $\{\mu_k \}_{k=1}^{\infty}$ by ${\mathbb N}$, denoting the eigenvalues of  $A^{\Ss^n}$ by $\{N\}_{N=0}^{\infty}$ and
those of  $A^{\Ss^d}$ by $\{M\}_{M=0}^{\infty}$. Let $\Pi_N^{\Ss^n}$ denote the spectral projections onto the $N$th eigenspace $\hcal_N^{\Ss^n}$ of  $A^{\Ss^n}$, i.e the orthogonal projection onto
the space of spherical harmonics of $\Ss^n$ of degree $N$.

The submanifolds $H$ we consider  are sub-spheres (with  standard metrics) of $\Ss^n$.  
The notation for totally geodesic and latitude subspheres is as follows.
In standard Euclidean coordinates on $\R^{n+1}$, the unit sphere is defined by
$\sum_{j=1}^{n+1} x_j^2 = 1$. We define a latitude $d$-sphere of height $a$  to be the subspheres
$$
	\Ss^{n,d}(a)  : = \{\vec x \in \R^{n+1}:  \sum_{j=1}^{n+1} x_j^2 =1 , \;\; \sum_{j = d + 2}^{n+1} x_j^2 = a^2\}.
$$
Henceforth, we drop the superscript $n$ when the dimension is understood. 
 $SO(d+1)$ act by isometries on $\Ss^n$ and all latitude sub-spheres $H = \Ss^d$ are orbits of the action. 

The sup-spheres $\Ss^d(a)$ are totally geodesic when $a = 0$ and are not totally geodesic if $a > 0$.
In particular, if $d =1$ we obtain the closed geodesic, 
$$\gamma: = \Ss^{n,1}(0) : = \{\vec x \in \R^{n+1}:  \sum_{j=1}^{n+1} x_j^2 =1 , \;\; \sum_{j=3}^{n+1} x_j^2 = 0\},$$
which is the intersection $\Ss^{n} \cap \R^2_{x_1, x_2} $ where $\R^2_{x_1, x_2} \subset \R^{n+1}$ is the plane $x_3 = \cdots = x_{n+1} = 0$.

The  eigenspaces $\hcal_N^{\Ss^n}$ of $\Delta$ on $\Ss^n$ are spaces of degree $N$ spherical harmonics, i.e. restrictions to $\Ss^n$ 
of homogeneous harmonic polynomials on $\R^{n+1}$. We denote the orthogonal projection onto $\hcal_N^{\Ss^n}$ by 
\begin{equation} \label{PROJ}\Pi_N^{\Ss^n}: L^2(\Ss^n) \to \hcal_N^{\Ss^n}. \end{equation} As is well-known, $D_N^n:  = \dim \hcal_N^{\Ss^n} \sim C_n N^{n-1} $  (see e.g. \cite{StW}.) 
The more interesting
calculation occurs when we restrict $\Pi_N^{\Ss^{n}}(x, y)$ to the $SO(d+1)$ invariant coordinate sub-sphere $\Ss^d$ in $x, y$ and and sift out one `Fourier coefficient'
of the sub-sphere, i.e. one degree of spherical harmonic. 
We denote the restriction (in both variables) of \eqref{PROJ} to $\Ss^d \times \Ss^d$ by 
\begin{equation} \label{RESNOT} \gamma_{\Ss^d(a)}  \Pi^{\Ss^n}_N \gamma_{\Ss^d(a)}^* (x, y)
= [\gamma_{\Ss^d} \otimes \gamma_{\Ss^d} \Pi_N^{\Ss^n} ](x,  y), \;\;\;  (x, y \in \Ss^d)
.\end{equation}
When integrating over $\Ss^d$ it is clear that the variables are restricted to this submanifold and we drop the restriction operators $\gamma_{\Ss^d}$
for simplicity of notation.

For any $0 < c  \leq 1$ and $N$,  and for $\epsilon $ sufficiently small (depending on $c$), there exists at most one  $M = M(N, c) $ satisfying $|M - c N| < \epsilon$. We always assume that $\epsilon$ is chosen this way.
In fact, as discussed in Section \ref{S2SPARSE},  there might not exist any such $M$ for small $\epsilon$. To illustrate this, suppose  $c=\half$. Then only even $N$ will contribute and $M = N/2$.
If $N$ is odd, there does not exist any $M$ within $\epsilon$ of $c N$, so that  \eqref{JDEF} is zero 
for odd $N$ and non-zero for even $N$.

We now give an explicit formula for the jumps  \eqref{JDEF} when $\Ss^d \subset \Ss^n$ is any latitude subsphere.  
\begin{lemma}\label{JUMPSPHERE}    Fix $c \in (0, 1]$ and  $(M, N, \epsilon) $  that one (and only one) degree $M(c, N)$ satisfies
$|M - c N| < \epsilon$, one has
\begin{equation} \label{JDEFSn} 
J^{c} _{\epsilon, \Ss^d  }(N) :=  \int_{\Ss^d} \int_{\Ss^d}  \Pi_N^{\Ss^n} (x, y)\;  \Pi_{M(c, N)} ^{\Ss^d}(x, y)  d V_{\Ss^d}(x) dV_{\Ss^d}(y).
\end{equation}
\end{lemma}
\begin{proof}
For each fixed tangential mode $\psi_k$,  the sum over $\ell$ in \eqref{JDEF}  takes the form, 
$$
	\sum_{\ell : \lambda_\ell = N} \left| \int_{H} \phi_{\ell} \overline{\psi_k}dV_H \right|^2  = \int_{\Ss^d} \int_{\Ss^d}  \Pi_N^{\Ss^n} (x, y)\overline{\psi_k(x)} \psi_k(y) d V_{\Ss^d}(x) dV_{\Ss^d}(y),
$$
and the sum over all $k$ 
giving the jump \eqref{JDEF} is given by \eqref{JUMPSPHERE}.
\end{proof}
Thus, there exists at most  one    full projector $\Pi_{M(c, N)} ^{\Ss^d}$ of  $\Ss^d$ contributing to the sum   if $\epsilon < \frac{c}{2} $ is sufficiently small that each term of the
arithmetic progression $\{c N\}_{N=1}^{\infty}  $ is  $\epsilon$- close to only one integer $M(N,c)$. Note that $\gamma_{\Ss^d(a)}  \Pi^{\Ss^n}_N \gamma_{\Ss^d(a)}^* (x, y)$ is not a spherical harmonic on $\Ss^d$  in either variable, but is a sum of harmonics of degrees $0, \dots, N$. The integrals in Lemma \ref{JUMPSPHERE} are special kinds of Clebsch-Gordan integrals. For $n=2$, the integrals are evaluated in \eqref{ZONALMER} when $H$ is a geodesic.  To our knowledge, the integrals have not been studied for subspheres of  higher dimensional spheres, although they are not hard if $d=1$ (see \cite{WXZ+}).


For general $n$ and $d=1$, we restrict to 
 $x_3 = x_4 = \cdots = x_{n+1} =0$. When $0<c<1$ is rational, we choose $M < N$ with $\frac{M}{N} = c$ and  assume $\epsilon $ small enough so it is the unique integer
 satisfying $|M - c N| < \epsilon$.  The restricted Fourier coefficient is then,
$$\int_{S^1}  \int_{S^1} \Pi^{\Ss^{n}}_N(\langle x(\theta_1), y(\theta_2) \rangle)  
e^{- i N (\theta_1 - \theta_2) } d\theta_1 d \theta_2.$$
Let $SO(2) \simeq S^1 \subset SO(d)$ denote the 1-parameter subgroup so that the orbit of $e_1$ is the circle above. In the space
$ \hcal_N^{\Ss^n}$,  the integral sifts out the orthogonal projection  $\Pi_N^{n, M}$  to the  subspace $\hcal_N^M$ of spherical harmonics which  transform by $e^{i M \theta}$ when translated
by the circle action. 
 The integral equals $\gamma_{S^1} \Pi_N^{n, 0}\gamma_{S^1}^*(x, x)$
where $x \in S^1$ is any point. Since  $\dim \hcal_N^M(\Ss^{n}) = C_{n} N^{n-1}$,  its order of magnitude is $N^{n-1}$. This proves that the statement of Theorem \ref{main 3} is sharp for general $n$
and $d=1$. 

 For general $(n,d)$ with $d > 1$,   explicit estimates for  the Fourier coefficient sums of restrictions to latitude spheres  for $\Ss^n$ when $n > 3$ are much
more difficult  by means of classical analysis  (see
Section \ref{LEGENDREc<1} for $n=2$).



\subsection{Flat tori}\label{FLATSECT}

Let $H$ be a $d$-dimensional coordinate plane in the $n$-dimensional torus with the usual eigenfunctions $\phi_j(x) = \exp(i j \cdot x)$ and $\psi_k(x') = \exp(i k \cdot x')$. Here, $x = (x_1,...,x_n)=(x',x'')$ with $x' = (x_1,...,x_d)$ and $x'' = (x_{d+1},...,x_n)$. Then each $ |\langle \gamma_H \phi_j, \psi_k \rangle_{L^2(H)}|^2$ is a power of $2\pi$ times $\delta_{j' - k}$. The ladder sum now just counts the lattice points $j \in \Z^n$ such that $|c|j| - |j'|| < \epsilon$ and $|j| \leq \lambda$. When $0<c<1$, this region is asymptotic to an $\epsilon$ thickening of a (codimension 1) cone in $\R^n$ and has volume to the order of $\lambda^{n-1}$. Now run this construction again except replace the sharp cutoff by a fuzzy cutoff. The main term of the ladder sum agrees (up to a constant and a lower order term) with the volume of this thickened cone, $\lambda^{n-1}$, which does not depend on $d$. 

To illustrate the necessity of the hypotheses of Theorem \ref{main 3}, we consider the two-dimensional case. We write $j = (j_1,j_2) \in \Z^2$ and observe
\[
	N_{\epsilon,H}^c(\lambda) = (2\pi)^{-1} \#\{ j \in \Z^n : |j| \leq \lambda, \ ||j_1| - c|j|| \leq \epsilon \}.
\]
The region capturing the lattice points in the set above,
\begin{equation}\label{lattice ladder}
	\{ \xi \in \R^2 : ||\xi_1| - c|\xi|| \leq \epsilon \},
\end{equation}
has hyperbolas for boundaries and is asymptotically
\[
	\left\{ \xi \in \R^2 : \left|\sqrt{1 - c^2} |\xi_1| - c |\xi_2|\right| \leq \frac{\epsilon}{\sqrt{1 - c^2}} \right\},
\]
the union of two strips of slope $\pm \frac{\sqrt{1 - c^2}}{c}$ and thickness $\frac{2\epsilon}{\sqrt{1 - c^2}}$. We conclude that the area of the region \eqref{lattice ladder} within the the ball of radius $\lambda$ is asymptotic to
\[
	8 \epsilon (1 - c^2)^{-1/2} \lambda,
\]
which is consistent with the main term of Theorem \ref{main 3}. However, if the slope $c/\sqrt{1 - c^2}$ is rational, we may carefully select two different values of $\epsilon$ which yield exactly the same count of points for $N_{\epsilon,H}^c(\lambda)$. The difference in the main terms must be absorbed into the remainder. Hence, the improved remainder {in Theorem \ref{main 3}}  may not be obtained in this setting.

\subsection{Hyperbolic quotients}\label{HYPERBOLIC} 
 Fourier coefficients of restrictions of  eigenfunctions on hyperbolic surfaces  to closed geodesics, 
to horocycles or distance circles is a classical problem in automorphic forms. 
 To our knowledge, the only case studied rigorously to date is that of cuspidal eigenfunctions $\psi$ of the modular curve ${\mathbb  H}^2/SL(2, \Z)$  to a closed horocycle $H_y$ of `height' $y$  \cite[Page 428]{Wo04}. This is a case where $c=1$ and caustic effects occur,
and the estimates of \cite{Wo04} are of the same nature as the $N^{1/6}$ estimate for spherical harmonics in the caustic case. It seems likely
that in the negatively curved case, one can improve such estimates by powers of $\log N$.  Such effects
will be studied in future work.

\subsection{Singularities for $t \not=0$ \label{SINGtnot0} } 
In this section, we discuss the existence of two-term asymptotics discussed  in Section \ref{REMAINDERSECT}.

To determine two-term asymptotics, it is necessary to allow the support of $\hat{\rho}$ to be any finite interval and   to calculate the
contribution of all sojourn times $t \in \Sigma^c(\psi)$ with fixed point sets of maximal dimension. Thus, we
 need to determine the connected components  of all $(c, s, t)$ bi-angles for general $t$, i.e. solve \eqref{BIANGLEDEF} for all $(s,t)$ with $s \in \supp \hat\psi$ and to locate the maximal
components with $t \not=0$ of  the same dimension as the principal component at $t=0$.

We now assume further  that the set of $(c, s, t)$-bi-angles with $t = T \in {\rm singsupp}\; S^c(t, \psi) \backslash \{0\}$  is a union of clean components $Z_j(T)$ of dimension $d_j(T)$, where  $Z_j(T)$ is a component of  $ \gcal^T_c$.
Then, for $t$ sufficiently close to $T$,  there exist Lagrangian distributions $\beta_j$ on $\R$ with singularities only at $t=0$ such that, $$\begin{array}{l} S^c(t, \psi) = 
\sum_j \beta_j(t - T), \;\; \beta_j(t) = \int_{\R} \alpha_j(s) e^{- i s t} ds,\\ \\ 
 \;\; \text{ with}\;\; \alpha_j(s) \sim (\frac{s}{2 \pi i})^{ -1 + \half (n -d)+\frac{d_j(T)}{2}}\;\; i^{- \sigma_j} \sum_{k=0}^{\infty} \alpha_{j,k} s^{-k}, \end{array}$$
where $d_j(T) $ is the dimension of the component $Z_j(T)$.   We refer to Section \ref{CLUSTERSECT}  for background on the role of  periodicity properties of $G_M^t$ and $G_H^s$ in the existence of  maximal
components $Z_j(T)$ for $T \not=0$, i.e. in whether  ``fixed point sets'' defined by \eqref{EQ} and \eqref{LIFT}  can  be of the same full dimension as for the principal component at 
$t = 0$. When $c < 1$, $G_M^{cs + t}$ must take  $S^c_q M \to S^c_q M$ for every $q \in H$,  and moreover must map each 
set $S^c_{q, \eta} M: = \{\xi \in S^c_q M: \pi_H \xi = \eta\}$ for $(q, \eta) \in B^*H$  into itself. 



\begin{proposition} \label{MORESINGS} 
  Let $\rho \in \scal(\R)$ with $\hat{\rho} \in C_0^{\infty}$ and with $0 \notin {\rm supp} \hat{\rho}$. Assume that  the bi-angle equation is clean in the sense of Definition \ref{CLEAN}, and let $\scal_{\psi} = {\rm singsupp}  S^c(t, \psi) \backslash \{0\}.$ Denote by  $d_j$ the dimension  of a component $Z_j$ of  
  $\gcal_c^{t} $ where $t$ is a non-zero period. Then, there exists $\beta_j \in \R$ and a complete asymptotic expansion,
$$  N^{c} _{\rho, \psi, H  }(\lambda) \sim \lambda^{-1 + \half (n -d) }  \sum_{T \in \scal(\psi)} \sum_{\ell=0}^{\infty} \beta_{\ell} (t -T) \; \lambda^{\frac{d_j(T) }{2} -\ell},$$
The asymptotics are of lower order than the principal term of Theorem \ref{main 2} (resp. Theorem \ref{main 5})  unless there exists a maximal component.
\end{proposition} 

To obtain two-term asymptotics of the type discussed in  Section \ref{REMAINDERSECT},  one needs to specify the maximal components, and to calculate the associated $\beta_{\ell}$ and
$\alpha_j$ in geometric terms. In effect, the function $\qcal^c_{\psi, H}(\lambda)$ is a sum over the maximal components for all $t \not= 0$. Its calculation  is postponed to a subsequent study.

\section{Fuzzy ladder projectors and Kuznecov formulae}
In this section, we set up the main objects in the proof  Theorem \ref{main 4}. We use the terminology of
 the `fuzzy ladders'  of \cite{GU89} to describe the main operators and their canonical relations.  However, there are some significant differences in that we consider `fuzzy' ladders with respect
to two elliptic operators with erratically distributed eigenvalues, rather than with respect to a compact group such as $S^1$ in \cite{GU89} with 
a lattice of eigenvalues. 

\subsection{Notation}

Since we are often dealing with operators on product spaces, we use the notation $f\otimes g$ for  a function on $X \times Y$
of the product form $f(x) g(y)$. Linear combinations of such functions are of course dense in $L^2(X \times Y)$ and it suffices
to define operators on product spaces on such product functions.

We introduce the two commuting operators on $M \times H,$
\[
	P_M =  \sqrt{-\Delta_M}  \otimes I \qquad \text{ and } \qquad P_H = I \otimes  \sqrt{-\Delta_H},
\]
and denote an orthonormal basis of their joint eigenfunctions by 
\begin{equation}\label{def phi_j,k}
	\phi_{j,k} = \overline{\psi_k} \otimes \phi_j.
\end{equation}
Thus, we have
$$(P_M,P_H)\phi_{j,k} = (\lambda_j,\mu_k)\phi_{j,k}. $$ As discussed in \cite{GU89},  $P_M$ and $P_H$ are  not quite pseudodifferential operators due to 
singularities in their symbols on $0_H \times \dot T^*M \cup \dot T^*H \times 0_M$, where $0_M$ denotes the zero section. These singularities lie far from the
canonical relations determining the asymptotics and therefore may be handled by suitable cutoffs as in \cite{GU89}. 

We then  introduce the operators on $C^{\infty}(M \times H)$ defined by, 
\begin{equation} \label{PQ} \left\{\begin{array}{l} P : = P_M: =   \sqrt{-\Delta}_M \otimes I
 , \\ \\ 
 Q_c:= c \sqrt{-\Delta}_M \otimes I -  I \otimes \sqrt{-\Delta}_{H} = c P_M - P_H. \end{array} \right. \end{equation} 
The system $(P, Q_c)$ is elliptic;  $Q_c $ is  a non-elliptic first order pseudo-differential operator of real principal type with characteristic variety, \begin{equation}
\label{CHARQ} \Char(Q_c): \{(x, \xi, q, \eta) \in \dot T^*M \times \dot T^*H: c |\xi|_g - |\eta|_{g_H} = 0\}. \end{equation} 


As in \cite{GU89} we are interested in the ``nullspace" of  $Q_c$, i.e. its $0$-eigenspace. The corresponding pairs of eigenvalues
of $(P_M, P_H)$ would concentrate along a ray of slope $c$ in $\R^2$.  Except in rare situations with symmetry, there are at most finitely
many such eigenvalue pairs, but there always exist an approximate or `fuzzy' null-space, intuitively defined  by a strip around the  ray or ladder in the joint spectrum of the pair
$(P_M, P_H)$.
A  `ladder strip' as in  \cite{GU89} is  defined by a strip around a ray in the joint spectrum, 
$$
	\{(\lambda_j, \mu_k):  |\mu_k - c \lambda_j| \leq \epsilon \} \subset \R_+^2.
$$
It is usually difficult to study a strip  directly, and in place of the indicator function of a strip one constructs Schwartz test functions which concentrate in the strip and are
rapidly decaying outside of it in $\R^2$.  When one uses such a test function rather than an indicator function one
gets a  `fuzzy ladder.'

Such ladders, sharp or truly fuzzy, arise when one studies eigenvalue ratios $\frac{\mu_k}{\lambda_j}$ in intervals of width
$O(\lambda_j^{-1})$. These are very short intervals and it is much more difficult to obtain asymptotics for such short intervals than
for intervals of constant width. To this end, one can study wedged (or coned)
$$
	\{(\lambda_j, \mu_k): | \frac{\mu_k}{\lambda_j} - c| \leq \epsilon\} \subset \R_+^2
$$
along the ray of slope $c$ in $\R_+^2$ in the joint spectrum. See section \ref{RELATED} for the corresponding Weyl sums. Asymptotics for cones and some slowly thickening ladders were obtained in \cite{WXZ20}, as well as ladders which are sufficiently fuzzy as in Theorem \ref{main 2}. In this article, the aim is to obtain improved asymptotics for both sharp and fuzzy ladders.

\subsection{Fuzzy ladder projectors} 
To prove Theorem \ref{main 4} by Fourier integral operator methods, we need to smooth out the projection operators onto the spectral subspaces
corresponding to these ladders. We therefore introduce $\psi \in \scal(\R)$ with $\hat{\psi} \in C_0^{\infty}(\R)$ and define, 
\begin{equation} \label{FLP}
	\begin{split}
	\psi(Q_c) &: L^2(M \times H) \to L^2(M \times H), \\
	\psi(Q_c) &= \int_{\R} \hat{\psi}(s) e^{i s Q_c} ds = \sum_{j,k} \psi(\mu_k - c \lambda_j) \phi_{j,k} \otimes \phi_{j,k}^*.
	\end{split}
\end{equation}
Here, $\varphi_{j,k}^*$ denotes the linear functional dual to $\varphi_{j,k}$ in $L^2(M \times H)$, and hence $\varphi_{j,k} \otimes \varphi_{j,k}^*$ is the rank-$1$ projector onto the line spanned by $\varphi_{j,k}$. $\psi(Q_c)$ is sometimes denoted $\Pi_c$ to emphasize that it is an approximate projector.

The operator \eqref{FLP}  is a smoothing of the  sharp fuzzy ladder projection 
of  $L^2(M \times H)$ onto the span of the joint eigenfunctions
for which $| \mu_k - c \lambda_j | \leq \epsilon$ for some $\epsilon > 0$. 
It  is a Fourier integral operator of real principal type, whose properties we now review. 
The  characteristic variety of $Q_c$ is the hypersurface in $\dot T^* M \times \dot T^*H$ defined by \eqref{CHARQ}.
Its characteristic (null) foliation is given by the integral curves of 
the Hamiltonian $c |\xi|_g -  |\eta|_{g_{H}}$, i.e. the orbits of the flow $G^{c s}_M \otimes G_H^{-s}$  on $T^* M \times T^*H$
restricted to the level set $\Char(Q_c) $. 
The next Lemma is similar to calculations in \cite{GU89} and \cite[Proposition 2.1]{TU92}:

\begin{lemma} \label{WFQc}
$\psi(Q_c)$ of \eqref{FLP} is a Fourier integral operator in the class $ I^{-\half}((M \times H) \times (M \times H)), {\ical_{\psi}^c}')$ with canonical relation
\begin{multline*}
	\ical^c_{\psi} :=  \{(x, \xi, q, \eta; x', \xi',  q', \eta') \in  {\Char}(Q_c)  \times  {\Char}(Q_c) : \\
	\exists s \in \supp(\hat{\psi}) \text{ such that } G^{c s}_M \times G^{-s}_{H}(x, \xi, q, \eta) = (x', \xi', q', \eta') \}.
\end{multline*}
The symbol of $\psi(Q_c)$ is the transport of $(2 \pi)^{-\frac{1}{2}} \hat{\psi}(s) |\dd s|^{\half} \otimes |\dd \mu_L|^{\half}$ via the implied parametrization $(s,\zeta) \mapsto (\zeta, G_M^{cs} \times G_H^{-s}(\zeta))$, where $\mu_L$ is Liouville surface measure on $\Char(Q_c)$. 
\end{lemma}

Since we often use the term {\it Liouville measure} we give the general definition.

\begin{definition} \label{LIOUVILLEDEF} Let $Y \subset T^*M$ be a hypersurface defined by $\{f = 0\}$ . By the Liouville measure on $Y$, we mean
the Leray form $\frac{\Omega}{df}$  on $Y$, where $\Omega$ is the symplectic volume form of $T^*M$. \end{definition}

We now prove Lemma \ref{WFQc}.
\begin{proof} 
Let $\zeta = (x, \xi, q, \eta) \in T^* (M \times H)$.
Since $Q_c$ is of real principal type, we may apply  \cite[Proposition 2.1]{TU92} to obtain that $\psi(Q_c) \in I^{-\half} ((M \times H) \times (M \times H), {\ical_{\psi}^c}')$, where
\begin{multline} \label{TU}
	\ical^c_{\psi} = \{(\zeta_1, \zeta_2) \in  \dot{T}^* (M \times H) \times \dot{T}^* (M \times H): \sigma_{Q_c} (\zeta_1) = 0, \\
	\exists s \in \supp \hat{\psi}, \ 
	\exp s H_{\sigma_{Q_c}} (\zeta_1) = \zeta_2\}.
\end{multline}
The Hamiltonian flow $\exp s H_{\sigma_{Q_c}}$ on the characteristic variety is given by
$$
	G^{cs}_M \times G^{-s}_{H}:  \Char(Q_c)  \to  \Char(Q_c).
$$
By \cite[Lemma 2.6]{TU92}, the principal symbol of $\psi(Q_c)$ is $(2 \pi)^{-\frac{1}{2}} \hat{\psi}(s) |\dd s|^{\half} \otimes |\dd \mu_L|^{\half}$.
\end{proof}

\begin{remark} Without the support condition, the wave front relation is an equivalence relation on points of $\Char(Q_c)$, namely $(x, \xi, q, \eta) \sim (x', \xi', q', \eta')$ if they lie on the null bi-characteristic.
\end{remark}

\subsection{Elliptic cutoff} 

Since $\psi(Q_c)$ is not elliptic, we  introduce a second smooth cutoff $\rho \in \scal(\R), $ with $\hat{\rho} \in C_0^{\infty}$ and define
$$
	\rho(P - \lambda) = \frac{1}{2\pi} \int_{\R} \hat{\rho}(t)  e^{- it \lambda} e^{i t P} dt.
$$
Thus, $\rho(P- \lambda) \psi(Q_c) : L^2(M \times H) \to L^2(M \times H)$ is the operator,
\begin{equation} \label{RHOPSI} \rho(P- \lambda) \psi(Q_c) = \sum_{j,k}\rho(\lambda_j - \lambda)   \psi(\mu_k - c \lambda_j) \phi_{j,k}  \otimes
\phi_{j,k}^*. \end{equation} 
To understand its purpose, we note that if  the cutoff $\rho$ were the indicator function of an interval $[-1, 1]$, it would
restrict the $\lambda_j$ to $[\lambda -1, \lambda +1]$, while if $\psi$ were an indicator function, it  would restrict the $\mu_k$ to $|\mu_k  - c \lambda_j| \leq A$. Hence, the pair of cutoffs would restrict the joint spectrum to a rectangle. The smooth cutoffs $\psi, \rho$ should
be thought of as smoothings of such indicator functions.

By Fourier inversion, \eqref{RHOPSI} is given by
$$
	\rho(P- \lambda) \psi(Q_c) = \frac{1}{2\pi} \int_{\R}  \hat{\rho}(t)  e^{- it \lambda} e^{i t P} \psi(Q_c) \, dt
$$
and the next step is to elucidate the integrand. 
For simplicity of notation we denote $\zeta = (\zeta_1, \zeta_2) \in T^*(M \times H)$.  Since the canonical relation of $e^{itP}$ is
the graph of the bicharacteristic flow of $\sigma_P$ on $T^*(M \times H)$, the composition theorem for Fourier integral operators gives,

\begin{lemma} \label{WFsPIc} $e^{it P} \psi(Q_c): L^2(M \times H) \to L^2(\R \times M \times H)$ is a Fourier integral operator in the class $I^{-\frac{3}{4}}((\R \times M \times H) \times (M \times H), {\ccal^c_{\psi}}')$, with canonical relation
\begin{multline*}
\ccal^c_{\psi} := \{(t, \tau, G_M^{cs + t} \times G_H^{-s}(\zeta), \zeta) \in T^* \R \times  \Char(Q_c)  \times  \Char(Q_c): \\ 
s \in \supp(\hat{\psi}), \ \tau + |\zeta_M|_g = 0\}
\end{multline*}
In the natural parametrization of $\ccal_\psi^c$ by $(s,t,\zeta) \in \supp \hat \psi \times \R \times \Char(Q_c)$ given by
$$
(t, - |\zeta_M|_g, G_M^{cs + t} \times G_H^{-s}(\zeta), \zeta),
$$
the symbol of $e^{itP} \psi(Q_c)$ is
$(2 \pi)^{-\frac{1}{2}} \hat{\psi}(s)|ds|^{\half} \otimes |\dd t|^{\half} \otimes |\dd \mu_L|^{\half}$,
where $\mu_L$ is Liouville surface measure on $\Char(Q_c).$
\end{lemma}

\begin{proof}
We recall that if
$\chi: \dot T^*M \to \dot T^*M$ is a homogeneous canonical
transformation and $\Gamma_{\chi} \subset \dot T^*M \times \dot T^*M$ is its
graph, and if $\Lambda \subset T^*M \times T^*M$ is any
homogeneous Lagrangian submanifold with no elements of the form
$(0, \lambda_2)$, then $\Gamma_{\chi} \circ \Lambda$ is a
transversal composition with composed relation
$\{(\chi(\lambda_1), \lambda_2): (\lambda_1, \lambda_2) \in
\Lambda\}. $ The condition that $\lambda_1 \not= 0$ is so that
$\chi(\lambda_1)$ is well-defined.

It follows that $e^{it P} \circ \psi(Q_c)$ is a transversal composition, and therefore its order is the sum of the
order $\frac{-1}{4}$ of $e^{it P}$ \cite{DG75}  and the order $-\half$ of $\psi(Q_c)$ (Lemma \ref{WFQc}).
\end{proof}

\section{Reduction to $H$}

The Schwartz kernel of $e^{it P} \psi(Q_c)$ lies in $\dcal'((\R \times M \times H) \times (M \times H))$. To study sums of  squares of inner products, $ \left| \int_{H} \phi_j \overline{\psi_k}dV_H \right|^2$, we  need to restrict the Schwartz kernels to $(\R \times H \times H) \times (H \times H)$. To this end, we introduce
the restriction operator,
$$
	\gamma_H \otimes I: C(M \times H) \to C(H \times H).
$$
For instance (in the notation of \eqref{def phi_j,k}),
$$
	(\gamma_H \otimes I)  (\phi_{j,k}) \in C(H \times H), \;\; (\gamma_H \otimes I)   \phi_{j,k} (y_1, y_2) =(\gamma_H \phi_j) (y_1) \psi_k(y_2).
$$
We are interested in the operator with Schwartz kernel in $C((H \times H) \times (H \times H))$ given by,
\begin{equation}
	(\gamma_H \otimes I)  (\phi_{j,k})  \otimes [(\gamma_H \otimes I)  (\phi_{j,k}) ]^*.  \end{equation}
This is an operator from $L^2 (H \times H) \to L^2 (H\times H).$ 
We may construct this operator as the composition of the rank one projection 
$\phi_{j,k} \otimes \phi_{j,k}^* : L^2(M \times H) \to L^2(M \times H)$ with the restriction operator and its adjoint,
$$(\gamma_H \otimes I)^* : C(H \times H) \to \dcal'(M \times H). $$
Note the the adjoint is an extension as a kind of delta-function and does not preserve continuous functions; see e.g.
 \cite[Proposition 4.4.6]{D73} or \cite{TZ13} for
background. The relevant composition is, $(\gamma_H \otimes I) \circ (\phi_{j,k} \otimes \phi_{j,k}^*) \circ (\gamma_H \otimes I)^*$ given by the map
\begin{align*}
C(H \times H) &\mapsto C(H \times H), \\
K &\to \left( \int_H \int_H K(q_1, q_2) \phi_{j,k}(q_1,q_2) \, dV_H(q_1) \, dV_H(q_2) \right) \cdot (\gamma_H \otimes I) \phi_{j,k}.
\end{align*}
Extending to the operator $e^{itP}\psi(Q_c)$, we define
\[
(\gamma_H \otimes I) \circ e^{itP} \psi(Q_c)  \circ (\gamma_H \otimes I)^* \\ := \sum_{j,k} e^{it\lambda_j} \psi(\mu_k - c \lambda_j) ((\gamma_H \otimes I) \phi_{j,k} ) \otimes
((\gamma_H \otimes I) \phi_{j,k})^*.
\]


If $X \subset M$ is a submanifold, we refer to the  composition $\gamma_{X} F \gamma_{X}^*$   of a Fourier integral operator as its reduction to $L^2(X)$. Such operators are studied in many articles; we refer to \cite{TZ13, Si18} for background. We will need Sipailo's Theorem 3.1 \cite{Si18}, stated below for convenience. In the following, $\pi_X: T_X^*M \to T^* X$ is the natural restriction (or projection) of covectors to $T^*X$.

\begin{lemma}\cite[Theorem 3.1]{Si18} \label{REDUCTIONLEM} 
Let $F \in I^m (M \times M, \Lambda)$ be a Fourier integral operator of order $m$ associated to the canonical relation $\Lambda'
\subset \dot{T}^* M \times \dot{T^*} M$. If
\begin{itemize}
\item (i) \; $\Lambda \cap (T^*_X M \times T^*_X M)$ is a clean intersection, and
\item (ii)\; $\Lambda \cap N^*(X \times X) = \emptyset $,
\end{itemize}
then
$$
	\gamma_X F \gamma_X^* \in I^{m^*}(X \times X, \Lambda_X)
$$
where $\Lambda_X = (\pi_X \times \pi_X)( \Lambda \cap (T^*_X M \times T^*_X M))$, and where
$$
	m^*  = m - \dim X + \half \codim X + \half \dim (\Lambda \cap (T^*_X M \times T^*_X M)).
$$
\end{lemma}

\begin{remark}
We recall that  $X \cap Y$ is a {\it clean}  intersection of two submanifolds of $Z$ if $X \cap Y$ is a submanifold of $Z$ and
 $T_p(X\cap Y) = T_p X \cap T_p Y$ at all points $p \in X\cap Y$.
Failure of clean intersection can happen in two ways: (i) $ X\cap Y$ fails to be a submanifold of $Z$, or (ii) it is a submanifold of $Z$
but $T_p (X\cap Y) \not= T_p X \cap T_p Y$. Usually, (i) is easy to check, but (ii) can be hard to check when (i) holds.
\end{remark}

We also quote the result in the  simplest case, when  $\Lambda$ is the graph of a symplectic diffeomorphism $g$ of $\dot{T}^* M$.  Here and henceforth,   for any 
subset  $V \subset T^* M$, $\dot{V} = V \backslash \{0\}$ is $V$ minus its intersection with the zero section. The graph of
$g$ is denoted by ${\rm Graph}(g): = \{(g(\zeta), \zeta): \zeta \in \dot{T}^* M\}$.

\begin{corollary} \cite[Corollary 3.9]{Si18} \label{REDUCTIONCOR}  Let $\Phi$ be a Fourier integral operator  of order $m$ quantizing a canonical
transformation $g$. Assume that $g$ satisfies the conditions:
\begin{itemize}
\item (i) \; the intersection $\dot{T}_X^* M \cap g (\dot{T}^*_X M) $ is clean. \bigskip

\item (ii) \;  $\dot{N}^* X \cap g(\dot{N}^*(X)) = \emptyset$.

\end{itemize}

Then, $(\pi_X \times \pi_X)({\rm Graph} (g)  \cap (T^*_X M \times T^*_X M))  $ is a Lagrangian submanifold of $\dot{T}^*(X \times X)$ 
and $\gamma_X \Phi \gamma_X^*$ is a Fourier integral operator in the class $I^{m^*}(X \times X, \pi_X \times \pi_X)({\rm Graph} (g))$,
where $m^*$ is defined in Lemma \ref{REDUCTIONLEM}.
\end{corollary}

\subsection{Generalization to $F(t) :=  e^{it P} \psi(Q_c)$} \label{FtSECT}

In fact, we need to extend  Lemma \ref{REDUCTIONLEM}  to the case  where $F$ 
is replaced by  $  F(t) :=  e^{it P} \psi(Q_c)$  as   in Lemma \ref{WFsPIc}.  We state the result in the generality
of Lemma \ref{REDUCTIONLEM} where $F$ is replaced by $F(t) = e^{it P} F$, where (for each $t$)  $e^{it P}: L^2(M) \to L^2(M)$
is the unitary group generated by a first order elliptic pseudo-differential operator $P$.  Note that $F: \dcal'(M) \to \dcal'(\R \times M)$,
so the domain and range are not the same, and  Lemma \ref{REDUCTIONLEM} does not apply as stated, although it applies
for each fixed $t$.

As is well-known \cite{DG75}, $e^{it P}
\in I^{-\frac{1}{4}}( \R \times M \times M, \wt{\rm Graph}(g^t))$ where $\wt{\rm Graph}(g^t) = \{(t, \tau, g^t(\zeta), \zeta): \tau + \sigma_P(\zeta) = 0\}$
is the space-time graph of the flow.
Let $F \in I^m(M \times M, \Lambda)$ be a Fourier integral operator as in Lemma \ref{REDUCTIONLEM},  let $e^{it P}$ be as in the preceding paragraph and let $F(t) =e^{it P} F$. Assume:
$$
	(**)\;\; \; \wt{\rm Graph}(g^t) \circ \Lambda' \subset \dot{T}^* \R \times \dot{T}^* M \times \dot{T}^* M  \;\; \text{ is a clean composition}.
$$
Then by the composition theorem for Fourier integral operators, 
$F(t,x, y) )  \in I^{m - \frac{1}{4}} (\R \times M \times M, \wt \Lambda)$ is  a Fourier integral operator  associated to the canonical relation, $$\wt \Lambda : = \{(t, \tau, g^t(\zeta), \zeta'):   \tau  + p(\zeta) = 0, (\zeta, \zeta') \in \Lambda\}  \subset \dot{T}^* \R \times  \dot{T}^* M \times \dot{T^*} M. $$
We define the t-slice of $\wt \Lambda$ by,
$$\wt{\Lambda}_t :=  \{(g^t(\zeta), \zeta'):  (\zeta, \zeta') \in \Lambda\} \subset  \dot{T}^* M \times \dot{T^*} M. $$
Since $g^t$ is a diffeomorphism, it is evident that if $\Lambda$ is a manifold, then so is $\wt{\Lambda}_t$ for every $t$. It follows
that also  the spacetime graph $\wt \Lambda$  of $\wt{\Lambda}_t$ is a manifold and   that the composition is always clean.

We continue to use the notation $\gamma_X F(t) \gamma_X^*$ for the restriction in the $M$-variables, and $\pi_X$ as the projection
of a covector in $T^*M$ to $T^* X$, both with $t$ as a parameter.

\begin{lemma} \label{REDUCTIONLEMEXT} With the notation and assumptions of Lemma \ref{REDUCTIONLEM}, let $F(t) := e^{it P} F: C^\infty(M) \to C^\infty(\R \times M)$. 
Let $X \subset M$ be a submanifold and assume: 
 \begin{enumerate}
\item[(i)] For each $t$, $\wt \Lambda_t  \cap  \dot{T}^*_X M \times \dot{T}^*_X M$ is a clean intersection; 
\item[(ii)] For each $(\zeta, \zeta') \in \Lambda$, the curve $t \to g^t(\zeta) $ intersects $T^*_X M$ cleanly;
\item[(iii)] For each $t$, $\wt \Lambda_t \cap  \dot{N}^*(X \times X) = \emptyset $.
\end{enumerate}
Then,
 $\gamma_{\R \times X} F(t)\gamma_{ X}^* \in I^{m^*}(\R \times X \times X, \wt \Lambda_X)
$
where
\[
	\wt \Lambda_X = (I \times \pi_X \times \pi_X)(\wt \Lambda \cap (\dot{T}^* \R \times \dot{T} ^*_X M \times \dot{T}^*_X M)),
\]
and where
$$
	m^*  = \ord F(t) +  \frac{1}{2} \codim X + \half  \dim  \wt \Lambda \cap (T^*\R \times T^*_X M \times T^*_X M)  -  \half(2 \dim X  + 1) 
$$
Here, $\ord F(t)$ denotes the order of $F(t) $ as a Fourier integral kernel in $I^*(\R \times M \times M; \wt \Lambda)$.
\end{lemma}


\begin{proof} (Sketch) Since the proof is almost the same as for Lemma \ref{REDUCTIONLEM} we only provide a brief sketch
of the proof, emphasizing the new aspects.

We claim first that the conditions (i) - (iii) of the Lemma are equivalent to:
\begin{enumerate}
\item[(i)'] \; $\wt \Lambda \cap (\dot{T}^*\R \times \dot{T}^*_X M \times \dot{T}^*_X M)$ is a clean intersection, and
\item[(ii)'] \; $\wt \Lambda \cap \dot{T}^* \R \times  \dot{N}^*(X \times X) = \emptyset $.
\end{enumerate}
It is obvious that (iii) and (ii)' are equivalent, so we only show that (i)' is equivalent to $(i) - (ii)$. In fact, it is clear that (i)' implies
(i)-(ii) since $dt \not=0 $ and all $t$ slices of the intersection are submanifolds if (i)' is a clean intersection and the tangent space
condition is satisfied.  The non-trivial statement is the  converse, that (i)-(ii) implies (i)'.  To prove it, let $f_X: M \to \R^k$ be a local defining
function for the codimension k  submanifold $X$, i.e. locally $X = \{f_X = 0\}$ and $df_X$ has full rank on $X$. For instance, one may use the normal
variables $y$ of local Fermi normal coordinates. Then, $\wt \Lambda \cap T^*\R \times T^*_X M \times T^*_X M$ is the set
of points in $\wt \Lambda$ where $\pi^*f_X \times \pi^* f_X  = 0$ (here, we use the notation $\pi^* f_X$ for its pullback to $T^*M$), and the intersection 
in (i)' is a submanifold if $\pi^* f_X \times \pi^* f_X $ (the pullback to $T^*\R \times T^*M \times T^*M$)  is non-singular on $\wt \Lambda$,
that is, if $\pi^* f_X \times \pi^* f_X: \wt \Lambda \to \R^k \times \R^k$ has a surjective differential at each point $(t, \tau, g^t(\zeta), \zeta')$.
Since $\pi^* f_X \times \pi^* f_X ((t, \tau, g^t(\zeta), \zeta')) = (f_X (\pi g^t(\zeta)), f_X(\zeta'))$
we can first calculate $D \pi^* f_X \times \pi^* f_X: T \wt \Lambda \to \R^k \times \R^k$ on tangent vectors
to curves in $\zeta$ and $\zeta'$ for fixed $t$ and then for for tangent vectors as $t$ varies. If $t$ is fixed, then the calculation is
the same as for the $t$-slice and by assumption (i) the derivative is already surjective. 
Afortiori, it is surjective if $t$ is allowed to vary. 

The more difficult condition is that the tangent space to the intersection equals the intersection of the tangent spaces. As mentioned
in the introduction, the tangent space of the intersection always contains the intersection of the tangent spaces but may possibly be larger.
If we decompose the tangent space into the $\frac{\partial}{\partial t}$ direction and the tangent vectors to the slices, we find that the
only condition not contained in (i) is that the the  tangent vectors to an  orbit $t \to g^t(\zeta)$ may be tangent to the intersection but
not in the intersection of the tangent spaces. Since this vector lies only in one component, the condition that there are no such
additional tangent vectors is precisely that the curve  $g^t(\zeta) $ intersects $T^*_X M$ cleanly.

 Once it is proved that the composition is clean, the  order can be calculated just using the standard calculus of Fourier integral operators under clean composition
\cite{HoIV, DG75, D73}.  As in \cite[(1.20)]{DG75} or \cite[Example, page 111]{D73},  restriction $\gamma_X$ to a submanifold is a Fourier integral operator
of order $\frac{1}{4} \codim X$ (after cutting away normal directions, which are irrelevant to our application). The adjoint $\gamma_X^*$
has the same order. 

In our application, the domain and range are different, creating an asymmetry in the order calculation. The domain (incoming  variables)  is $M$ and the range (outgoing variables)  is $\R \times M$. The left
 restriction $\gamma_{\R \times X} $ acts on the outgoing variables $\R \times M \to \R \times X$ and  is of order $\frac{1}{4} (\dim M - \dim X)$.
The right restriction  $\gamma_X^*$ acts on the incoming variables,  Hence the
 order of $\gamma_X^*$ is  $\frac{1}{4} (\dim M  - \dim X)$. 
 
 The clean  composition $\gamma_{\R \times X} F(t) \gamma_X^*$ has order
\[
	\ord  \gamma_{\R \times X} + \ord  F(t) + \ord  (\gamma_X^*) + \frac{e}{2} = \ord F(t) + \frac12 \codim X + \frac{e}{2}
\]
 where $e$ is the excess (see \eqref{EXCESS}). This is  the formula for the composition of two operators, but here we extend
 it to three operators by successively computing excesses and adding the two excesses.  To complete the proof, we need to show
 that 
 \begin{equation} \label{e/2} \frac{e}{2}  =\half  \dim  \wt \Lambda \cap (T^*\R \times T^*_X M \times T^*_X M)  -  \half(\dim X + \dim X + \dim \R) . 
 \end{equation}
 
 The excess of a general clean composition $A_1 \circ A_2$   of Fourier integral operators, with respective canonical relations $C_1 \subset \dot T^* X \times \dot T^* Y$, $C_2 \subset \dot T^*Y \times \dot T^*Z$,  is defined as follows (cf. \cite[Page 18]{HoIV}). The composition is defined
 in terms of the clean intersection $\hat{C} : = C_1 \times C_2 \cap T^*X \times {\rm Diag} (T^* Y \times T^* Y) \times T^* Z$.
Denote by $C$ the range of the map $\pi_{T^* X \times T^* Z} : \hat{C} \to T^*X \times T^* Z, \; (x, \xi; (y, \eta, y, \eta), z, \zeta) \mapsto (x, \xi, z, \zeta) $. Cleanliness
implies that the map $\hat{C} \to C $
has constant rank. The excess is the dimension of the fiber $C_{\gamma}$ over a point $\gamma \in C$.

The composition $\gamma_{\R \times X} F(t)\gamma_X^*$ involves three canonical relations, but in the case of restriction operators there is a
convenient way to  summarize the above two sided composition composition $\WF'(\gamma_{\R \times X} ) \circ \wt \Lambda \circ \WF'(\gamma_X^*) $ (where $\WF'(A)$ is the canonical relation of $A$). Namely, for the right composition, the submanifold $\hat{C}_R$
is $\wt{\Lambda} \cap T^* \R \times T^* M \times T^*_X M $.  It projects to  $\pi_{T^* \R \times T^* M \times T^*X} (\hat{C}_R)$.
Similarly, the  left  composition produces the submanifold $\hat{C}_L:= \wt{\Lambda} \cap T^*\R
\times T^*_X M \times T^*M$. The double composition produces the intersection  $\hat{C} = \hat{C}_L \cap \hat{C}_R = \wt \Lambda \cap (\dot{T}^*\R \times \dot{T}^*_X M \times \dot{T}^*_X M)$ and the projection  map
$$
	\Pi_{T^*\R \times T^*X \times T^*X}: \wt \Lambda \cap (\dot{T}^*\R \times \dot{T}^*_X M \times \dot{T}^*_X M) \to  C \subset T^* \R \times T^*X \times T^*X
$$
projects
the $T^*_X M $ components to $T^*X$.
 The combined excess of the double composition  is the dimension of the fiber of  this map over its image.  We know that
 the map has constant rank,  since $C$ is a Lagrangian submanifold,  $\dim C = (2 \dim X +1)$ and therefore the dimension of the fiber is 
 $$\dim \wt \Lambda \cap (\dot{T}^*\R \times \dot{T}^*_X M \times \dot{T}^*_X M)  - (2 \dim X + 1),$$
 agreeing with \eqref{e/2}. 
\end{proof}

\subsection{Geometry of submanifolds, transversality to the geodesic flow and cleanliness} \label{HSECT}

 We further review the geometry of unclean intersections in restriction problems from \cite{TZ13}. In that article, $H$ was
 assumed to be a hypersurface, whereas here   $H \subset M$ can  be any submanifold. We briefly generalize the statements
 accordingly. In the following, we assume that  $H$ is locally defined in an open set $U$  by $\{f_H = 0\}$ with $df_H \not= 0$ on 
 $H$ and with $f_H  : U \to \R^k$. Then $(f_H, d f_H) : T M \to T \R^k$ is a local defining function of $T H$.
 A natural choice is to use Fermi-normal coordinates along $H$, the coordinates defined by $\exp^{\perp}: N H \to M $. We let
 $(s_1, \dots, s_d)$ be a choice of coordinates on $H$ and let $f_H = (y_1, \dots, y_{n-d})$ be normal coordinates. We also
 let $(\sigma_1, \dots, \sigma_d, \eta_1, \dots, \eta_{n-d})$  be the symplectically dual coordinates on $T^*M$.  The following 
 Lemma explains why geodesics tangent to $H$ may cause a lack of cleanliness. 
 
\begin{lemma} \label{FAILLEM} Let $H \subset M$ be a submanifold. Then,  $S^* H$ is the set of points of $S^*_H M$ where $S^*_H M$
fails to be transverse to the geodesic flow $G^t$, i.e. where the Hamilton vector
field $H_{p_M}$ of $p_M = |\xi|_M$ is tangent to $S^*_H M$. 
\end{lemma}
This is proved in \cite{TZ13} for hypersurfaces. 

\begin{proof} The generator $H_{p_M}$ of the geodesic flow of $M$  is the vector field on $S^*M$ obtained by 
horizontally lifting  $\eta \mapsto \eta^h$ a covector $(q, \eta) \in S^*M $ to $T_{(q, \eta)} S^*M$  with respect to
the Riemannian connection on $S^* M$; here,  we freely identify
covectors and vectors by the metric. Lack of transversality  $G^t$ and $H$ occurs
when $\eta^h  \in T_{(q, \eta)} (S^*_H M)$. The latter
is the kernel of $d f_H$. Since $f_H$ is a pullback to $S^*M$,  $d f_H (\eta^h) = d f_H  (\eta)= 0 $
if and only if $\eta \in T H$. 

When $H$ is totally geodesic, the orbits of $G^t$ starting with initial data  $(s, \xi) \in S^*H$ remain in $S^*H$, proving the last
statement.
\end{proof}

We now consider an example  of unclean bi-angle sets  $\gcal_c$, resp. $\gcal_c^0$,  in the sense of Definition \ref{CLEAN}  of Section \ref{CLEANJF}, which arise when $c < 1$
and the $M$-geodesic of the $(c, s, t)$  bi-angles have transverse intersection with $H$,  i.e. examples
where the solution set of  $G_H^{ - s} \circ \pi_H \circ G_M^{cs + t} (q, \xi) = \pi_H (q, \xi) $ is unclean. In these examples $t \not=0$.

Suppose that  $\gamma_0$ is a closed geodesic of $\Ss^d$, which we  envision  as a meridian through the poles.  Let $SO(2)$ be the one-parameter subgroup of rotations
fixing $\gamma_0$. 
Let $p \in \gamma_0$.
Then for any $\xi \in S^*_p S^2$, $G_{S^2} ^{\pi} (p, \xi) \in S^*_H S^2$ and $\exp_p (\pi \xi) = - p$.

We then define $H$ to be the bumped geodesic in which we  add a small `bump' on some proper subinterval of $\gamma_0$. For instance, we deform $\gamma_0$ by a nearly rectangular bump centered along the equator which is $\epsilon$ in length along $\gamma_0$ and comes away from the geodesic $\gamma_0$ by $\epsilon^2$. 
Then we smooth it out near the corners.

We then consider  $c$-bi-angles
i.e. solutions of $G_H^{ - s} \circ \pi_H \circ G_M^{cs + t} (q, \xi) = \pi_H (q, \xi) $ for some $(s,t)$ and some $(q, \xi)$ with $q \in H$. 

First, assume that $q$ is not on the bump, i.e.  $q \in \gamma_0$,  and let $\xi \in S^c_q H$.  If $ q$ is sufficiently far from the equator and
$\epsilon$ is small enough, then the $M$- geodesic $\gamma_{q, \xi}(\sigma) = G^{\sigma} (q, \xi)$ does not intersect the bump. It will produce a c-bi-angle in which
the geodesic $\gamma_{q, \xi}(\sigma)$ hits $H$ at $\sigma = \pi$.
Let $s$ be the arc-length on $H$ between the two intersection points and define $t$ by $\pi = cs + t$.

We now move the initial data $(q, \xi)$ under $g_{\theta} \in SO(2)$. Let $d(g_{\theta} q)$ be the distance from $g_{\theta} q \in \gamma_0$ to the equator, where the bump lies. 
 As $d(g_{\theta} q) \to 0$, the geodesic $\gamma_{g_{\theta}(q, \xi)}$ first intersects  the bump when  $\theta= \theta_0$. The angle of intersection
 of  $\gamma_{g_{\theta}(q, \xi)}$ and $H$ ceases to be constant at $\theta_0$ and then depends on $\theta$. In particular, the angle ceases to
 be the one corresponding to $c$. Thus, the $c$-bi-angle set has a boundary and therefore the solution set is not even a manifold. For $\theta \geq \theta_0$, a c-bi-angle with footpoint on $H$ no longer
 lies in the 1-parameter family obtained from $g_{\theta} q$ where $q \in \gamma_0$.

 A related example is to put two bumps into $\gamma_0$. An extreme case is that the bumps touch at a point $q_0$ where they are tangent to 
 to $\gamma_0$ to high order. If $c, \epsilon$ are chosen so that $\gamma_{q_0, \xi}(t)$ does not intersect the bumps, then one can find
 $(s, t)$ to have a $c$-bi-angle. The bi-angle cannot be deformed preserving $c$, i.e it is an isolated bi-angle.



\subsection{Microlocal cutoffs}

The hypotheses of Lemma \ref{REDUCTIONLEM}  and Lemma \ref{REDUCTIONLEMEXT} will not be satisfied in all the
cases for which  we wish to prove Theorem \ref{main 3} and Theorem \ref{main 4}. As discussed in Section \ref{SHARPSECT}, 
non-clean intersections arise due to tangential intersections of geodesics with $H$. This problem is discussed at length in
\cite{TZ13}, to which we refer for much of the background. 
As  in \cite{TZ13}, we introduce some cutoff operators  supported away from glancing and conormal directions to $H.$
For fixed $\epsilon >0,$  let $\chi^{(tan)}_{\epsilon}(x, D) = Op(\chi_{\epsilon}^{(tan)}) \in \Psi^0(M),$ with 
homogeneous symbol $\chi^{(tan)}_{\epsilon}(x,\xi)$ supported in an $\epsilon$-aperture conic neighbourhood of $T^*H \subset T^*M$ with $\chi^{(tan)}_{\epsilon} \equiv 1$ in an $\frac{\epsilon}{2}$-aperture subcone. The second cutoff operator $\chi^{(n)}_{\epsilon}(x,D) = Op(\chi_{\epsilon}^{(n)}) \in Op(S^0_{cl}(T^*M))$ has its homogeneous symbol  $\chi^{(n)}_{\epsilon}(x,\xi)$ supported in an $\epsilon$-conic neighbourhood of $N^*H$ with $\chi^{(n)}_{\epsilon} \equiv 1$ in an $\frac{\epsilon}{2}$ subcone. Both $\chi_{\epsilon}^{(tan)}$ and $\chi_{\epsilon}^{(n)}$ have spatial support in the tube $\tcal_{\epsilon}(H)$, the tube of radius $\epsilon $ around $H$ (see \cite[(5.1) and (5.2)]{TZ13} ). 
To simplify notation, define the total cutoff operator
\begin{equation} \label{CUTOFF}
\chi_{\epsilon}(x,D) := \chi^{(tan)}_{\epsilon}(x,D) + \chi^{(n)}_{\epsilon}(x,D).
\end{equation}
We put, 
$$B_{\epsilon}(x, D) = I - \chi_{\epsilon} (x, D).$$

We use these cutoff operators to ensure that the relevant cleanliness conditions are satisfied. However, our ultimate goal is to 
prove singularity results for the traces \eqref{SpsiDEF}, which involve diagonal pullbacks and pushfowards that will erase some
of the uncleanliness problems, and make it possible to remove the cutoffs. Moreover, the cutoff is unnecessary when $H$ is  totally
geodesic.

\subsection{Application to fuzzy ladder propagators}

We now apply Lemma \ref{REDUCTIONLEMEXT} to the operator $e^{it P} \psi(Q_c)$ of Lemma \ref{WFsPIc},
 where  $M$ in the lemma is replaced  with $  M \times H$, and where $X =   H \times H$, so $\dim X = 2 \dim H, \half \codim X = \half \codim H$.  We use the following notation:   $\zeta = (\zeta_M, \zeta_H) \in T^*_H M \times T^*H$ and $\pi_H(\zeta) = (\pi_H(\zeta_M), \zeta_H)$.
 
 Since tangential  intersections of geodesics with $H$ and conormal vectors to $H$ apriori cause problems, we will apply 
 Lemma \ref{REDUCTIONLEMEXT} to the reduction of the  cutoff operator, 
 \begin{equation} \label{CUTOFFUtPi} (B_{\epsilon}(x,D) \otimes I) e^{it P} \psi(Q_c): L^2(M \times H) \to L^2(\R \times M \times H). 
\end{equation}
The tensor product notation  means, as usual, that the  cutoff is applied only in the $M$ variables (we omit  tensor product with $I$ for the time variables). The cutoff has the effect of cutting
down the canonical relation $\ccal^c_{\psi}$  of $e^{it P} \psi(Q_c)$ in Lemma \ref{WFsPIc} to,
\begin{equation} \label{CUTOFFCR} 
\ccal^c_{\psi, \epsilon} := \{(t, \tau, \ G_M^{cs + t} \times G_H^{-s}(\zeta), \zeta) \in \ccal^c_{\psi} : (1 - \chi_{\epsilon}) (G_M^{cs + t}(\zeta_M))) \not= 0. \}
\end{equation}
Here, we use that if $F$ is any Fourier integral operator with canonical relation $C_F = \{(x, \xi, y, \eta)\} \subset T^* X \times T^*Y$
and symbol $\sigma_F$, 
and if $a(x, D)$ is a pseudo-differential operator, then the symbol of the left composition $a(x, D) F$ at $(x, \xi, y, \eta)$
is $a(x, \xi) \sigma_F(x, \xi, y, \eta). $

\begin{lemma} \label{LRLEMLEM} 
Let $\ccal_{\psi, \epsilon}^c$ be the canonical relation of \eqref{CUTOFFCR}. Then, for $c < 1$, the intersection
\begin{equation}\label{INTER}
	\ccal_{\psi, \epsilon}^c \cap (T^*\R \times T^*_{H } M \times T^*H \times T^*_H M \times T^*H )
\end{equation}
is always clean.
\end{lemma}

\begin{proof}

 Denote
the codimension of $H$ by $\codim H = k$.
 By Lemma \ref{REDUCTIONLEMEXT}, the cleanliness of \eqref{INTER}  holds as long as it holds for (i)  fixed t-slices, and (ii) 
for $t$-curves with fixed $\zeta$.

 In the case of fixed $t$ slices, we claim that the intersection is clean if and only if
\begin{equation} \label{GCAP}
	G^{t + c s} _M(S^*_H M) \cap S^*_H M
\end{equation}
is a clean intersection for all $s \in \supp \hat \psi$ and $\zeta_M$ in the support of the above cutoff. Indeed, the intersection \eqref{INTER} at time $t$ is parametrized by the points $(\zeta, s) \in S^*_H M \times \supp \hat \psi$  
such that
$$
	G_M^{cs + t} \times G_H^{-s} (\zeta)  \in S ^*_H M \times B^*H.
$$
	 Since cleanliness in the $ \zeta_H$ component is automatic, 
	 the cleanliness of this intersection is equivalent  to cleanliness of the intersection    \eqref{GCAP}. By Lemma \ref{FAILLEM}, the intersection is necessarily clean
	 unless there exists $\zeta_M \in S^*_H M$ such that $G^{t + cs}_M(\zeta_M)  \in S^* H$.   However, in this case $(1- \chi_{\epsilon}) (G^{t + cs}_M(\zeta_M)) = 0$ and the point is not in the
	 canonical relation.

	  In addition, we need to check that the
	 orbits $t \to G_M^{cs + t}(\zeta)$ intersect $S_H^c M$ cleanly. But this case is again covered by Lemma \ref{FAILLEM} for 
	 the same reasons as above.
\end{proof}

\begin{proposition} \label{LRLEMPROP} Let 
\begin{multline*}
\Gamma^c_{\psi, \epsilon} := (\pi_{\R \times H \times H} \times \pi_{H \times H}) \ccal_{\psi, \epsilon}^c \cap (T^*\R \times T^*_{H } M \times T^*H \times T^*_H M \times T^*H ) \\
= \{(t, \tau, \pi_{H \times H} \zeta, \pi_{H \times H} (G^{c s +t}_M \times G^{-s }_{H}) (\zeta)) : |\zeta_M|_g + \tau = 0,\\
\zeta \in \Char Q_c \cap T^*_{H \times H} (M \times H), \ G^{c s +t}_M(\zeta_M) \in T^*_H M \\
 (1- \chi_{\epsilon}) (G^{c s +t}_M(\zeta_M)) \not=0,  \ s \in \supp \hat \psi \} \\
\subset T^*\R \times (T^*H \times T^*H \times T^*H \times T^*H).
\end{multline*}
For $0<c< 1$,
$\Gamma^c_{\psi,\epsilon}$
is a Lagrangian submanifold and 
the `reduced' Fourier integral operator 
\begin{equation} \label{LRLEM}
	\gamma_{\R \times H \times H} \circ (B_{\epsilon}(x,D) \otimes I) e^{it P} \psi(Q_c) \circ \gamma_{H \times H}^*
\end{equation}
belongs to  the class 
$$
	I^{\rho(m, d)} (\R \times (H \times H) \times (H\times H), \Gamma^c_{\psi,\epsilon}),
$$
with
$$\rho(m, d) = \ord  e^{it P} \psi(Q_c)  + \half (n-d) + 2 d +\half -\half  (4 d+ 1) = \ord  e^{it P} \psi(Q_c)  + \half (n-d) . $$
\end{proposition}

The principal symbol of this kind of composition is calculated in \cite{TZ13} using symbol calculus and in \cite{Si18} using oscillatory
integrals. We postpone the calculation until the end, since it involves two different restrictions and a pushforward, which can be done in one step rather than in three steps. The ultimate symbol is obtained by a sequence of canonical  pushforward and pullback
operations.


\begin{proof}

By Lemma \ref{LRLEMLEM},  \eqref{INTER} is a clean intersection and it follows from Lemma \ref{REDUCTIONLEMEXT} that 
$ \Gamma^c_{\psi, \epsilon} 
$
is a homogeneous Lagrangian submanifold of $T^* \R \times T^*H \times T^*H \times T^*H \times T^*H$ and that \eqref{LRLEM} is a Fourier integral operator with 
canonical relation $\Gamma_{\psi, \epsilon}^c.$

To complete the proof, we compute the order of \eqref{LRLEM} when $\dim M = n$ and $\dim H = d$. We have, $ \dim (\R \times H \times H ) = 2 d + 1$ and $\half 
\codim (\R \times H \times H  \subset \R \times M \times H) = \half  \codim H = \half (n-d)$.  The main problem is to calculate the dimension,
\begin{equation}
	\label{DIM} D^c(n,d)  := \dim (\ccal_{\psi, \epsilon}^c \cap (T^*\R \times T^*_H M \times T^*H \times T^*_H M \times T^*H)),
\end{equation}
and how it depends on $c$ and on whether or not $H$ is totally geodesic.  Note that  $\wt \Lambda = \ccal_{\psi, \epsilon}^c$  in the notation of Lemma \ref{REDUCTIONLEMEXT}, where
\begin{multline*}
\ccal^c_{\psi, \epsilon} := \{(t, \tau, G_M^{cs + t} \times G_H^{-s}(\zeta), \zeta) \in T^* \R \times  \Char(Q_c)  \times  \Char(Q_c): \\  s \in 
\supp(\hat{\psi}), \ \tau + |\zeta_1|_g = 0\}.
\end{multline*}

\begin{lemma} \label{DIMLEM}  If  $c < 1$ and $H$ is any submanifold of dimension $d$,  we have,
$$
\half D^c(n, d) =    2 d + \half.
$$
\end{lemma}

\begin{proof}
The equation for  $\ccal_{\psi, \epsilon}^c$  involves  $2 + 2n + 2d$  parameters   
$(t, s, \zeta_M, \zeta_H) \in \R \times \R \times \dot T^* M \times \dot T^*H$, and $1 + 2 (n-d)$ constraints. One is that 
$\zeta \in \dot T_H^* M$, so we may regard the parameters as $(t, s, \zeta_M, \zeta_H) \in \R \times \R \times \dot T_H^* M \times  \dot T^*H$,  and then we have
 $1 + (n-d)$ further constraints,
\begin{itemize}

\item $\sigma_{Q_c} (\zeta) = 0$ (which implies  $G_M^{cs +t} \times G_H^{-s}(\zeta) \in {\rm Char}(Q_c)$);

\item $G_M^{t + cs} (\zeta_M) \in T^*_H M$. If $f_H:U \to \R^{n-d} $ is a local  defining function of $H$ in $U \subset M$, then we may write
the $(n-d)$ constraints as  $\pi^* f_H (G_M^{t + cs} (\zeta_M)) = 0$ (see Section \ref{HSECT} and Lemma \ref{FAILLEM}.)

 \end{itemize} 

The first return time from $S^*_H M$ to itself is a smooth function on the support of $1 - \chi_{\epsilon}$ (see \cite[Section 2.3]{TZ13}). Hence, the
solutions $(\sigma, \zeta_M)$  of $\pi^* f_H(G_M^{\sigma}(\zeta_M)) = 0$ is a smooth submanifold of $T^*_H M$ of codimension
$n -d$. 
 
 Hence, the dimension \eqref{DIM} is given by
$$D^c(n,d)  = 2 + d + n + 2d - (n-d)  - 1 = 4 d + 1 . $$

This completes the proof of Lemma \ref{DIMLEM}.
\end{proof}
 
We now complete  the proof of Proposition \ref{LRLEMPROP}.

We now apply  Lemma \ref{WFsPIc},  Lemma \ref{REDUCTIONLEMEXT} and  Lemma \ref{DIMLEM}, with 
  $
X = H \times H \subset M \times H$. 
The  intersection has dimension $4 d + 1$. The fiber dimension 
over $\R \times (H \times H) \times (H \times H)$ is $4 \dim H + 1$, so we subtract $\half (4 d +1)$. 
By Lemma \ref{REDUCTIONLEMEXT},
$$\ord  e^{it P} \psi(Q_c)  + \half (n-d) + 2 d +\half -\half  (4 d+ 1) = \ord  e^{it P} \psi(Q_c)  + \half (n-d) . $$
\end{proof}

%

\section{Asymptotics of $  N^{c}_{ \psi, \rho,  H  }(\lambda)$ : Proof of Theorem \ref{main 4} and Theorem \ref{main 5}} \label{ASYMPTOTICSECT}

\subsection{Diagonal pullback and pushforward to $\R$. }
The next (and final) step is to compose with the diagonal pullback and to integrate over $H$. By the diagonal embedding $\Delta_H \times \Delta_H$
we mean   the partial diagonal embedding
$$
	(\Delta_H \times \Delta_H) (x,y)\in H \times H \to (x,x, y,y)  \in (H \times H) \times (H \times H),
$$
and let $(\Delta_H \times \Delta_H)^*$ be the corresponding pull back operator. For instance, when applied to the rank one 
orthogonal projections onto the  joint eigenfunctions, the Schwartz kernels satisfy
\begin{multline*}
(\Delta_H \times \Delta_H)^* (\gamma_H \otimes I)  (\phi_{j,k})  \otimes [(\gamma_H \otimes I)  (\phi_{j,k}) ]^*(x, y) \\
= (\gamma_H  \phi_j(x) \psi_k(x)) \otimes  \overline{ (\gamma_H \phi_j(y)\psi_k(y)) } \in C(H \times H).
\end{multline*}
  We then compose with the   pushforward under the projection,  $\Pi: \R \times H \times H \to \R$. The pushforward of the eigenfunctions is given by,
\[
	\Pi_*  \left((\Delta_H \times \Delta_H)^* (\gamma_H \otimes I)  (\phi_{j,k})  \otimes [(\gamma_H \otimes I)  (\phi_{j,k}) ]^* \right)
=   \left| \int_H \gamma_H \phi_j(x) \psi_k(x) \, dV_H(x) \right|^2.
\]

 We then apply the pushforward-pullback operation to the Fourier integral operator \eqref{LRLEM}.
To keep track of which components are being paired by the diagonal embedding, we note that \eqref{LRLEM} is
$$
V_H(t, \psi) = \sum_{j,k} e^{it \lambda_j }  \psi (\mu_k - c \lambda_j)  ((\gamma_H \otimes I) \phi_{j,k}  \otimes [(\gamma_H \otimes I )\phi_{j,k}]^*)
$$
and its pushforward-pullback is given by the following: The $S^c(t, \psi)$ defined in  \eqref{SpsiDEF}  is given by,

\begin{align} \label{St}
S^c(t, \psi) &= \Pi_*(\Delta_H \times \Delta_H)^* (\gamma_H \otimes I ) e^{i t P} \psi(Q_c) (\gamma_H \otimes I)^* \\
\nonumber
&= \sum_{j,k} e^{it \lambda_j} \psi(\mu_k-c \lambda_j) \left| \int_H \phi_{j,k}(x,x) dV_H (x) \right|^2.
\end{align}
Of course, 
\begin{align} \label{FTFORM}
{\mathcal F}_{\lambda \to t} d N^{c} _{\psi, H  }(t) &= \sum_{j,k} e^{it\lambda_j} \psi(\mu_k -c  \lambda_j) \left| \int_H \phi_j \overline{\psi_k} \, dV_H  \right|^2\\
\nonumber
&= S_c(t, \psi).
\end{align}

\begin{definition} \label{SOJOURN} Recall $\gcal_c$ from \eqref{gcalc}.  Let
\begin{align*}
\scal_{c, \psi} &:= \{t \in \R: \exists s \in \supp \psi, \gamma \in \gcal_c: \gamma \; \text{ is a $(c, s, t)$-bi-angle}\}, \\
\scal_c &:= \{t \in \R : \exists s \in \R, \gamma \in \gcal_c: \gamma \text{ is a $(c, s, t)$-bi-angle}\},
\end{align*}
\end{definition}

$S^c(t, \psi) $ is the distribution trace of  the restriction of the Schwartz kernel of $e^{it P} \psi(Q_c)$, and 
\eqref{cpsirho} is its integral in $dt$ against $\hat{\rho} (t) e^{it \lambda}$. 
Thus, in   \eqref{St}, we have   expressed the smoothed  Kuznecov sums as compositions of Fourier
integral operators, specifically as the composition of the diagonal pullback  to a suitable diagonal in $H \times H \times H \times H$
and then the pushforward over the diagonal. These operations are also Fourier integral operators, and as reviewed in Section \ref{FIOSECT},
we can use the calculus of Lagrangian distribution under pullback and pushforward to   determine the singularities
 of $S^c(t, \psi)$ and then  the
 asymptotics as $\lambda \to \infty$ of \eqref{cpsirho} (see also \cite[Proposition 1.2]{DG75} and \cite{GU89} and many subsequent articles).

The next step is to prove that \eqref{St} is Lagrangian distribution and to  use Proposition \ref{LRLEMLEM} to calculate the singular set $\scal_c$
of Definition \ref{SOJOURN},  and the order and
principal symbol of \eqref{St} at the singular points. 

We will need to recall the definition of a homogeneous Lagrangian distribution. By definition,   $I^{\frac{\nu}{2} - \frac{1}{4}}(\Lambda_T)$, with 
$\Lambda_T =  \dot  T^*_T \R$,  consists of scalar multiples of the
distribution \begin{equation} \label{I*} \int_0^{\infty} s^{\frac{\nu-1}{2}} e^{- i s (t - T)} ds \end{equation} plus similar distributions of lower order and perhaps a smooth function.

\begin{proposition} \label{MAINFIOPROP} Assume  that $\Gamma^c_{\psi, \epsilon} \subset  \dot T^*(\R \times H \times H \times H \times H)$ (cf. Proposition \ref{LRLEMPROP})  is a Lagrangian submanifold,  let 
$(\Delta_H \times \Delta_H)^*\Gamma^c_{\psi, \epsilon} $ be its pullback under the diagonal embedding $\Delta_H(y, y') = (y, y, y', y')$  in the sense of \eqref{PB} and let $\Pi: \R \times H \times H \to \R$ be the natural projection.   Let 
\begin{equation} \label{LambdacDEF} \Lambda^c_{\psi} := \Pi_*(\Delta_H \times \Delta_H)^*\Gamma^c_{\psi, \epsilon} \end{equation}
be the pushforward-pullback of $\Gamma^c_{\psi, \epsilon}$.  Then,

\begin{align} 
\label{PUSHPULL}\Lambda^c_{\psi}  &= \{(t, \tau) \in T^*\R: \exists (s, \sigma, q, \eta): s \in \supp(\hat{\psi}) : \eqref{EQ} \; \text{ is satisfied}\}\\
\nonumber &= \bigcup_{t \in \scal_{c, \psi}} \gcal_c(t), \;\; (\text{see}\;\; \eqref{gcalct}).
\end{align} 

Moreover,  
\eqref{PUSHPULL} is a clean composition if and only if the equation  \eqref{EQ}  is clean in the sense of condition of Definition  \ref{CLEAN},
and then
\begin{equation} \label{ordStpsi}
S^c(t, \psi) 
\in
I^{n - \frac{7}{4} } (\R, \Lambda_{\psi}^c),  (0 < c < 1).
\end{equation}
The displayed order occurs at $t=0$ and the  symbol $\sigma(S^c(t, \psi))$ at $t =0$ equals, 
\begin{equation} \label{sigmaStpsi}
	\sigma(S^c(t, \psi)) |_{t =0}  = C_{n, d} \; a_c^0(H, \psi) \tau^{n-2} |d \tau|^{\half}.
\end{equation}
$C_{n,d}$ is a dimensional constant depending only on $n = \dim M, d = \dim H$.

\end{proposition}

\begin{remark} \label{IREM} 

$S^c(t, \psi)$ is a sum of (translates of)  homogeneous distributions with singularities at the discrete set $\scal_c$. The order
displayed above is the order of the singularity at $t=0$. The order of the  singularity of \eqref{St} at any $t$ is less than or equal to the order at $t=0$. In the dominant case of Definition \ref{DOMDEF}, the order displayed above only occurs at $ t =0$.
\end{remark}

\begin{proof} The first step is to calculate the wave front relation of the pullback $(\Delta_H \times \Delta_H)^*\Gamma^c_{\psi, \epsilon} $  using  the pullback formula  \eqref{PB}. The  calculation is similar to that of  \cite[(1.20)]{DG75} for the pullback to the `single diagonal' in $M \times M$. The pullback to the `double-diagonal' 
 $\Delta_{H \times H} \subset H \times H \times H \times H$ subtracts  the two covectors at the same base points in the double-diagonal, i.e.
 \begin{multline*}
(\Delta_H \times \Delta_H)^*\Gamma^c_{\psi, \epsilon} = \{(t, \tau, (q,\eta -\pi_H \xi) ; (q', \eta' - \pi_H  \xi')) \in T^*\R \times T^*H \times T^*H: \exists s \\ 
(t, \tau, (q, \eta), (q, \pi_H \xi), (G_H^s(q, \eta), \pi_H G_M^{t + cs}(q, \xi)))  \in \Gamma^c_{\psi, \epsilon}\}, \end{multline*}
or in the notation  $\zeta = (\zeta_H, \zeta_M) =  (x, \xi, y, \eta'')  \in {\rm Char}(Q_c)$ such that 
  $(x,  \xi) \in T^*_H M,   G^{c s +t}_M(x, \xi) \in T^*_H M$,   $\zeta_H = (q, \eta), \zeta_H' = (q', \eta')$, and with $\pi: T^*X \to X$ the natural projection,

 \begin{multline*}(\Delta_H \times \Delta_H)^*\Gamma^c_{\psi, \epsilon}  =
\{(t, \tau, \zeta_H, \zeta_H') :
\exists (s,  \pi_{H \times H} (x, \xi, y, \eta'') , \pi_{H \times H} (G^{ cs +t}_M \times G^{-s }_{H}) ((x, \xi, y, \eta'') )) \in \Gamma^c_{\psi, \epsilon}, \\
(q, q, q', q') =  (x, y, \pi G_M^{s + t}(x, \xi), \pi G_H^{-s}(y, \eta'') ), \\
(\zeta_H, \zeta_H') =   (\Delta_H \times \Delta_H)^*  (\pi_{H \times H} \zeta, \pi_{H \times H} (G^{ c s +t}_M \times G^{-s }_{H}) (\zeta))) \} \end{multline*}

\begin{remark} For the sake of clarity, we note that $ (\Delta_H \times \Delta_H)^*\Gamma^c_{\psi, \epsilon} $  consists of analogues for c-bi-angles of geodesic loops. Unlike
a closed geodesic, the initial and terminal directions of a geodesic loop do not have to be the same. A bi-angle is the analogue of a closed
geodesic but the `bi-angle-loop'  consists of two geodesic arcs, an $M$-arc and an $H$-arc  from $q$ to $q'$, with 
no constraint that the projection of the initial or terminal directions of the $M$ arc agree with those of the $H$ arc.

\end{remark} 

 Next, we pushforward the canonical relation  $ (\Delta_H \times \Delta_H)^*\Gamma^c_{\psi, \epsilon}$  under the projection $\Pi_t: \R \times H \times H \to \R$,
 $\Pi_t(t, \zeta_H, \zeta_H') = t$.  As in   \eqref{PF}  (cf.  \cite[(1.21)]{DG75}), the pushforward 
operation erases  points of $$ (\Delta_H \times \Delta_H)^*\Gamma^c_{\psi, \epsilon} = \{(t, \tau, (q,\eta -\pi_H \xi) ; (q', \eta' - \pi_H  \xi'))\} $$
unless $\eta -\pi_H \xi =  \eta' - \pi_H \xi' = 0$. Equivalently,   the pushforward relation only retains covectors  normal to the fiber, which results in `closing' the bi-angle-loop wave front set to the set of 
`closed bi-angles'. Hence,
the pushed forward Lagrangian is,
  $$\begin{array}{lll} \Pi_{t *}  (\Delta_H \times \Delta_H)^*\Gamma^c_{\psi, \epsilon}  & = &  \{(t, \tau); \;
  (t, \tau, (q,\eta -\pi_H \xi) ; (q', \eta' - \pi_H  \xi'))  \in  (\Delta_H \times \Delta_H)^*\Gamma^c_{\psi, \epsilon}: 
\\&&\\ & &\eta -\pi_H \xi =  \eta' - \pi_H  \xi' = 0 \}  \\ &&\\
  & = & \{(t, \tau): \gcal_c(t) \not= \emptyset\} =   \Lambda_{\psi}^c,\;\; ({\rm cf.} \; \eqref{gcalc})
 
 \end{array}$$

 The pushforward Lagrangian is the stage in the  sequence of compositions where closed geodesic bi-angles first occur.  $\zeta_M$ is constrained to make an angle of $\arccos \; c $ with $H$. At this
state, the cutoffs $\chi_{\epsilon}$ away from tangential and normal directions become unnecessary and $B_{\epsilon}$ may
be removed. Cleanliness of the diagonal pullback is equivalent to the cleanliness conditions on geodesic bi-angles of Definition \ref{CLEAN}, completing the proof of \eqref{PUSHPULL}.

The next step is to calculate the order \eqref{ordStpsi}  of $S^c(t, \psi)$ at its singularities.
As in \cite[Lemma 6.3]{DG75},  $\Pi_{t*} \Delta_H^*$ is  an operator of order $0$  from $\R \times H \times H \to \R$ with the Schwartz kernel of the identity operator on $\R \times H \times H \times \R$.    The excess of its composition with $\Gamma_{\psi,t}^c$  is by definition (and by \eqref{gcalc}),
\begin{equation} \label{etform}
	e(t) = \dim \{ (s, \xi) \in{\rm supp} \;\hat{\psi} \times  S^c_H M: G_H^{-s} \circ \pi_H G_M^{cs + t} (\xi) = \pi_H(\xi)\} = \dim \gcal_c(t).
\end{equation}

We refer to Section \ref{FIOSECT} for background  (see \eqref{EXCESSCOMP} and \eqref{EXCESS}). The calculation of the excess is parallel to 
to the calculation of the excess of $(\pi_* \Delta^*) U(t)$  in \cite{DG75}, where the excess is the dimension of the fixed point set of $G^t$ on $S^*M$.
 Another description of the excess is given   in \cite[Page 18]{HoIV}. It is the fiber dimension of the fiber of the intersection over the
composition (see Section \ref{FIOSECT} for the precise statement); that explains  why we may restrict to  $S^c_H M$ and not $T^c_H M$ and the fiber is $\dim \gcal^t_c$. By the order calculation at the end of the proof of Proposition \ref{LRLEMPROP},   the order of the singularity at each $t
\in \scal_c$ in the singular support is given by,
\begin{equation} \label{ORDST}
{\rm ord} \;S^c(t, \psi)  =
{\rm ord} e^{it P} \psi(Q_c)   + \half (n-d) + \frac{e(t)}{2} = -\frac{3}{4} + \half (n-d) +  \frac{\dim \gcal^t_c }{2}.
\end{equation}

In the notation  \eqref{gcalct},
$$e(0)  = \dim \gcal_c^{0}  =  
\dim \gcal_c^{(0,0)} = \dim S^c_H M = n + d - 2,   \; (0 < c < 1).
$$   
Indeed,  at each point $y \in H$ and for $c \leq 1$,  $$S^c_y M  = \{\xi \in S_y^* H, \xi = \eta + \nu,  \eta \in T_y^* H, \nu \in N^*_y H,
|\eta| = c,  |\nu| = \sqrt{1 - c^2}\}.$$ For $0 < c < 1$, $S_y^c M $  is a hypersurface in $S^*_y M$ and has dimension $n  - 2$. 

Combining with \eqref{ORDST}, it  follows that the order of \eqref{St} at $t=0$ equals,
\begin{equation} {\rm ord} \;S^c(t, \psi) |_{t =0} = 
- \frac{3}{4} + \half (n -  d)  + \half (n + d-2) 
= n - 1 - \frac{3}{4}    \;\;(0 < c < 1).
\end{equation}

For general $T \in \scal_c$, 
\begin{equation} \label{genord} {\rm ord} \;S^c(t, \psi) |_{t =T} = - \frac{3}{4} + \half (n -  d)  + \half \dim \gcal^T_c.\end{equation}


To complete the proof of 
Proposition \ref{MAINFIOPROP}, we need to calculate the symbol of \eqref{St} at each singularity, and particular to
show that the symbol at  $t=0$ is given by \eqref{sigmaStpsi}. Since the coefficient at $t=0$ was
calculated in great detail in the earlier paper \cite{WXZ20}, we only briefly sketch the argument.

 The symbol is calculated using iterated pushforward and pullback formulae as reviewed in \eqref{PFPB} (see also \eqref{PB} -  \eqref{PF}).   The pushforward is by the canonical projection $\Pi_t$,  and the pull-back is under the canonical embedding $\Delta_H \times \Delta_H$.
In the terminology of \cite{GS13}, these maps must  be `enhanced' with natural half-densities to make them `morphisms'
on half-densities. As reviewed in Section \ref{MORPHISM}, the embedding must be enhanced by a half-density on the conormal bundle
of the image, and the projection must be enhanced by a half-density along the fiber. As discussed in \cite[Page 66]{DG75}, there is a natural
half-density on the conormal bundle of the diagonal. In fact, the enhancement construction in this case is quite simple, since the 
pull-back under $\Delta_H \times \Delta_H$ of the half-density symbol on $\Gamma_{\psi,\epsilon}^c$ produces a density on the fibers
of the projection $\Pi_t$.  One integrates this density over the fibers to obtain the symbol in \eqref{sigmaStpsi} at any $t$.
The half-density symbols are always volume half-densities on their respective canonical relations. In particular
when $c \in (0,1)$ the coefficient of the singularity at $t=0$ is a dimensional constant times ${\rm Vol}(S^c_H M)$.
\end{proof}

\subsection{Proof of Theorem \ref{main 4},  Theorem \ref{main 5} and  Proposition \ref{MORESINGS} } \label{3PROOFS}


By \eqref{FTFORM},  $  {\mathcal F}_{\lambda \to t} d N^{c}_{\psi, H  } = S^c(t, \psi)$ and Proposition \ref{MAINFIOPROP}
shows that $S^c(t, \psi)$ is a discrete sum of translated polyhomogeneous distributions in $t$. It follows from Remark \ref{IREM} that if a distribution lies in $I^{\frac{\nu}{2} - \frac{1}{4}}(\Lambda_T)$ then its inverse Fourier
transform  is asymptotic to $ s^{\frac{\nu-1}{2}}$. 
This suggests that the singularity at $t =0$ produces an asymptotic expansion of order $ n-2$ equal to   $ n-2$ for $0 < c < 1$.


The next Lemma gives the precise statement and concludes  the proof of Theorem \ref{main 4} for test functions whose support contains only the singularity at $t=0$. If one uses general test functions with Fourier transforms in $C_0^{\infty}(\R)$, one
adds similar contributions from the  non-zero $t \in \scal_c$. As mentioned above, the order of these singularities 
is no larger than the order at $ t=0$.  If they are less than that order $t= 0$ is called dominant (Definition \ref{DOMDEF}).

\begin{lemma}\label{CONVOLUTION} Assume $0 < c < 1$.  Let $\rho \in \scal(\R)$ with $\hat{\rho} \in C_0^{\infty}$, $\int \rho = 1$, and with ${\rm supp} \hat{\rho}$ in a
sufficiently small interval around $0$.  Then,
there exists $\beta_j \in \R$ and a complete asymptotic expansion,
$$  N^{c} _{\rho, \psi, H  }(\lambda) \sim  \lambda^{n-2} \sum_{j=0}^{\infty} \beta_j \; \lambda^{-j}, $$
with $\beta_0 =  \; C_{n,d} \;    a_c^0(H, \psi)$ (see Theorem \ref{main 5} for the general formula).

\end{lemma}

\begin{proof} The asymptotic expansions follow immediately from Proposition \ref{MAINFIOPROP}.
By definition (see \eqref{St}),

\begin{equation} \label{EXPRESSIONS2} \begin{array}{lll} N^{c} _{\rho, \psi, H  }(\lambda) & = &   \sum_{j, k}  \rho(\lambda - \lambda_k) \psi( \mu_j - c \lambda_k)  \left| \int_{H} \phi_j \overline{\psi_k}dV_H \right|^2 \\ && \\ & =& \int_{\R} \hat{\rho}(t) e^{it \lambda}   S^c(t, \psi)  dt
. \end{array}  \end{equation}
If $\gcal_c^{0,0}$ is the only component with $t=0$, then by Proposition \ref{MAINFIOPROP} and by definition of Lagrangian distributions \eqref{I*}, 
for sufficient small $|t|$, 
\begin{equation} \label{SI*} S^c(t, \psi) = \sum_{j=0}^{\infty} \alpha_j  \int_0^{\infty} s^{n-2 -j} e^{- i s t} ds \; {\rm mod}\; C^{\infty}.  \end{equation}

{$$ \begin{array}{lll} \int_{\R} \hat{\rho}(t) e^{it \lambda}   S^c(t, \psi)  dt & = &
\sum_{j=0}^{\infty} \alpha_j  \int_0^{\infty} s^{n-2  -j} \left(  \int_{\R} \hat{\rho}(t) e^{it \lambda}e^{- i s t} dt \right) ds + O(\lambda^{-\infty})
\\ &&\\& = &
\sum_{j=0}^{\infty} \alpha_j  \int_0^{\infty} s^{n-2 -j} \rho(\lambda -s)  ds + O(\lambda^{-\infty}) \\ && \\
& = & \sum_{j=0}^{\infty} \alpha_j  \int^{\lambda}_{-\infty} (\lambda - s) ^{n-2 -j} \rho(s)  ds + O(\lambda^{-\infty})
\\ &&\\ & = & \sum_{j=0}^{\infty} \alpha_j \int^{\infty}_{-\infty} (\lambda - s) ^{n-2 -j} \rho(s)  ds + O(\lambda^{-\infty})
\\ &&\\ & 
\end{array}$$}where $\tilde \alpha_0 = \alpha_0$ is the principal symbol of $S^c(t, \psi)$ at $t=0$ (given in Proposition \ref{MAINFIOPROP}), and $\tilde \alpha_j$ are obtained in part by expanding
\[
	\int_{-\infty}^\infty (\lambda - s)^{n-2-j} \rho(s) \, ds
\]
in $\lambda$.

As discussed in Section \ref{supplarge}, when $\hat{\psi}$ is arbitrarily large, the principal symbol is changed by summing over the
components of $\gcal_c^0$.  As in that section, the $(c, s, 0)$-bi-angles with $t = 0$  is assumed to be a  union of clean components $Z_j(0)$ of dimension $d_j$. In our situation $Z_j$ is a component of  $ \gcal^0_c$. 
Then, for $t$ sufficiently close to $0$,  $$ S^c(t, \psi) = 
\sum_j \beta_j(t), $$ with 
\begin{equation}\label{betaeq2} 
\beta_j(t) = \int_{\R} \alpha_j(s) e^{- i s t} ds, \;\; \text{ with}\;\; \alpha_j(s) \sim (\frac{s}{2 \pi i})^{ -1 + \half (n -d)+\frac{d_j}{2}}\;\; i^{- \sigma_j} \sum_{k=0}^{\infty} \alpha_{j,k} s^{-k}, \end{equation}
where $d_j $ is the dimension of the component $Z_j(0) \subset \gcal_c^0$.
\end{proof}

\subsubsection{Proof of Proposition \ref{MORESINGS}}
\begin{proof} 

By Proposition \ref{LRLEMPROP}, at  a non-zero period, the exponents are calculated from \eqref{ORDST}, with the excess \eqref{etform} given by $\dim \gcal^t_c$.
Hence, the exponents at a non-zero period are now $\half (n-d)  - 1 +  \frac{\dim \gcal^t_c }{2}$.

In the notation of   \cite[Theorem 4.5]{DG75}, we are assuming that the set of $(c, s, t)$-bi-angles with $t = T \in {\rm singsupp}\; S^c(t, \psi) \backslash \{0\}$  is a union of clean components $Z_j$ of dimension $d_j$. In our situation $Z_j$ is a component of  $ \gcal^t_c$. 
Then, for $t$ sufficiently close to $T$,  $$S^c(t, \psi) = 
\sum_j \beta_j(t - T), $$ with 
$$
\beta_j(t) = \int_{\R} \alpha_j(s) e^{- i s t} ds, \;\; \text{ with}\;\; \alpha_j(s) \sim (\frac{s}{2 \pi i})^{ -1 + \half (n -d)+\frac{d_j}{2}}\;\; i^{- \sigma_j} \sum_{k=0}^{\infty} \alpha_{j,k} s^{-k},$$
where $d_j $ is the dimension of the component $Z_j$ of $\gcal_c^t$.
\end{proof}

\subsection{Sub-principal term of $N_{\rho, \psi}(\lambda)$ when $\hat{\psi}$ has small support and both $\hat{\psi}$
and $\hat{\rho} $ are even.}\label{SUBPRINCIPALSECT}

{The subprincipal term may be calculated by the stationary phase method, but even the subprincipal term is a sum of a large number of terms with
up to six derivatives on the phase and two derivatives of the amplitude. It is easily seen (for instance, on a flat torus) that the subprincipal term contains
sums with the derivatives of $\hat{\rho}(t)$ and $\hat{\psi}(s)$ at $t=s=0$). It is assumed in \cite[Proposition 2.1]{DG75} (which only involves $\hat{\rho}$) that
$\hat{\rho} \equiv 1$ near $0$, hence none of its derivatives contribute to the subprincipal term. We are making the same assumption. 

In addition, derivatives of the amplitude of the wave kernel may contribute to the subprincipal term. We use a parity argument as in \cite[Page 48]{DG75} to show that the
contribution of these terms is also zero. 
 The sub-principal symbol of $\sqrt{-\Delta_M}$ and of $\sqrt{-\Delta_H}$ both vanish, 
    hence the subprincipal symbols of $\sqrt{-\Delta_M} \otimes I$ and of $Q_c$ both vanish.   The homogeneous
part of degree k in  $\sigma_P(x, \xi)$ is even, resp. odd if k is even, resp. odd. By induction with respect to  $r$ it follows that $(\frac{\partial}{\partial t})^r
a_{-j}$ is an even, resp. odd. if $r-j$ is even, resp. odd. 
    
The amplitude of
    $e^{it P} e^{i s Q_c}$    is obtained by integrating the parametrix formula for $e^{i t P} \circ e^{- i s Q_c}$ as
    a tensor product $e^{i (t - c s) P_M} \otimes     e^{i s P_H} $. The parities of the terms in the amplitude are thus determined as in \cite{DG75},
for $s = t = 0$. One has to restrict the $M$ amplitudes to $H$ but they still have the same parity. One further has to restrict them to
    the diagonals in $H \times H$, which seems to multiply the amplitudes. But  the subprincipal term can only be obtained as the product of the
    principal symbol and the  subprincipal symbol. 
    Hence it is odd and its integrals over cospheres vanishes.
    
}

\begin{remark} We opt not to  calculate the Maslov indices $\sigma_j$ for the sake of brevity,  and absorb them into the constants $\beta_{\ell}$.
\end{remark}

\section{Proof of  Theorem \ref{main 2} and Theorem \ref{main 2b}}

In this section we apply a cosine Tauberian theorem to deduce Theorem \ref{main 2} from Theorem \ref{main 4}. 

\subsection{Tauberian Theorems} \label{SS Tauberian}

For the reader's convenience we quote the statements of two Fourier Tauberian theorem from \cite{SV}.
In what follows, $\rho$ will be{ a strictly positive even Schwartz-class function} on $\R$ with compact Fourier support satisfying that $\hat\rho(0)=1$. $N$ will be a tempered, monotone increasing function with $N(\lambda) = 0$ for $\lambda < 0$, and $N'$ its distributional derivative as a nonnegative measure on $\R$.

\begin{proposition}[Corollary B.2.2 in \cite{SV}] \label{tauberian 1} 
	Fix $\nu \geq 0$. If $N' * \rho(\lambda) = O(\lambda^\nu)$, then
	\[
		N(\lambda) = (N * \rho)(\lambda) + O(\lambda^\nu).
	\]
	This estimate holds uniformly for a set of such $N$ provided $N' * \rho(\lambda) = O(\lambda^\nu)$ holds uniformly.
\end{proposition}

\begin{proposition}[Theorem B.5.1 in \cite{SV}] \label{tauberian 2}
	Fix $\nu \geq 0$. If $N' * \rho(\lambda) = O(\lambda^\nu)$ and additionally
	\[
		N' * \chi(\lambda) = o(\lambda^\nu)
	\]
	for every Schwartz-class $\chi$ on $\R$ whose Fourier support is contained in a compact subset of $(0,\infty)$. Then,
	\[
		N(\lambda) = N * \rho(\lambda) + o(\lambda^\nu).
	\]
\end{proposition}

\subsection{Proof of Theorem \ref{main 2}}\label{TAUBPSISECT}
\begin{proof} 
Theorem \ref{main 2} pertains to the Weyl function $N^{c} _{\psi, H  }(\lambda)$ of  \eqref{cpsi},
$$N^{c} _{\psi, H  }(\lambda): = 
\sum_{j, k:  \lambda_k \leq \lambda}  \psi( \mu_j - c \lambda_k)  \left| \int_{H} \phi_j \overline{\psi_k}dV_H \right|^2. $$
{Recall our assumption that $\psi \geq 0$.} Then, $N^{c} _{\psi, H  }(\lambda)$ 
 is monotone non-decreasing and has Fourier transform $S^c(t, \psi)$ \eqref{St}.

We apply Proposition \ref{tauberian 1} with $\hat{\rho}\; \cap\; {\rm singsupp}\; S^c(t, \psi) = \{0\}$ and  to $d N^{a,c} _{\psi, H  }(\lambda) $. By  Lemma \ref{CONVOLUTION},  $\rho* d N^{c} _{\psi, H  }(\lambda)  = \beta_0 \;\lambda^{n-2} + O(\lambda^{n-2-1})$, and therefore,
$$\begin{array}{lll}  N^{c} _{\psi, H  }(\lambda) & = &  \rho* N^{c} _{\psi, H  }(\lambda) + O(\lambda^{n-2}) \\ &&\\
& = & \dfrac{\beta_0}{n-1} \;\lambda^{n- 1} + O(\lambda^{n-2}),\end{array}  $$
concluding the proof of Theorem \ref{main 2}.
\end{proof}



\subsection{Proof of Corollary \ref{JUMPCOR} } 
To prove  Corollary \ref{JUMPCOR} it suffices to prove that, for any $\epsilon > 0$ there exists a test
function $\psi \geq 0, \hat{\psi} \in C_0^{\infty}(\R), \hat{\psi}(0) =1$ and a universal constant $C(\epsilon, \delta)$
depending only on $(\epsilon, \delta)$  so that for all $\lambda_j$,
\begin{equation} \label{LB} J_{\psi, H}^c(\lambda_j) \geq C(\epsilon, \delta) \; J_{\epsilon, H}^{c} (\lambda_j). \end{equation}
Then the upper bound for $ J_{\psi, H}^c(\lambda_j) $ given  in Corollary \ref{main 2cor} provides the
upper bound for $ J_{\epsilon, H}^{c} (\lambda_j)$.

\begin{proof}
We have
\[
	\sum_{ k: |\mu_k - c \lambda_j | \leq \epsilon  }    \left| \int_{H} \phi_j \overline{\psi_k}dV_H \right|^2 \leq \sum_k  \psi(\mu_k - c\lambda_j)\left| \int_{H} \phi_j \overline{\psi_k}dV_H \right|^2 
\]
provided $\psi$ is chosen to be a nonnegative Schwartz function with small Fourier support, with $\psi(0) > 1$, and perhaps scaled wider so that $\psi \geq \mathbf{1}_{[-\epsilon,\epsilon]}$. By Corollary \ref{main 2cor}, the right side is $O(\lambda_j^{n-2})$, concluding the proof.
\end{proof}

\subsection{Proof of Theorem \ref{main 2b}}\label{main 2bSECT}

\begin{proof}

By Proposition \ref{tauberian 2}, it suffices to check the additional condition,
\begin{equation} \label{ADDCOND} 
	 \chi * d N_{\psi, H}^c (\lambda) = o(\lambda^{n-2} )
\end{equation}
for every Schwartz-class $\chi$ on $\R$ whose Fourier support is contained in a compact subset of $(0,\infty)$.

To prove this, we consider the expansions given in Theorem \ref{main 4} and Theorem \ref{main 5} and  especially Proposition \ref{MORESINGS}. In addition to the assumptions of Theorem \ref{main 4}, the assumption of  Theorem \ref{main 2b} is that $d_j(T) < d_j(0)$ for $T \not=0$, i.e. that $\dim Z_j(0) > \dim Z_j(T)$ for all $T \neq 0$. This assumption together with Proposition \ref{MORESINGS} shows that \eqref{ADDCOND} holds; then Proposition \ref{tauberian 2} implies the first claim in Theorem \ref{main 2b}.

Finally, we bound the jumps in the Weyl function of Theorem \ref{main 2b} by the remainder as before and obtain $J_{\psi, H}^c(\lambda) = o_\psi(\lambda^{n-2})$.
\end{proof}


\section{Proof of Theorem \ref{main 3}}\label{BUSECT}

In this section we deduce Theorem \ref{main 3} from Theorem \ref{main 5} and an additional Tauberian theorem, which allows
us to replace $\psi$ in the inner sum of \eqref{cpsirho} (or \eqref{EXPRESSIONS2}),
 by an indicator function ${\bf 1}_{[-\epsilon, \epsilon]}$. Throughout this section, we assume that $c < 1$.
 
  For simplicity of notation,  when $\psi = {\bf 1}_{[-\epsilon, \epsilon]}$, we write,
  \begin{equation} \label{cpsirho2} 
 N^{c} _{\epsilon, H  }(\lambda) := \sum_{j: \lambda_j \leq \lambda} J^c_{{\bf 1}_{[-\epsilon, \epsilon]}}(\lambda_j)
 \end{equation}
where as in \eqref{JDEF}
\begin{equation} \label{SUM} 
	J^c_{ {\bf 1}_{[-\epsilon, \epsilon]}}(\lambda_j) : =  \sum_{\ell: \lambda_{\ell} = \lambda_j} 
\sum_{  \substack{k: | \mu_k - c \lambda_j|  \leq \epsilon}}   \left| \int_{H} \phi_{\ell} \overline{\psi_k}dV_H \right|^2
\end{equation}
 The Tauberian theorem is not used in the  traditional way, i.e. to replace a monotone increasing function with jumps by a smoothly varying
 sum. The
relevant monotone function is the  Weyl-Kuznecov sum \eqref{cpsirho2}, and as with \eqref{cpsi}, its jump discontinuities occur only at the points $\lambda = \lambda_j$, with jumps $J^c_{ {\bf 1}_{[-\epsilon, \epsilon]}}(\lambda_j)$. 


We treat the sum \eqref{SUM}    as a semi-classical Weyl
 function with semi-classical parameter $\lambda_j^{-1}$ and deploy
 the semi-classical  Tauberian theorem of \cite{PR85, R87}.     The main complication is that the terms 
are weighted by the unbounded and  non-uniform weights $\left| \int_{H} \phi_j \overline{\psi_k}dV_H \right|^2$ (in $(\lambda_j, \mu_k))$. Moreover, 
$J^c_{ {\bf 1}_{[-\epsilon, \epsilon ]}}(\lambda_j)$  does not usually
 have an asymptotic expansion as $\lambda_j \to \infty$ due to lack of asymptotics for the individual eigenfunctions $\phi_j$. We need to sum
 in $\lambda_j$ as well to obtain asymptotics. 
 
 To set things up for the Tauberian arguments in \cite{PR85,R87}, we define
\begin{equation} \label{EMPIRICAL}
\left\{
\begin{array}{l} \displaystyle d\mu^c_{\lambda}(x) :=\sum_{j: \lambda_j \leq \lambda}  \sum_{k}   \left| \int_{H} \phi_j \overline{\psi_k}dV_H \right|^2 \delta_{\mu_k - c\lambda_j}(x). \\
\displaystyle \sigma^c_{\lambda} (x) = \int_{-\infty}^x d\mu^c_{\lambda}(y).
\end{array}
\right.
\end{equation}
By  \eqref{c}, 
$$
N^{c}_{\epsilon, H}(\lambda) :=  \int_{ -\epsilon}^{ \epsilon} d\mu^{c}_{\lambda} =  \sigma^c_{\lambda} (\epsilon) -
\sigma^c_{\lambda} (-\epsilon) =  \sum_{j,\lambda_j \leq \lambda} J_{{\bf 1}_{[-\epsilon,\epsilon]}}^c(\lambda_j)
. $$
 Note that $\supp  \mu^c_{\lambda} \subset \{x \geq - c \lambda\}$ and that,
\begin{equation} \label{MASS}
 	\int_{-\infty}^{\infty} d \mu_{\lambda}^c  = \sum_{j: \lambda_j \leq \lambda}  \|\gamma_H \phi_j\|_{L^2(H)}^2.
\end{equation}

In comparison with the proofs of Theorems \ref{main 2} - \ref{main  2b}, we  do not have a complete, or even two-term,  asymptotic expansion for either of the once-smoothed sums, \eqref{cpsi} or 
 $\sum_j \rho(\lambda - \lambda_j)  J^c_{{\bf 1}_{[-\epsilon, \epsilon]}}(\lambda_j). $

The  Tauberian  strategy is to smooth out the indicator functions ${\bf 1}_{[-\epsilon, \epsilon]}$ and apply Theorem \ref{main 5}  to this
 kind of mollified sum.  Our aim is to show that, under the assumption that $\gcal^{0,0}_c$ is dominant, we can sharpen the sum at the expense of weakening the remainder to $o_{c,\epsilon}(\lambda^{n-1})$.
 The smoothing error is an `edge effect'  due to the sum of the  terms $ \left| \int_{H} \phi_j \overline{\psi_k}dV_H \right|^2$ 
 near the endpoints $|\mu_k - c \lambda_j |= \epsilon $ of the interval $|\mu_k - c \lambda_j |\leq  \epsilon $. The Tauberian theorem is used to show that the
 eigenvalues and the quantities $|\mu_k - c \lambda_j |= \epsilon $ are sufficiently uniform, i.e. do not concentrate 
 near the endpoints.










\subsection{The proof of Theorem \ref{main 3}}

Following \cite{PR85, R87}, we denote by
 $\rho_1 \in C_0^{\infty}(-1,1) $ a smooth cutoff satisfying $\rho_1(0) = 1$, $\rho_1(-t) = \rho_1(t)$. With no loss of
 generality, we 
assume $\hat \rho_1(\tau) \geq 0$ and $\hat \rho_1(\tau) \geq \delta_0 > 0$ for $|\tau| \leq \epsilon_0$.  
Then set,
\begin{equation} \label{thetaDEF} \rho_T(\tau) = \rho_1(\frac{\tau}{T}), \;\;\; \theta_{T}(x) := \hat \rho_T (x)  = T \hat  \rho_1(T x). \end{equation}
In particular, $\int \theta_T(x) dx = 1$ and $\theta_T(x) > T \delta_0$ for $|x| < \epsilon_0/T$. 
Note that $\theta_{T}* d\mu^c_{\lambda}$ is by definition the measure,
$$
	\theta_{T}* d\mu^c_{\lambda}(x) =  \sum_{j: \lambda_j  \leq \lambda}  \sum_k \theta_T(\mu_k - c \lambda_j - x)  \left| \int_{H} \phi_j \overline{\psi_k}dV_H \right|^2.
$$
Of course, $\theta_{T}* d\mu^c_{\lambda} (x)  \to d\mu^c_{\lambda} (x) $ as $T \to \infty$.

Let us record the relation between the various relevant quantities.

\begin{lemma} \label{RELS} We have,
\[
	\int_{-\epsilon}^{\epsilon} \theta_{T}* d\mu^c_{\lambda} (x) =  N^{c} _{\theta_T * {\bf 1}_{[-\epsilon, \epsilon]}, H  }(\lambda) =  \sigma_{\lambda}^c * \theta_{T} (\epsilon) - \sigma_{\lambda}^c  * \theta_{T} (-\epsilon)
\]

\end{lemma}

\begin{proof}
	This follows from the definitions.
\end{proof}

%

The asymptotics of $ N^{c} _{\theta_T * {\bf 1}_{[-\epsilon, \epsilon]}, H  }(\lambda) $  are given in Theorem  \ref{main 5} with 
$$
	\psi = \psi_{T, \epsilon} : = \theta_T * {\bf 1}_{[-\epsilon, \epsilon]}.
$$
We use the notation $\psi_{T, \epsilon}$ henceforth to simplify the notation. Putting things together,
\begin{equation} \label{CORCOR}
	N^{c} _{\epsilon, H}(\lambda) =  N^{c} _{\psi_{T, \epsilon}, H  }(\lambda) +  N^{c}_{(\psi_{\infty, \epsilon} - \psi_{T, \epsilon}), H  }(\lambda),
\end{equation}
where $\psi_{\infty} = \mathbf{1}_{[-\epsilon,\epsilon]}$.
Here, $\epsilon $ is fixed.
The hard step is to estimate the error in the smoothing approximation,
\begin{equation}\label{NEEDT}\begin{array}{l} N^{c}_{(\psi_{\infty, \epsilon} - \psi_{T, \epsilon}), H  }(\lambda) =\left(\sigma_{\lambda}^c (\epsilon) - \sigma_{\lambda}^c (- \epsilon) \right)-  \left(\sigma_{\lambda}^c *  \theta_{T} (\epsilon)- \sigma_{\lambda}^c  * \theta_{T}  (- \epsilon) \right),
 \end{array} \end{equation}  in terms of $(\lambda, T)$.

  

\begin{proposition}\label{TLEMa} With the same notation and assumptions as in Theorem \ref{main 3},  for any $\epsilon > 0$ and $c \in (0,1)$, there exist constants $\gamma(c, \epsilon) $ such that, for any $T >0$,
	$$
		|N^{c}_{(\psi_{\infty, \epsilon} - \psi_{T, \epsilon}), H  }(\lambda)| = \left| \int_{-\epsilon}^{\epsilon} (\theta_T * d\mu_{\lambda}^c- d\mu_{\lambda}^c) \right|  \leq 
	\frac{\gamma(c, \epsilon)}{T}  \lambda^{n-1} + O_{T, \epsilon} (\lambda^{n-3/2})
	.$$ 
\end{proposition}

Before giving the proof, we verify that Proposition \ref{TLEMa} implies Theorem \ref{main 3}.
By Theorem \ref{main 5} and the hypothesis that $\gcal_c^{0,0}$ is dominant, we have, 
$$
N^{c} _{\psi_{T, \epsilon}, H}(\lambda) = \hat{\psi}_{T, \epsilon}(0) \;  A_{n, d}^c  \hcal^{d}(H) \lambda^{n-1 } +   R_{\psi_{T, \epsilon}} (\lambda),
$$
where $A_{n,d}^c $ is the leading coefficient, e.g. $A_{n,d}^c = C_{n, d} c^{d-1} (1 - c^2)^{\frac{n-d-2}{2}} $ for $0 < c < 1$, and where $R_{\psi_{T, \epsilon}}(\lambda) =   O_{T, \epsilon}(\lambda^{n-3/2}) $. Moreover, 
$\hat{\psi}_{T, \epsilon} (0) = 2 \epsilon.$ The full error term in \eqref{CORCOR} is therefore, 
\begin{align*} 
\wt{R}_{T, \epsilon}(\lambda): & =    N^{c}_{(\psi_{\infty, \epsilon} - \psi_{T, \epsilon}), H  }(\lambda) +  R_{\psi_{T, \epsilon}}(\lambda) \\
& = O( \frac{\gamma(c, \epsilon)}{T}  \lambda^{n-1}) + O_{T, \epsilon} (\lambda^{n-3/2 }) +  O_{T, \epsilon}(\lambda^{n-3/2}).
\end{align*}
The bound $\tilde R_{T,\epsilon}(\lambda) = o_\epsilon(\lambda^{n-1})$ follows by taking $T = T(\lambda)$ as a function of $\lambda$ increasing $T(\lambda) \nearrow \infty$ sufficiently slowly.

\subsection{Proof of Proposition \ref{TLEMa}}

We have,
$$
\begin{array}{lll}
\int_{-\epsilon }^{\epsilon } (\theta_T * d\mu_{\lambda}^{c} - d\mu_{\lambda}^{c} )& = & 
\int_{\R} \left(\mu_{\lambda}^c ( [-\epsilon , \epsilon ] - \tau)  - \mu^c_{\lambda} [- \epsilon , \epsilon ]\right) \theta_{T}(\tau) d\tau\\&&\\
& = &T \int_{\R}  \left( \mu^c_{\lambda} ([- \epsilon ,\epsilon ] - \tau) - \mu_{\lambda}^c[-\epsilon , \epsilon ]) \right)  \hat{\rho}_1( \tau T) d \tau \\&&\\
& = &T  \int_{|\tau| \leq \frac{1}{T} }    \left( \mu^c_{\lambda} ([- \epsilon ,\epsilon ] - \tau)  - \mu^c_{\lambda} ([- \epsilon, \epsilon ]) \right) \hat{\rho}_1( \tau T) d \tau \\&&\\
&  + & T \int_{|\tau| > \frac{1}{T}}    \left( \mu^c_{\lambda} ([-\epsilon , \epsilon ] - \tau) - \mu^c_{\lambda} ([-\epsilon ,\epsilon ]) \right)   \hat{\rho}_1( \tau T) d \tau \\&&\\& =: & I_1  + I_2.
\end{array}
$$

The key point is to prove the analogue of \cite[Proposition 3.2]{PR85}.
\begin{proposition}\label{TLEM2}  With the same notation and assumptions as in Theorem \ref{main 3},  and for any  $0 < c < 1$ here  exist constants $\gamma_1(c, \epsilon)$ such that, for any $T > 0$,  
	$$\left|\mu^c_{\lambda} ([- \epsilon , \epsilon ] - \tau)   -\mu^c_{\lambda} ([- \epsilon, \epsilon ] )   \right| \leq  \gamma_1(c, \epsilon)   (\frac{1}{T} + |\tau|)  \lambda^{n - 1 } + C_1(T,c) O(\lambda^{n-3/2}), $$

\end{proposition}

We first  show
that Proposition \ref{TLEM2} implies Proposition \ref{TLEMa}.

\begin{proof} We only verify this in the case $0 < c < 1$, since the second case is proved in the same way. First,  observe that Proposition \ref{TLEM2} implies,
	\begin{equation} \label{Izbd} I_1  \leq \sup_{|\tau| \leq \frac{1}{T}} 
	\left|   \mu^c_{\lambda} ([- \epsilon, \epsilon ] - \tau)   - \mu^c_{\lambda} ([-\epsilon, \epsilon] )  \right|, \end{equation}
	and  Proposition \ref{TLEM2} immediately implies the desired bound  for $|\tau| \leq \frac{1}{T}$.   For $I_2$ one uses that   $\hat{\rho}_1 \in \scal(\R)$.
	Since
	$T \int_{|\tau| \geq \frac{1}{T} } \hat{\rho}_1(\tau T) d \tau \leq 1, $ Proposition \ref{TLEM2} implies that there exist constants $A > 0$, $C_1(T,c)$  so
	that 

	\begin{equation} \label{IIzbd} \begin{array}{lll}  I_2   &\leq & \; A \;
	\lambda^{n-1}\gamma_1(c,\epsilon)   T \int_{|\tau| > \frac{1}{T}} (\frac{1}{T} + |\tau|) 
	\hat{\rho}_1( T \tau) d \tau \\&&\\ &&+  C_1(T,c) O(\lambda^{n-3/2 }) T \int_{|\tau| > \frac{1}{T}}    \hat{\rho}_1(T \tau ) d \tau.  \end{array}\end{equation}
	If one changes variables to $r = T \tau$ one also gets the estimate of the Tauberian Lemma. 
	\end{proof}

We now prove Proposition \ref{TLEM2}. 

\begin{proof} 	Since we are studying the increments $ \mu^c_{\lambda} ([-\epsilon , \epsilon ] - \tau) - \mu^c_{\lambda} ([-\epsilon ,\epsilon ])$ and since the integral is a sum where $|\tau| \leq \frac{1}{T}$ and $|\tau | \geq \frac{1}{T}$, 
the proof is broken up into 3 cases: (1) $|\tau | \leq \frac{\epsilon_0}{T}$,
	\;  (2) $\tau = \frac{\ell}{T} \epsilon_0,$ for some $\ell \in \Z$, and (3)  $ \frac{\ell}{T} \epsilon_0 \leq \tau \leq  \frac{\ell+1}{T} \epsilon_0 $, for some $ \ell \in \Z$.
	
	The key assumption that the only maximal component is the principal component is used to obtain the factor of $\frac{1}{T}$, which is responsible
	for the small oh of the remainder. 
	We use the exact formula for the leading coefficient when $\hat{\psi}$ has arbitrarily large compact support
	in Theorem \ref{main 5}. When $0 < c < 1$ and when  the only maximal component is the principal component, the sum over $s_j^m$ is merely
	the value of $s =0$. When $\hat{\psi}(0) = 1$, $a_c^0(H, \psi)$ is independent of $\supp \; \hat{\psi}$. When there do exist many maximal 
	components, as in the case of subspheres of spheres, the sum $\sum_j \hat{\psi}(s_j^m)$ essentially counts the number of the components
	with $s$-parameter in $\supp \psi$, and that can cancel the $\frac{1}{T}$.


		
		


\noindent{\bf (1)}  Assume $|\tau| \leq \frac{\epsilon_0}{T}$. Also assume $\tau > 0$ since the
		case $\tau < 0$ is similar. We claim that,
		$$  \left| \mu_{\lambda}^c ([-\epsilon, \epsilon] -\tau)  - \mu_{\lambda}^c [- \epsilon, \epsilon])  \right| \leq \frac{2\gamma_0(c,\epsilon)}{T \delta_0} \lambda^{n-1}. $$
		
		 Write 
		$$ \begin{array}{lll} \mu_{\lambda}^c([- \epsilon, \epsilon] -\tau)  - \mu_{\lambda}^c[- \epsilon, \epsilon])  & = & \int_{\R} [{\bf 1}_{[- \epsilon - \tau, \epsilon- \tau]} - {\bf 1}_{[- \epsilon, \epsilon]}] (x)d \mu_{\lambda}^c(x)
		. \end{array} $$
		For $T$ sufficiently large so that $ \tau \ll 2 \epsilon$, $$ [{\bf 1}_{[- \epsilon - \tau,\epsilon - \tau]} - {\bf 1}_{[-\epsilon, \epsilon]}] (x)= {\bf 1}_{[- \epsilon - \tau, -\epsilon]} - {\bf 1}_{[\epsilon - \tau, \epsilon]}. $$

		Since they are similar we only consider the $[- \epsilon - \tau, - \epsilon]$ interval. Since for $|\tau| < \epsilon_0/T$, we have $\theta_T(\tau) > T \delta_0$, 
		it follows from Theorem \ref{main 5} that,
$$\begin{array}{lll} 
\mu_{\lambda}^c ([- \epsilon - \tau, - \epsilon]) &\leq& \frac{1}{T \delta_0} \int_\R \theta_T(-\epsilon-x) d \mu_{\lambda}^c (x) \\
&\leq& \frac{\gamma_0(c,\epsilon)}{T \delta_0} \lambda^{n-1} + O_{T,\epsilon}(\lambda^{n-3/2}).
\end{array}$$
The estimate in the third line uses the formula for the leading coefficient of Theorem \ref{main 5} with $\psi(s) = \theta_T(-\epsilon + s)$. Under the hypotheses of the Proposition, the 
sum of $\hat{\psi}(s_j^m)$ is just equal to $\hat{\psi}(0)$ and is independent of $T$, as explained above. This completes the proof of the claim.
		\bigskip

	\noindent{\bf (2)}  Assume $\tau = \ell \frac{\epsilon_0}{T} , \ell \in \Z.$ With no loss of generality, we  may assume  
		$\ell \geq 1.$ Write
		$$ \mu_{\lambda}^c([- \epsilon, \epsilon])  - \mu_{\lambda}^c ([- \epsilon, \epsilon]- \frac{\ell}{T} \epsilon_0)  
		= \sum_{j = 1}^{\ell} \mu_{\lambda}^c ([- \epsilon, \epsilon] -  \frac{j-1}{T} \epsilon_0 )- \mu_{\lambda}^c([-\epsilon, \epsilon] - \frac{j}{T} \epsilon_0 )   $$
		and apply the estimate of (1) to upper bound the sum by 
		$$\frac{2 \ell \gamma_0(c,\epsilon)}{T \delta_0} \lambda^{n-1} + O_{T,\epsilon}(\lambda^{n-3/2}) = \frac{2 \gamma_0(c,\epsilon)}{\epsilon_0 \delta_0} \tau \lambda^{n-1} + O_{T,\epsilon}(\lambda^{n-3/2}).$$
		
		\bigskip 
		
	\noindent{\bf (3)} Assume $ \frac{\ell}{T} \epsilon_0 \leq \tau \leq  \frac{\ell+1}{T} \epsilon_0$ with $\ell \in \Z$. Write
		$$\begin{array}{lll} \mu_{\lambda}^c ([-\epsilon, \epsilon] + \tau ) -  \mu_{\lambda}^c([-\epsilon, \epsilon]) & = &   \mu_{\lambda}^c ([- \epsilon, \epsilon] + \tau ) -   \mu_{\lambda}^c([-\epsilon,\epsilon] + \frac{\ell}{T} \epsilon_0 ) \\&&\\&&+ \mu_\lambda^c([-\epsilon,\epsilon] + \frac{\ell}{T} \epsilon_0 )
		- \mu_\lambda^c([\epsilon,\epsilon]).\end{array} $$
		Applying (1) and (2), it follows that
		$$|\mu_{\lambda}^c([- \epsilon, \epsilon] + \tau ) -  \mu_{\lambda}^c([- \epsilon,\epsilon])| \leq \frac{2 \gamma_0(c,\epsilon)}{\delta_0} \left(\frac{\tau}{\epsilon_0} + \frac{1}{T} \right)  \lambda^{n-1} + O_{T,\epsilon}(\lambda^{n - 3/2}). $$ 
		\bigskip
	
		This completes the proof of Proposition \ref{TLEM2}, hence also Proposition  \ref{TLEMa} and therefore
		 Theorem \ref{main 3}.
\end{proof}

\section{Appendix}
 
 \subsection{Background on Fourier integral operators and their symbols}\label{FIOSECT}
 The advantage of expressing $\Upsilon_{\nu, \psi}^{(1)}(t)$ in terms of pullback and pushforward is
 the symbol calculus of Lagrangian distributions is more elementary to describe for such compositions. 
 We refer to \cite{HoIV,GS77,D73} for background but quickly review the basic definitions.

 The space of
 Fourier integral operators of order $\mu$ associated to a canonical relation $C$ is denoted by $$K_A \in I^{\mu}(M \times M, C'). $$
 If $A_1 \in I^{\mu_1}(X \times Y, C_1'), A_2 \in   I^{\mu_2}(Y \times Z, C_2')$, and if $C_1 \circ C_2$ is a `clean' composition,
 then by \cite[Theorem 25.2.3]{HoIV},
 \begin{equation} \label{EXCESSCOMP} A_1 \circ A_2 \in I^{\mu_1 + \mu_2 + e/2} (X \times Z, C'), \;\; C = C_1 \circ C_2, \end{equation}
 where $e$ is the `excess' of the composition, i.e. if $\gamma \in C$, then $e =\dim C_{\gamma}$, the dimension of the fiber of $C_1 \times C_2
 \cap T^* X \times \Delta_{T^*Y} \times T^*Z$ over $\gamma$ (see \eqref{EXCESS} below).
 
 Pullback and pushforward of half-densities on Lagrangian submanifolds are more difficult to describe. 
 They depend on the map $f$ being a {\it morphism} in the language of \cite[page 349]{GS77}. Namely,
 if $f: X \to Y$ is a smooth map, we say it is a morphism on half-densities if it is augmented by a section
 $r(x) \in \mathrm{Hom}(|\Lambda|^{\half}( TY_{f(x)}, |\Lambda|^{\half} T_x X)$, that is, a linear transformation 
 mapping densities on $TY_{f(x)}$ to densities on $T_x X$. As pointed out in \cite[page 349]{GS77}, 
 such a map is equivalent to augmenting $f$ with a special kind  of half-density on the co-normal bundle $N^*(\mathrm{graph}(f))$ to the graph of $f$, which is constant along the fibers of the co-normal bundle. In our application, the maps are all restriction maps or pushforwards under canonical maps, and they are morphisms
 in quite obvious ways.  Note that under pullback by an immersion,  or  under a restriction, the number $n$ of independent variables is decreased  by the codimension $k$ and therefore
 the order goes up by $\frac{k}{4}.$ Pullbacks under submersions increase the number $n$.  Pushforward is adjoint to pullback and therefore also decreases the order by the same amount.

 Assume that $F: M \to N$ is a smooth map between manifolds.
 Let  $\Lambda \subset \dot  T^* N$ be a Lagrangian submanifold. Then its pullback is defined by, 
 \begin{equation} \label{PB} f^* \Lambda = \{(m, \xi) \in T^*M \mid \exists (n, \eta) \in \Lambda, f(m) = n, f^* \eta = \xi\}. \end{equation}
 On the other hand, let $\Lambda \subset  \dot  T^*M$. Then its pushforward is defined by,
 \begin{equation} \label{PF} f_* \Lambda = \{(y, \eta) \in T^* N \mid y = f(x), (x, f^* \eta) \in \Lambda\}. \end{equation}
 
 The principal symbol of a Fourier integral operator associated to a canonical relation $C$ is a half-density times a section of the  Maslov line bundle  on $C$.
 We refer to \cite[Section 25.2]{HoIV} and to \cite[Definition 4.1.1]{D73} for the definition; see also \cite{DG75, GU89} for further expositions and for
 several calculations of principal symbols closely related to those of this article.

 The \emph{order} of a homogeneous  Fourier integral operator\index{Fourier integral operator!order} $A \colon L^2(X) \to L^2(Y)$ in the
 non-degenerate case  is given in terms of a local oscillatory
 integral formula $$K_A(x,y) = \frac{1}{(2 \pi)^{n/4 + N/2}}\int_{\R^N} e^{\rmi \phi(x, y, \theta)} a(x, y, \theta) \dd \theta $$by \begin{equation} \label{ORDDEF} \ord A = m +
 \frac{N}{2} - \frac{n}{4}, \;\; \mathrm{where} \; n = \dim X + \dim Y, \; m = \ord \;a \end{equation} where the order  of the amplitude $a(x, y, \theta)$
 is the degree of the top order term of the polyhomogeneous expansion of $a$ in $\theta$, and $N$ is the number of phase
 variables $\theta$ in the local Fourier integral representation (see
 \cite[Proposition~25.1.5]{HoIV}); in the general clean case with
 excess $e$, the order goes up by $\frac{e}{2}$ (see \cite[Proposition~25.1.5']{HoIV}
 ). The order is designed to be independent of the specific representation of $K_A$ as an oscillatory integral. 
 
 Further, 
 the principal symbol of a Fourier
 integral distribution
 $$ I(x, y) = \int_{\R^N} e^{i \phi (x, y, \theta) }
 a( x, y, \theta) d \theta $$
 with non-degenerate homogeneous phase function $\phi$ and amplitude $a \in S^{0}_{cl}(M \times M \times \R^N),$
 is the transport to the Lagrangian $\Lambda_{\phi} =
 \iota_{\phi}(C_{\phi})$ of $a(\lambda) \sqrt{d_{C_\phi}}$ where $\sqrt{d_{C_\phi}}$ is the half density given by the square root of
 \begin{equation} \label{SYMBOLDEF}  d_{C_{\phi}}: =  \left|\frac{\partial(\lambda,
 	\phi_{\theta}')}{\partial(x,y,\theta)}\right|^{-1} |d \lambda| \end{equation}  on $C_{\phi}$, where $\lambda=(\lambda_1,...,\lambda_n)$ are local coordinates on the critical manifold $C_{\phi} =\{ (x,y,\theta); d_{\theta}\phi(x,y,\theta) = 0\}. $

We next review the definition of the excess in a fiber product diagram. 
Let $F = \{(x, y) \in X \times Y, f(x) = g(y)\}$, 

\begin{equation} \label{FP} 
\begin{tikzcd}
	X \arrow[d, "f"'] & F \arrow[l] \arrow[d] \\
	Z                 & Y \arrow[l, "g"']    
\end{tikzcd}
\end{equation} 
 The maps $f: X \to Z$ and $g: Y \to Z$ are said to intersect cleanly if the fiber product $F$ is a submanifold of $X \times Y$ and if 
 the tangent diagram is a fiber product diagram. 
 The excess is \begin{equation} \label{EXCESS} e = \dim F + \dim Z - (\dim X + \dim Y). \end{equation} Then $e =0$ if and only if the diagram is transversal. 
 Above $d = \dim {\rm Fix}(G^T)$ is the excess of the diagram.
\subsection{Enhancement, morphisms and pullbacks and pushforward of symbols} \label{MORPHISM}

The behavior of symbols under pushforwards and pullbacks of
Lagrangian submanifolds is described in \cite{GS77}, Chapter IV. 5
(page 345). The main statement (Theorem 5.1, loc. cit.) states
that the symbol map $\sigma: I^m(X, \Lambda) \to S^m(\Lambda)$ has
the following pullback-pushforward  properties under maps $f: X
\to Y$ satisfying  appropriate transversality conditions,
\begin{equation} \label{PFPB} \left\{\begin{array}{l} \sigma (f^* \nu) =
f^* \sigma(\nu), \\ \\
\sigma(f_* \mu) = f_* \sigma(\mu), \end{array} \right.
\end{equation} To be precise, $f$ must be ``enhanced'' as defined in \cite[Chapter 7]{GS13} in order to define a pullback or
pushforward on symbols. This is because the pullback/pushforward  of a half-density on $\Lambda$   is often not  a half-density
on $f^* \Lambda$.

The enhancement of a smooth map $f: X \to Y$ is a map $(f, r)$
with $r: |f^* T Y|^{\half} \to |TX|^{\half}$. Thus if $\rho$ is a
half-density on $Y$
$$(f, r)^* \rho = r(\rho(f(x)) \in |T_x X|^{\half}. $$

If $\iota: X \to Y$ is an immersion, then $N^*_{\iota} X$ consists
of covectors $\xi \in T^*_{\iota(x)} Y: d\iota_x^* \xi = 0$.
Enhancing an immersion is giving a section of $|N^*_{\iota}
X|^{\half}$.

If $\pi: Z \to X$ is a submersion and $V_z $ is the tangent space
to $\pi^{-1}(x)$. Then enhancing the fibration is giving a section
of $|V_z|^{\half}$.

In \cite[ p. 349]{GS13}, the authors explain that enhancement with $r$
is to define a half-density on $N^*(\Gamma_f)$ which is
constant along the fibers of $N^*(\Gamma_f) \to \Gamma_f$.  As a
result, a morphism $f: X \to Y$ induces a pushforward $$f_*:
\Omega^{\half}(\Lambda_X) \to \Omega^{\half}(f_* \Lambda_X). $$ It
also induces a pullback operation
$$f^* : \Omega^{\half}(\Lambda_Y) \to \Omega^{\half} f^*
(\Lambda_Y). $$

 Under appropriate clean or transversal assumptions,
if $f: X \to Y$ is a morphism of half-densities, then $f_*$ and
$f^*$ are morphisms of half-densities on Lagrangian submanifolds.

 \begin{remark} If $f: X \to Y$ is a submersion then $f^*$ is
 injective. Indeed if $f^* \eta = 0$ then $\eta \perp f_* TX = TY$.  If $f$ is an immersion, then $f_*$ is injective.
  \end{remark}

\end{document}